\documentclass[11pt, reqno]{amsart}

\usepackage[margin=2.71cm]{geometry}
\usepackage{lipsum}
\usepackage[T1]{fontenc}
\usepackage[T2A]{fontenc}
\usepackage[cp1251]{inputenc}
\usepackage[english]{babel}

\usepackage{longtable}
\usepackage{xcolor}
\usepackage{xfrac}
\usepackage{amsfonts,amssymb,mathrsfs,amscd,graphics,tikz,graphicx, verbatim}
\usepackage{tikz-cd}
\usepackage[all]{xy}
\usetikzlibrary{shapes.geometric}
\usepackage{hyperref}
\hypersetup{
    colorlinks=false,
    citecolor=black,
    filecolor=black,
    linkcolor=blue,
    urlcolor=black,
    unicode=true,
    linkbordercolor={1 0 0}
}

\numberwithin{equation}{section}
\theoremstyle{plain}
\newtheorem{theorem}{Theorem}[section]
\newtheorem{lemma}{Lemma}[section]
\newtheorem{propos}{Proposition}[section]

\newtheorem{cor}{Corollary}[section]
\newtheorem{conj}{Conjecture}
\theoremstyle{definition}
\newtheorem{definition}{Definition}
\newtheorem{remark}{Remark}

\newcommand{\ad}{\operatorname{ad}}
\newcommand{\alg}{\operatorname{alg}}
\newcommand{\Aut}{\operatorname{Aut}}
\newcommand{\aut}{\operatorname{aut}}

\newcommand{\band}{\operatorname{band}}
\newcommand{\bandU}{\operatorname{\underline{band}}}
\newcommand{\Band}{\operatorname{Band}}
\newcommand{\bandAut}{\operatorname{bandAut}}
\newcommand{\BandU}{\underline{\operatorname{Band}}}
\newcommand{\Br}{\operatorname{Br}}

\newcommand{\Cart}{\operatorname{Cart}}
\newcommand{\cart}{\operatorname{cart}}
\newcommand{\CC}{\mathbb C}
\newcommand{\Char}{\operatorname{char}}
\newcommand{\coker}{\operatorname{Coker}}

\newcommand{\Div}{\operatorname{Div}}

\newcommand{\End}{\operatorname{End}}

\newcommand{\et}{\operatorname{\acute{e}t}}

\newcommand{\F}{{\mathbb{F}}}
\newcommand{\fppf}{\operatorname{fppf}}
\newcommand{\fpqc}{\operatorname{fpqc}}
\newcommand{\Frob}{\operatorname{Frob}}

\newcommand{\Gal}{\operatorname{Gal}}
\newcommand{\Ger}{\operatorname{Ger}}
\newcommand{\GL}{\operatorname{GL}}
\newcommand{\Gm}{\mathbb{G}_m}

\newcommand{\preBand}{\operatorname{preBand}}

\newcommand{\zz}{{\mathbb Z}}
\newcommand{\qq}{{\mathbb Q}}
\newcommand{\R}{\operatorname{R}}

\newcommand{\HOM}{\operatorname{HOM}}

\newcommand{\id}{\operatorname{id}}

\newcommand{\infl}{\operatorname{inf}}
\newcommand{\im}{\operatorname{Im}}
\newcommand{\Isex}{\operatorname{Isex}}
\newcommand{\Isoc}{\operatorname{Isoc}}
\newcommand{\Isom}{\operatorname{Isom}}

\newcommand{\Mot}{\operatorname{Mot}}

\newcommand{\lcm}{\operatorname{lcm}}
\newcommand{\LOCLIB}{\operatorname{LOCLIB}}
\newcommand{\LocLib}{\operatorname{LocLib}}

\newcommand{\OO}{\operatorname{\mathcal{O}}}
\newcommand{\Vect}{\operatorname{Vect}}

\newcommand{\num}{\operatorname{num}}

\newcommand{\Ob}{\operatorname{Ob}}
\newcommand{\ord}{\operatorname{ord}}

\newcommand{\RGamma}{\operatorname{R\Gamma}}
\newcommand{\REP}{\operatorname{REP}}
\newcommand{\Rep}{\operatorname{Rep}}
\newcommand{\res}{\operatorname{res}}

\newcommand{\sep}{\operatorname{sep}}

\newcommand{\ShGroups}{\operatorname{ShGroups}}
\newcommand{\ShGroupsU}{\underline{\operatorname{ShGroups}}}

\newcommand{\sm}{\operatorname{sm}}

\newcommand{\Stab}{\operatorname{Stab}}

\newcommand \Int{\operatorname {Int}}
\newcommand{\TORS}{\operatorname{TORS}}
\newcommand{\Tors}{\operatorname{Tors}}

\newcommand{\Hom}{\operatorname{Hom}}
\newcommand{\HomI}{\operatorname{\underline{Hom}}}
\newcommand{\Hex}{\operatorname{Hex}}

\newcommand{\Kt}{\operatorname{Kt}}
\newcommand{\Res}{\operatorname{Res}}
\newcommand{\fl}{\operatorname{fppf}}

\newcommand{\Loc}{\operatorname{Loc}}

\newcommand{\Ab}{\operatorname{Ab}}
\newcommand{\op}{\operatorname{op}}
\newcommand{\Rlim}{\operatorname{R}\lim}
\newcommand{\loc}{\operatorname{loc}}
\newcommand {\Spec} {\operatorname{Spec}}
\newcommand{\cycl}{\operatorname{cycl}}
\newcommand{\QQ}{\mathbb Q}

\newcommand{\RR}{\mathbb R}

\newcommand\norm[1]{\left\lVert#1\right\rVert}
\newcommand\restr[2]{{% we make the whole thing an ordinary symbol
  \left.\kern-\nulldelimiterspace % automatically resize the bar with \right
  #1 % the function
  \vphantom{\big|} % pretend it's a little taller at normal size
  \right|_{#2} % this is the delimiter
  }}
\newcommand*\MapsTo{\xrightarrow{{\smash{\ensuremath{\sim}}}}}

\begin{document}
\title{Representations of the Kottwitz gerbes}
\author{Sergei Iakovenko}

\maketitle
\begin{abstract} Let $F$ be a local or global field and let $G$ be a linear algebraic group over $F$. We study Tannakian categories of representations of the Kottwitz gerbes $\Rep(\Kt_{F})$ and the functor $G\mapsto B(F,G)$ defined by Kottwitz in \cite{Kot}. In particular, we show that if $F$ is a function field of a curve over $\mathbb{F}_q$, then $\Rep(\Kt_F)$ is equivalent to the category of Drinfeld isoshtukas. In the case of number fields, we establish the existence of various fiber functors on $\Rep(\Kt_{\mathbb{Q}})$ and its subcategories and show that Scholze's conjecture \cite[Conjecture 9.5]{Sch} follows from the full Tate conjecture over finite fields \cite{Tate1994}.
\end{abstract}

\phantomsection
\hypertarget{MyToc}{} % Make an anchor to the toc
\setcounter{tocdepth}{1}
\tableofcontents

\section*{Introduction}

 \subsection{Context}
 Let $F$ be a local or global field and let $G$ be a linear algebraic group over $F$. In \cite{Kot}, Kottwitz introduced a functor $B(F,G)$ from linear algebraic groups over $F$ to pointed sets. We geometrize a purely cohomological construction of Kottwitz and study categories $\Rep(\Kt_{F})$ that are Tannakian categories over $F$ attached to Kottwitz gerbes $\Kt_{F}$, which are the main ingredients of the construction of $B(F,G)$.

In order to sketch the definition of $\Kt_{F}$ and $B(F,G)$, we first recall Tate-Nakayama isomorphisms for local and global fields. Let $K/F$ be a finite Galois extension. If $F$ is a local field, we set $X_{K/F} = \mathbb{Z}$ and $\mathbb{D}_{K/F} = \Gm$.

\begin{comment}If $F$ is a local field then Tate duality is given by the cup product with the local canonical class $\alpha_{K/F}$ in $H^2(\Gal(K/F), \mathbb{G}_m(K))$. Namely, we have an isomorphism of Tate cohomology groups
\begin{equation}
H^0(\Gal(K/F), \mathbb{Z}) \xrightarrow{\cup\, \alpha_{K/F}} H^2(\Gal(K/F), \mathbb{G}_m(K)) = \Br(K/F).
\end{equation}
\end{comment}
\begin{comment}Note that $\alpha_{K/F}$ can be identified with a generator of $\Br(K/F)\simeq \zz/ n\zz$.
\end{comment}
If $F$ is a global field, then a $\Gal(K/F)$-module $X_{K/F}$ is defined as the kernel of the following homomorphism of $\Gal(K/F)$-modules
\begin{equation*}
\bigoplus_{v\in V_{K/F}} \zz\cdot v \xrightarrow{\deg} \zz, \quad \sum_{v} n_v v \mapsto \sum_{v} n_v
\end{equation*}
where $V_{K/F}$ is the set of all places of $K$ over $F$. Then we define $\mathbb{D}_{K/F}$ as a pro-torus over $F$, whose group of characters is $X_{K/F}$.

There are canonical classes $\alpha_{K/F} \in H^2(\Gal(K/F), \mathbb{D}_{K/F}(K))$, such that
the cup product
\begin{equation}
\label{GlTD}
H^0(\Gal(K/F), X_{K/F}) \xrightarrow{\cup\, \alpha_{K/F}} H^2(\Gal(K/F), \Gm(K)) = \Br(K/F)
\end{equation}
is an isomorphism of Tate cohomology groups. Note that the existence of $\alpha_{K/F}$ in the global cases was shown by Tate in \cite{Tate}, where he also established Tate-Nakayama isomorphisms for cohomology of algebraic tori over global fields.

Kottwitz used the canonical classes $\alpha_{K/F}$, or rather the extensions
\begin{equation}
1 \to \mathbb{D}_{K/F}(K) \to \Kt_{K/F} \to \Gal(K/F) \to 1
\end{equation}
corresponding to these classes, to construct the functor $G \mapsto B(F, G)$. More precisely, for any finite Galois extension $K/F$, Kottwitz \cite[Section 2]{Kot} defined the set of equivalence classes of algebraic $1$-cocycles of $\Kt_{K/F}$ with values in $G(K)$ that is denoted $H^1_{\alg}(\Kt_{K/F},G(K))$. If $L\supset K\supset F$ is a tower of finite Galois extensions, then there is a naturally defined map of pointed sets $H^1_{\alg}(\Kt_{K/F},G(K))\to H^1_{\alg}(\Kt_{L/F},G(L))$ arising from a compatibility between the canonical classes $\alpha_{L/F}$ and $\alpha_{K/F}$. Subsequently, in \cite[Section 10]{Kot} the pointed set $B(F,G)$ is defined as follows
\begin{equation*}
B(F,G) = \varinjlim_{K} H^1_{\alg}(\Kt_{K/F}, G(K)),
\end{equation*}
where the colimit is taken over the set of finite Galois extensions $K/F$ contained in a fixed separable closure $\bar{F}$ of $F$. Moreover, the compatibility mentioned above gives rise to the canonical class $\alpha_F$ given by the limit of $\alpha_{K/F}$ (see Section \ref{DefnOfKtG}).
\begin{comment}Using the foundational work of Giraud on non-abelian cohomology \cite{Giraud} and the paper of Breen \cite{Breen}, we interpret both $\alpha_{K/F}$ and the extensions $\Kt_{K/F}$ in terms of gerbes. Then we apply Tannakian duality to obtain  Tannakian categories $\Rep(\Kt_{F})$ which are of particular interest.
\end{comment}

Remarkably, Kottwitz' construction of $B(F,G)$ works uniformly over local and global fields of any characteristic and has other surprising properties. For instance, $B(F,G)$  behaves nicely under localization, meaning that if $F$ is a global field, there are naturally defined maps $B(F,G)\to B(F_v,G)$, where $F_v$ is the completion of $F$ with respect to a chosen valuation. In particular, one gets such a localization in the case where $F$ is $\mathbb{Q}$ and $v$ is archimedean, that is $F_v = \mathbb{R}$. In \cite[Section 5]{Far} Fargues introduced a twistor projective line $\widetilde{\mathbb{P}}^{1}_\mathbb{R}$ and proved that $B(\mathbb{R}, G) = H^1_{\et}(\widetilde{\mathbb{P}}^{1}_\mathbb{R}, G)$. The relation between $B(\mathbb{R}, G)$ and Shimura data was investigated by Jaburi in his Master's thesis \cite{Jab}.

On the other hand, the local $B(F_v,G)$'s for non-archimedean places have long been known for their relation to the theory of isocrystals, which are objects of semi-linear algebra playing an important role in studying $p$-divisible groups, special fibers of Shimura varieties, and the geometry of the Fargues-Fontaine curve.

The category of isocrystals is defined as follows. For simplicity, let $F=\mathbb{Q}_p$ with the residue field $\mathbb{F}_p$, also denote by $\mathbb{F}$ the algebraic closure of $\mathbb{F}_p$. Let $\breve{F}$ denote the field of fractions of the ring of Witt vectors of $\mathbb{F}$. Note that $\breve{F}$ is the completion of the maximal unramified extension of $\mathbb{Q}_p$. Additionally, $\breve{F}$ is endowed with the action of the topological generator of the Galois group of $\mathbb{F}$ over $\mathbb{F}_p$, which we denote by $\sigma$.

An $F$-isocrystal over $\mathbb{F}$ is a finite-dimensional vector space $V$ over $\breve{F}$ with a $\sigma$-linear isomorphism $\alpha: V \MapsTo V$. $F$-isocrystals over an algebraically closed field were classified by Dieudonn\'{e} and Manin (see \cite{Dieu} and \cite{Manin}).

 In \cite{Kot86} and \cite{Kot97}, Kottwitz introduced $F$-isocrystals with $G$-structure  while studying $\sigma$-conjugacy classes of linear algebraic groups over non-archimedean local fields. In particular, he proved the following:
\begin{theorem}[Kottwitz \cite{Kot97}]
\label{Kot97Th}
Let $F$ be a non-archimedean local field of characteristic $0$. Then $B(F,G)$ is isomorphic to the set of $\otimes$-isomorphism classes of exact $\otimes$-functors from the category of representations of $G$ to the category of $F$-isocrystals.
In particular, $B(F, \GL_n)$ is identified with the set of isomorphism classes of $F$-isocrystals of dimension $n$.
\end{theorem}

Note that in \cite{Kot}, Kottwitz does not provide a similar interpretation of $B(F,G)$ for global fields even in the case $G= \GL_n$. Using the foundational work of Giraud on non-abelian cohomology \cite{Giraud} and the paper of Breen \cite{Breen}, both $\alpha_{K/F}$ and the extensions $\Kt_{K/F}$ can be interpreted in terms of gerbes. Recall that gerbes can be viewed as stacks that are locally equivalent to stacks of torsors for some sheaf of groups.
%Consequently, we can apply Tannakian duality to obtain a Tannakian category $\Rep(\Kt_{F})$ which is of particular interest.

 Scholze suggested to apply general Tannakian duality to the Kottwitz gerbes $\Kt_F$. The duality was introduced by Saavedra Rivano \cite{Saa} and completed by Deligne \cite{Del}, it produces Tannakian categories $\Rep(\Kt_F)$  that play the key role in this work.

The significance of $\Rep(\Kt_{\mathbb{Q}})$ was emphasized in the ICM talk of Scholze (see Section 9 in \cite{Sch}), where he noticed the analogy between $\Rep(\Kt_{\mathbb{Q}})$ and the category of isoshtukas over curves over finite fields. The latter was introduced by Drinfeld in his work on the Langlands correspondence in the function field case (see \cite{Dr}). Scholze also stated the following conjecture:
\begin{conj}[Scholze \cite{Sch}]
\label{Scholze}
There is a Weil cohomology theory $H^i_{\Kt_{\mathbb{Q}}}(X)$
for varieties $X$ over
$\mathbb{F}$ taking values in $\Rep(\Kt_{\mathbb{Q}})$. Under the functor $\Rep(\Kt_{\mathbb{Q}}) \to\Rep(\Kt_{\mathbb{Q}_p})$, it maps to crystalline
cohomology, and under the functor $\Rep(\Kt_{\mathbb{Q}}) \to \Rep(\Kt_{\mathbb{Q}_l})$
for $l\neq p$, it maps to \'{e}tale
cohomology. Under the restriction $\Rep(\Kt_{\mathbb{Q}})\to \Rep(\Kt_{\mathbb{R}})$,
it gives a Weil cohomology theory $H^i_{\Kt_{\mathbb{R}}}(X)$
with values in $\Kt_{\mathbb{R}}$.
\end{conj}

\subsection{Summary of contents and main results}
This work is organized as follows. In the first sections, we recall main definitions related to gerbes and present several approaches to gerbes and their categories of representations. We also prove a general result on gerbes with a smooth band that extends several well-known comparison isomorphisms into the setting of non-abelian $H^2$, namely we show that $H^2_{\sm}(F, G) = H^2_{\fl}(F, G)$, where $G$ is a smooth algebraic group over $F$ (see Theorem \ref{ComparisonNonAb}.)

In Section \ref{GStructure}, we study categories of representations of Kottwitz gerbes with $G$-structure, namely the categories of $\otimes$-functors $\Rep(G) \to \Rep(\Kt_F)$. It is shown in Proposition \ref{RepsGStruct} that the set of $\otimes$-isomorphism classes of such functors is identified with $B(F,G)$.

In Section \ref{NonArchDescr}, we describe $\Rep(\Kt_F)$ in the non-archimedean local cases, re-establishing the equivalence between $\Rep(\Kt_F)$ and the category of isocrystals in a very explicit fashion in contrast to a somewhat more abstract comparison that existed in the literature (see \cite{Kot86}). Namely, using the explicit definition of the category of representations of $\Kt_{F}$ given in Section \ref{ExplDefRep} and main results of local class field theory, we prove the following:
\begin{theorem}[Theorem \ref{TheoremNonArch}]
 Let $F$ be a non-archimedean local field. Then the category of representations of the Kottwitz gerbe $\Rep(\Kt_F)$ is equivalent to the category of pairs $(V,\rho)$, where $V$ is a finite-dimensional vector spaces over $\breve{F}$, and $\rho:V\MapsTo V$ is a $\sigma$-linear automorphism.
\end{theorem}

The case of archimedean local fields is treated in Section \ref{SecArchCases}, where we prove the following:
\begin{theorem}[Theorem \ref{KtRealRep}]
The category of representations of the real Kottwitz gerbe $\Rep(\Kt_\RR)$ is equivalent to the category of finite dimensional $\zz$-graded complex vector spaces with semi-linear automorphism $\alpha$, such that $\alpha^2$ acts on $V^m$ by $(-1)^m$.
\end{theorem}

In Section \ref{FibreFunctorsExist} we get the case of global fields off the ground by showing that $\Rep(\Kt_F)$ has a fibre functor in any extension $\widetilde{F}/F$ of cohomological dimension $\leq1$. In particular, we obtain the following result.
 \begin{theorem}[Theorem \ref{FibreFunctors}]
 \begin{itemize}
\item $\Rep(\Kt_{\mathbb{Q}})$ has a fiber functor in finite-dimensional vector spaces over $\mathbb{Q}^{\cycl}$, where $\mathbb{Q}^{\cycl}$ is the extension of $\mathbb{Q}$ obtained by adjoining all roots of unity.
 \item If $F$ is the function field of an algebraic curve over $\mathbb{F}_q$, then $\Rep(\Kt_F)$ has a fiber functor in finite-dimensional vector spaces over $\breve{F} = F\otimes_{\mathbb{F}_q}\mathbb{F}$.
 \end{itemize}
 \end{theorem}
\noindent Roughly speaking, the first statement says that $\Rep(\Kt_{\mathbb{Q}})$ can be described as the category of representations of the protorus $\mathbb{D}_{\mathbb{Q}}$ over $\mathbb{Q}^{\cycl}$ with extra data corresponding to descent from $\mathbb{Q}^{\cycl}$ to $\mathbb{Q}$. Unfortunately, this descent data cannot be made sufficiently explicit at the moment, due to its highly abstract nature.

The second part of Theorem \ref{FibreFunctors} leads us to the work of Drinfeld on shtukas and Langlands correspondence for curves over a finite field. Namely, in \cite{Dr} Drinfeld introduced the category of isoshtukas that is defined as the category of pairs $(V,\varphi)$, where $V$ is a finite-dimensional vector space over $F\otimes_{\mathbb{F}_q}\mathbb{F}$, and $\varphi$ is a $\Frob$-linear automorphism of $V$, where $\Frob$ is a toplogical generator of the Galois group of $\mathbb{F}$ over $\mathbb{F}_q$.

In Section \ref{LocalGlobalPrinSec}, we obtain one of the main technical ingredients needed for the comparison of $\Rep(\Kt_{F})$ and the category of isoshtukas. Namely, in Theorem \ref{LocalGlobalTheorem} we establish a local-global principle for Tannakian categories over $F$ banded by $\mathbb{D}_F$, where $\mathbb{D}_F$ is the pro-torus with the module of characters $X_F = \varinjlim X_{K/F}$.

The classification of isoshtukas obtained by Drinfeld combined with local Tate's duality and the local-global principle  give us the following result:
\begin{theorem}[Theorem \ref{FunctionFields}]
Let $F$ be a function field of a smooth algebraic curve over a finite field. Then the category of representations of the Kottwitz gerbe $\Rep(\Kt_F)$ is equivalent to the category of Drinfeld isoshutkas.
\end{theorem}

As a consequence, we obtain the following description of $B(F,G)$ for global function fields:
\begin{theorem}[Theorem
\ref{BFG}]
\label{BFGIntro}
Let $F$ be a global function field and let $G$ be a linear algebraic group over $F$, then $B(F,G)$ is isomorphic to the set of isomorphism classes of isoshtukas with $G$-structure.
\end{theorem}
Note that $G$-isoshtukas for connected reductive groups $G$ were recently studied by Hamacher and Kim \cite{HamKim} from a somewhat different perspective.  Namely, they defined isoshtukas as $G$-torsors $V$ over $F\otimes_{\mathbb{F}_q}\mathbb{F}$ together with an isomorphism $\Frob^\ast V \simeq V$ and classified them very explicitly \cite[Theorem 1.3]{HamKim}. Since $G$ is connected reductive, then $H^1(F\otimes_{\mathbb{F}_q}\mathbb{F}, G) = 0$ by Steinberg theorem (see \cite{St}). Consequently, any two fiber functors on $\Rep(G)$ in $\Vect_{F\otimes_{\mathbb{F}_q}\mathbb{F}}$ are isomorphic and the set of isomorphism classes in $\Hom^\otimes(\Rep(G),\Rep(\Kt_F))$ can be computed after fixing fibre functors. A standard Tannakian argument (see \cite[Chapter IX]{DOR}) shows that the set of isomorphism classes of $G$-isoshtukas of Hamacher and Kim is given by $B(F,G)$.

As it is often the case for seemingly analogous results for function fields and number fields, the study of the Kottwitz gerbes and their representations over number fields is admittedly more complicated.  One can study $\Rep(\Kt_{\mathbb{Q}})$ via its subcategories given by the categories of representations of algebraic quotients of the Kottwitz gerbe $\Kt_{\mathbb{Q}}$. Recall that a gerbe is called algebraic if it is locally equivalent to a stack of torsors for an algebraic group.

In Section \ref{MoreonFiber}, we prove the following refinement of the results on the existence of fibre functors.
Let $\Kt_\mathbb{Q}^\prime$ be an algebraic quotient $\Kt_{K,S}$ of $\Kt_{\mathbb{Q}}$, where $K/\mathbb{Q}$ is a finite Galois extension and $S$ is a finite set of places of $K$ satisfying Kaletha-Kottwitz-Tate conditions (see Section \ref{TateKalethaKottwitzCond}).
\begin{theorem}[Theorem \ref{FineFiberFunctors}]
For any algebraic quotient $\Kt_{\mathbb{Q}}^\prime$ of $\Kt_{\mathbb{Q}}$ as above, there exists infinitely many rational primes $q$, such that $\Rep(\Kt_{\mathbb{Q}}^\prime)$ has a fiber functor in finite-dimensional vector spaces over the $q$th cyclotomic field $\mathbb{Q}(\zeta_q)$.
\end{theorem}
\noindent
The main advantage here is that we are reduced to finite cyclic Galois extensions that are ramified at a unique finite prime, namely at $q$. This makes the descent data on $\Rep(\Kt_{\mathbb{Q}}^\prime)$ more concrete and easier to handle, as compared to the situation of Theorem \ref{FibreFunctors}. However, it raises several tricky questions about explicit computation of cohomology of certain non-split tori over $\mathbb{Q}(\zeta_q)$.

Finally, in Section \ref{ScholzeTate}, we show the following:
\begin{theorem}[Corollary \ref{TateScholze}] Scholze's Conjecture \ref{Scholze} follows from the Tate's conjecture over finite fields.
\end{theorem}
\noindent The key ingredient in proving this result is the explicit description of the category of numerical motives $\Mot(\mathbb{F})_{\num}$ obtained by Langlands and Rapoport in \cite{LR} under the assumption of the Tate conjecture (see Conjecture \ref{TateConjecture}), that was refined in subsequent works of Milne \cite{Milne1994} and \cite{MilneGerbes}.

\subsection*{Acknowledgements}
I am deeply grateful to my advisor Peter Scholze for introducing me to this problem, sharing his intuition on Kottwitz gerbes, and for numerous discussions that helped me to learn and understand a lot. I would also like to thank all the members of the Arithmetic Geometry group of the University of Bonn, with whom I had a chance to interact during my Ph.D. In particular, I express my gratitude to Alex Ivanov for his support and several useful conversations and suggestions. Likewise, I wish to thank Mitya Kubrak for various discussions and remarks concerning my results. In addition, I am very thankful to Eugen Hellmann and Eva Viehmann for the opportunities to present my work at seminars at the University of M\"{u}nster and the Technical University of Munich, respectively. Of course, I could not imagine my years in Bonn without Jo\~{a}o Louren\c{c}o, Mafalda Santos, Maxim Smirnov, and Dimitrije Cicmilovi\'{c}. Special thanks go to Tasho Kaletha for his constant support. Lastly, I wish to thank the University of Bonn and the Max Planck Institute for Mathematics for providing great platforms for my Ph.D. studies.

\section{Gerbes: from abstract definitions to group extensions}
\label{GeneralitiesOnGerbes}
A gerbe $\mathcal{E}$ over a site $S$ is defined as a stack fibred in groupoids that satisfies the following extra conditions:
\begin{itemize}
\item There exists a cover $U \in \Ob(S)$ such that the fibre groupoid $\mathcal{E}_U$ is not empty.
\item For all $U\in\Ob(S)$, any two objects  $x, y \in \mathcal{E}_U$ are locally isomorphic, i.e. there exists a cover $U^\prime \xrightarrow{p} U$ such that $p^\ast x \simeq p^\ast y$.
\end{itemize}

In this work we consider fppf, fpqc and smooth sites over a field $F$. Gerbes form a $2$-subcategory of the $2$-category of stacks.

We refer the reader to \cite{Giraud} for generalities on non-abelian cohomology and gerbes.

\subsection{Local description of gerbes}
\label{LocDescrG}
We recall several results on local description of gerbes (see Section 4 in \cite{Breen}). Firstly, we sketch the main steps in describing a gerbe via local data consisting of a concrete bitorsor satisfying several compatibility conditions, we proceed further by describing $1$- and $2$-morphisms of gerbes via local data.

Let $\mathcal{E}$ be a gerbe over $S$. Choosing a cartesian section $x$ of $\mathcal{E}$ over some $U\in\Ob(S)$, we obtain an equivalence
\begin{equation}
\label{EquivAbove}
\Phi : \mathcal{E}\times U \to \TORS(U, G)
\end{equation}
that sends an object $y$ of the fiber groupoid $\mathcal{E}_V$ over $V\xrightarrow{p}U$ to the torsor $\Isom(p^\ast x, y)$. In this situation, we say that $G=\Aut(x)$ is the band of $\mathcal{E}$, or that $\mathcal{E}$ is banded by $G$. When $G$ is a smooth algebraic group scheme, $\mathcal{E}$ is called a smooth algebraic gerbe.

Throughout the rest of the section, we assume that all gerbes are banded by smooth bands, since it is always the case in our applications.

After pulling back the equivalence (\ref{EquivAbove}) to $U\times U$, we get an equivalence of stacks over $U^2$
\begin{equation}
\label{DescentGerbe}
\varphi: p_1^\ast \Phi \circ (p_2^\ast\Phi)^{-1} : \TORS(U^2, p_1^\ast G) \to \TORS(U^2, p_2^\ast G),
\end{equation}
which restricts to the identity on the diagonal.

Using the definition of $\Phi$ and the fact that the category of equivalences between $\TORS(U\times U, p_1^\ast G)$ and $\TORS(U\times U, p_2^\ast G)$ is equivalent to the category of $(p_2^\ast G, p_1^\ast G)$-bitorsors (see Chap.~IV, Prop.\,5.2.5 in \cite{Giraud}), we may pass to the   $(p_2^\ast G, p_1^\ast G)$-bitorsor
\begin{equation}
\label{BiTors}
E = \Isom(p_1^\ast x, p_2^\ast x)
\end{equation}
over $U^2$, whose set of sections over $V\xrightarrow{(f,g)} U^2$ is given by the set $\Isom(f^\ast x, g^\ast x)$ of arrows in the groupoid $\mathcal{E}_V$, moreover, $E$ restricts to $G = \Aut(x)$ on the diagonal $\Delta:U \to U^2$.

Note that the assumption that $G$ is smooth implies that $E$ is representable by a smooth scheme over $U^2$. Additionally, since $E$ is a gerbe, the projection map $E \to U^2$ is an epimorphism. Combining the latter with the monomorphism $\Delta^\ast(E)\to E$ we obtain the following sequence of morphisms
\begin{equation}
\label{GroupoidsSeq}
G =\Aut(x) \to E = \Isom(p_1^\ast(x),p_2^\ast(x))\to U^2.
\end{equation}

In order to define the descent data on $\TORS(U,G)$, the equivalence (\ref{DescentGerbe}) and therefore the bitorsor $E$ must satisfy several compatibility conditions after pulling back to $U^3$ and $U^4$; namely, the condition over $U^3$ translates into an isomorphism of $(p^\ast_1 G, p^\ast_3 G)$-bitorsors
\begin{equation}
\label{CocyclePsi}
\psi: p^\ast_{12} E \wedge p^\ast_{23} E \MapsTo p^\ast_{13}E
\end{equation}
given by associating an element of $\Isom(p^\ast_3(x), p^\ast_1(x))$ to the composition of $f \in \Isom(p^\ast_2 x, p^\ast_1 x)$ and $g \in \Isom(p^\ast_3 x, p^\ast_2 x)$. The condition over $U^4$ boils down to the equality of the two possible induced maps
\begin{equation}
\label{Associativity}
p^\ast_{12} E \wedge p^\ast_{23} \wedge p^\ast_{34} E \to p^\ast_{14}E.
\end{equation}

\subsection{Local description of morphisms between gerbes}
\label{LocDescrMms}
Given a gerbe $\mathcal{E}^\prime$ with a section $y\in \mathcal{E}^\prime_U$ such that $\Aut(y) = H$, where $H$ is a smooth algebraic group over $U$, we let $\varphi^\prime$ and  $\psi^\prime$ denote the equivalence analogous to (\ref{DescentGerbe}) and the cocycle isomorphism (\ref{CocyclePsi}) attached to $\mathcal{E}^\prime$ and $y$, respectively. One describes a morphism $\mathcal{E} \to \mathcal{E}^\prime$ locally as follows. Note that giving a morphism $f: \mathcal{E} \to \mathcal{E}^\prime$ is equivalent to giving a morphism $f_U : \TORS(U, G) \to \TORS(U, H)$ and an isomorphism
\begin{equation}
\label{MorphDiagr}
\alpha: p^\ast_2 f_U \circ \varphi \MapsTo \varphi^\prime \circ p^\ast_1 f_U
\end{equation}
that is compatible with the cocycle isomorphisms $\psi$ and $\psi^\prime$ over $U^3$.

Let $E^\prime= \Isom(p^\ast_1(y), p^\ast_2(y))$ be the $(p^\ast_1 H,p^\ast_2 H)$-bitorsor attached to $\mathcal{E}^\prime$ and a trivializing object $y\in\mathcal{E}^\prime_U$. Assuming for simplicity that $f_U(x) \simeq y$ in $\mathcal{E}^\prime_U$, we obtain a morphism of bitorsors $\alpha\circ p_2^\ast f_U: E \to E^\prime$ over $U^2$ compatible with the cocycle isomorphisms and whose restriction to the diagonal defines a morphism $\nu: G\to H$ of smooth algebraic groups  over $U$.

\subsection{Specializing to the case of Galois covers}
\label{PassingToGalois}
Let $U=\Spec(K)$, where $K/F$ is a Galois extension. In this case, we have the isomorphism $\Spec(K)\times_{\Spec(F)}\Spec(K) = \Gal(K/F)\times \Spec(K)$ dual to the isomorphism $K\otimes_{F} K \to \prod_{\Gal(K/F)} K$ given by sending $a\otimes b$ to $(a\sigma(b))_{\sigma\in\Gal{K/F}}$. Using this isomorphism and taking $U$-points of schemes in (\ref{GroupoidsSeq}), we obtain the short exact sequence of groups
\begin{equation}
\label{ExSeqOfPoints}
1 \to G(K) \to E(K) \to \Gal(K/F) \to 1.
\end{equation}
For brevity, extensions of the form (\ref{ExSeqOfPoints}) together with a sheaf of groups $G$ over $\Spec(U)$ are called $K/F$-Galois gerbes.

The group structure on $E(K)$ is given by the isomorphism $\psi$ in (\ref{CocyclePsi}). More explicitly, we identify $U^3$ with $U\times\Gal(K/F)\times\Gal(K/F)$, so that $p^\ast_1(x) = x$, $p^\ast_2(x) = \sigma^\ast x$, and $p^\ast_3(x) = (\sigma\tau)^\ast x$, where $\sigma, \tau\in \Gal(K/F)^2$. We can rewrite the isomorphism $\psi$ as the family of isomorphisms
\begin{equation}
\psi_{\sigma,\tau}: \sigma^\ast\Isom(\tau^\ast x, x)\wedge \Isom(\sigma^\ast x,x) \to \Isom((\sigma\tau)^\ast x,x)
\end{equation}
given by assigning to a pair of sections $(v,u) \in \Isom(\sigma^\ast x,x)(U)\times \Isom(\tau^\ast x,x)(U)$ their composition $u\sigma^\ast(v) \in \Isom((\sigma\tau)^\ast(x),x)(U)$. The collection of isomorphisms $(\psi_{\sigma,\tau})_{(\sigma,\tau)\in \Gal(K/F)}$  defines the multiplication on $E(K)$, its associativity is the consequence of the equality of two induced maps in (\ref{Associativity}). We also see that $G(K)$ is identified with the fiber of $E(K)\to \Gal(K/F)$ over the identity.

Similarly, under the assumptions of Section \ref{LocDescrMms}, a morphism of gerbes $f: \mathcal{E} \to \mathcal{E}^\prime$ induces a homomorphism of the corresponding $K/F$-Galois gerbes
\begin{equation}
\begin{tikzcd}
1 \arrow{r} & G(K) \arrow{r} \arrow{d}{\widetilde{\nu}} & E(K) \arrow{r}\arrow{d}{\widetilde{f}} & \Gal(K/F) \arrow{r} \arrow[equal]{d}& 1\\
1 \arrow{r} & H(K) \arrow{r}          & E^\prime(K) \arrow{r}   & \Gal(K/F) \arrow{r}& 1
\end{tikzcd}
\end{equation}
where $\widetilde{f}$ and $\widetilde{\nu}$ denote the restrictions to $U$-points of morphisms $\alpha\circ p_2^\ast f_U$ and $\nu$, respectively. We will see further that in some cases $f$ can be recovered up to isomorphism from such data.

\section{Bands}
In this section we recollect briefly the main definitions and constructions leading to the definition of a banded gerbe.
\subsection{Stack of bands}
Let $S$ be a site. We let $\ShGroups(S)$ denote the cloven category of sheaves of groups over $S$ and let $\ShGroupsU(S)$ denote the category of cartesian sections of $\ShGroups(S)$. For any two sheaves of groups $X,Y$ we define the sheaf of external morphisms
\begin{equation}
\Hex(X,Y) = \Hom_{\ShGroups}(X,Y)/\Int(Y)
\end{equation}
and the sheaf of external isomorphisms $\Isex(X,Y)$ as the image of $\Isom(X,Y)$ in $\Hex(X,Y)$.

Firstly, we define a prestack of bands over $S$. For any $U\in \Ob(S)$ we define $\preBand(S)_U$ as follows
\begin{equation}
\Ob(\preBand(S)_U) = \Ob(\ShGroups(S)_U)
\end{equation}
and for any two objects in $\preBand(S)_U$ we set
\begin{equation}
\Hom_{\preBand}(X,Y) = H^0(S_{/ U},\Hex(X,Y))
\end{equation}
moreover, the invertible morphisms $\Isom_{\preBand}(X,Y)$ are given by the sections of $\Isex(X,Y)$ over $S_{/U}$. Since $\Hex(X,Y)$ commutes with restriction morphisms $V\to U$ we have a cloven category $\preBand(S)$, and a functor between cloven categories
\begin{equation}
\label{PreBand}
\ShGroups(S) \to \preBand(S).
\end{equation}

The latter category is a prestack that admits a natural functor
\begin{equation}
\label{PreBandBand}
\preBand(S) \to \Band(S)
\end{equation}
where $\Band(S)$ is obtained as an associated stack and is called the stack of bands over $S$. Moreover, we define a functor
\begin{equation}
\band(S): \ShGroups(S) \to \Band(S)
\end{equation}
by composing (\ref{PreBand}) and (\ref{PreBandBand}).

A cartesian section of $\Band(S)$ is called a band over $S$. The category of bands over $S$ is defined by taking cartesian sections of $\Band(S)$ and is denoted by $\BandU(S)$. By $\bandU(S)$ we denote the functor induced by $\band(S)$ on the categories of cartesian sections
\begin{equation}
\bandU(S): \ShGroupsU(S) \to \BandU(S).
\end{equation}

\subsection{Band acting on a stack} Let $\mathcal{E}$ be a stack over $S$, recall that $\mathcal{E}^{\cart}$ denotes the stack in groupoids associated to $\mathcal{E}$. There is a cartesian functor
\begin{equation}
\aut(E): \mathcal{E}^{\cart} \to \ShGroups(S)
\end{equation}
defined by taking sheaves of automorphisms of $x\in \mathcal{E}_U$, where $U\in \Ob(S)$. Therefore, there is a natural morphism of stacks
\begin{equation}
\bandAut: \mathcal{E}^{\cart} \to \Band(S)
\end{equation}
given by the composition of $\aut(S)$ and $\band(S)$. Note that the stack $\Band(S)$ is defined in such a way that  $\bandAut(S)$ sends any two isomorphisms $m,n: x\rightrightarrows y$ between elements of $\mathcal{E}_U$ into the same morphism.

Furthermore, any morphism of stacks induces a morphism of bands by passing to the associated stacks.

Let $L$ be a band over $S$. An action of $L$ on $\mathcal{E}$ is a morphism of morphisms of stacks $a: L\circ f\to \bandAut(S)(\mathcal{E})$ in the following diagram
\[
\begin{tikzcd}
\mathcal{E}^{\cart} \arrow{dr}[below]{f} \arrow{rr}{\bandAut(S)}{}
& & \Band(S)  \\
& S  \arrow{ur}[below]{L}
\end{tikzcd}
\]

In other words, $a$ is a family of morphisms
\begin{equation}
\label{LocalBand}
a(U): L(U)\to \band(\Aut_U(x)),\quad U\in\Ob(S), x\in \Ob(\mathcal{E}_U).
\end{equation}
The triple $(L,a,\mathcal{E})$ is called a stack with an action of a band.

Given two stacks with actions of bands $(L, a, \mathcal{E})$ and $(M, b, \mathcal{G})$, there is an obvious notion of compatibility for a morphism of stacks $m:\mathcal{E}\to \mathcal{H}$ and a morphism of bands $u: L\to M$.

\subsection{Band of a gerbe}
Let $\mathcal{E}$ be an $S$-gerbe, and let $(L,a)$ be a band acting on $\mathcal{E}$. The following conditions are equivalent:
\begin{itemize}
\item $a$ is an isomorphism.
\item for any band $L^\prime$ over $S$ and any action $b^\prime$ of $L^\prime$ on $\mathcal{E}$, there exists a unique morphism $u: L^\prime \to L$ making $b^\prime$ an  action induced via $u$ and $a$. That is $b^\prime = a\circ (u\cdot f)$.
\end{itemize}
Moreover, there exists a band $(L,a)$ satisfying the conditions above. For the proof see  IV.Proposi\-ti\-on 2.2.1 in \cite{Giraud}.

A morphism of gerbes $f: \mathcal{E} \to \mathcal{H}$ is banded by a morphism of bands $\nu: L \to M$ if the following diagram commutes
\[
\begin{tikzcd}
\band(\restr{\Aut(x)}{U}) \arrow{r}{f}  & \band(\restr{\Aut(f(x))}{U})  \\
L(U) \arrow{r}{\nu} \arrow{u} & M(U) \arrow{u}
\end{tikzcd}
\]
where the vertical arrows are the isomorphisms from (\ref{LocalBand}). Furthermore, any morphism of gerbes is banded by a unique morphism of bands.

\begin{definition}
\label{GiraudH2}Let $L$ be a band over $S$ and let $\mathcal{E}$ and $\mathcal{H}$ be gerbes banded by $L$. A morphism $f:\mathcal{E}\to\mathcal{H}$ is called an $L$-morphism or $L$-equivalence, if $f$ is banded by $\id_L$.
 The second non-abelian cohomology set of Giraud $H^2(S, L)$ is defined as the set of $L$-equivalence classes of $L$-gerbes (see IV.3.1 in \cite{Giraud}).  If $G$ is a sheaf of groups over $S$, we set $H^2(S, G) = H^2(S,\band(G))$.
\end{definition}
Note that the substack of abelian bands is equivalent to the stack of abelian groups over $S$ (see Proposition IV.1.2.3 in \cite{Giraud}). Moreover, if $L$ is an abelian band, $H^2(S, L)$ is endowed with a group structure and coincides with the usual derived functor definition of $H^2$ (see Theorem IV.3.4.2 in \cite{Giraud})
\begin{comment}
The proof is completely formal, but we sketch the main points here. The functor in one direction is defined by noticing that applying
the isomorphism (\ref{MorphDiagr}) to $p^\ast_1 G$ considered as a trivial $p^\ast_1G$-torsor and noticing that $E = \varphi(p^\ast_1 G)$ is a $(p^\ast_1 G, p^\ast_2 G)$-bitorsor. Consequently, $\alpha\circ p_2^\ast f_U$ defines a morphism from $E$ to $\varphi^\prime \circ p^\ast_1 f_U (p^\ast_1 G)$. Without loss of generality we can assume that the latter is given by $F$, since we can compose $p^\ast_1 f_U$ with twisting by the image of $p^\ast_1 G$; thus we obtain the desired morphism of bitorsors. Going in the other direction, we recover a functor between the stacks $\TORS(U, G)$ and $\TORS(U, H)$ from the given morphism $G\to H$ of restrictions of the bitorsors $E$ and $F$ to the diagonal. Then we can easily recover an isomorphism (\ref{MorphDiagr}) by using the equivalence stated above and the fact that all the morphisms in question can be described explicitly in terms of contracted products.

Note that so far we have not used the assumption that $U = \Spec(K)$, where $K/F$ is a finite Galois extension, which will be crucial in Section \ref{explBreen}.
\end{comment}

\section{Gerbes with a smooth band and a comparison isomorphism for $H^2$}
In this section, we establish a comparison isomorphism for non-abelian $H^2$ that would be applied in Section \ref{GStructure}. Namely, we prove the following
\begin{theorem}
\label{ComparisonNonAb}
Let $F$ be a field, and let $G$ be a smooth group scheme over $F$. Then the following pointed sets classifying $G$-gerbes over the corresponding sites are isomorphic
\begin{equation}
H^2(F_{\sm}, G) = H^2(F_{\fl}, G) = H^2(F_{\fpqc}, G).
\end{equation}
\end{theorem}
\noindent The proof of this theorem will be given in Section \ref{ComparisonProof} after we develop main technical tools in Sections \ref{InverseDirect}--\ref{ShapiroSubs}.

\subsection{Inverse and direct images of stacks}
\label{InverseDirect}
Let us first recall the definitions of the direct and inverse image $2$-functors defined for stacks.

Let $S$ and $S^\prime$ be two sites.  Let $f^{-1}: S \to S^\prime$ be an adjoint functor to a morphism of sites
\begin{equation}
f: S^\prime \to S.
\end{equation}
Note that we are primarily interested in the case where $S=\Spec(F)_{\fpqc}$ and $S^\prime=\Spec(K)_{\fpqc}$, and $K/F$ is a finite field  extension.
\begin{definition}
For any stack $\mathcal{E}^\prime$ over $S^\prime$ its direct image $f_\ast(\mathcal{E}^\prime)$ is given by the stack $\mathcal{E}^\prime\times_{S^\prime} S$ obtained from $\mathcal{E}^\prime$ via base change along $f^{-1}$. Namely, the fibre of $f_\ast(\mathcal{E}^\prime)$ over $U \in \Ob(S)$ is given by $\mathcal{E}^\prime(f^{-1}(U))$.
\end{definition}

\begin{definition}
A pair $(\mathcal{E}^\prime, \varphi)$ where $\mathcal{E}^\prime$ is a stack over $S^\prime$ and $\varphi: \mathcal{E} \to f_{\ast}(\mathcal{E}^\prime)$ a morphism of stacks over $S$ is called an inverse image of $\mathcal{E}$ if for any stack $\mathcal{G}^\prime$ over $S^\prime$ the functor
\begin{equation}
\label{AdjunctionStacks}
\Cart_{S^\prime}(\mathcal{E}^\prime, \mathcal{G}^\prime) \to \Cart_S (\mathcal{E}, f_\ast(\mathcal{G}^\prime))
\end{equation}
is an equivalence of categories.
\end{definition}

In particular, the sites $S$ and $S^\prime$ can be considered as stacks in groupoids over themselves, in this case the couple $(S^\prime, \id_{S})$ is an inverse image of $S$.

Given a stack $\mathcal{E}$ over $S$ and a morphism of sites $S^\prime \to S$ one can construct a pair $(f^\ast(\mathcal{E}), \varphi)$ passing to the underlying fibred category and then stackifying the inverse image of $\mathcal{E}$ as a fibred category, note that $\varphi: \mathcal{E}\to f_\ast f^\ast \mathcal{E}$ is a natural functor defined by $2$-adjunction of direct image and inverse image functors of underlying fibred categories and the stackifickation.

\subsection{Inverse and direct images of gerbes}
If $\mathcal{E}$ is a gerbe over $S$ then its inverse image $f^\ast \mathcal{E}$ is a gerbe over $S^\prime$ (see III. Corollaire 2.1.5.6 in \cite{Giraud}), thus gerbes are well-behaved under $f^\ast$, however, given a gerbe $\mathcal{E}^\prime$ over $S^\prime$, its direct image $f_\ast(\mathcal{E}^\prime)$ is not a gerbe in general.

\begin{comment}
Note that in \cite{Giraud}, Giraud does not assume that stacks are fibered in groupoids, however for a given stack $E$ one can construct a stack $E^{\cart}\subseteq E$, that has the same set of objects as $E$ and the morphisms are morphisms in $E$ that are cartesian. In addition, any functor $E \to E^\prime$ induces a functor $E^{\cart} \to (E^\prime)^{\cart}$.
\end{comment}

 From now on, we assume that $\mathcal{E}$ is a gerbe over $S$ banded by a sheaf of groups $G$, such that its inverse image $(\mathcal{E}^\prime,\varphi:\mathcal{E}\to f_\ast \mathcal{E}^\prime))$ along $f:S^\prime \to S$, has a section, that is there exists a cartesian functor $S^\prime\to \mathcal{E}^\prime$. Thus $\mathcal{E}^\prime$ is equivalent to the neutral gerbe $\TORS(S^\prime, G^\prime)$, where $G^\prime = f^\ast(G)$ is the inverse image of the sheaf of groups associated to $G$.

We would like to deduce a sufficient criterion for $f_\ast(\TORS(S^\prime, G^\prime))$ to be a gerbe. In order to achieve this, we apply theory developed by Giraud in  \cite[V.3]{Giraud}. For a sheaf of groups $G^\prime$ over $S^\prime$ he defines a sheaf $R^1 f_\ast(G^\prime)$ on $S$ that is associated to a presheaf
\begin{equation}
V\in \Ob(S) \rightsquigarrow H^1(f^{-1}(V), G^\prime).
\end{equation}

There is a canonical isomorphism of sheaves between $R^1f_\ast(G^\prime)$ and the sheaf of maximal subgerbes of $f_\ast(E^\prime)$ that is denoted by $\Ger(\mathcal{E}^\prime)$ (see \cite[V. Lemme 3.1.5]{Giraud}). This implies that $f_\ast \mathcal{E}^\prime$ admits a fully faithful functor $\mathcal{H}\to f_\ast(\mathcal{E}^\prime)$ where $\mathcal{H}$ is a gerbe over $S$, if and only if $R^1f_\ast(G^\prime)$ has a section.

Moreover, there is an exact sequence of pointed sets (see  \cite[V. Prop. 3.1.3]{Giraud})
\begin{equation}
\label{ExSeqNonab}
\begin{tikzcd}
0 \arrow{r}& H^1(S, f_\ast(G^\prime)) \arrow{r} &H^1(S^\prime, G^\prime) \arrow{r}{\gamma} & H^0(S, R^1f_\ast(G^\prime))
\end{tikzcd}
\end{equation}
where the first arrow is induced by an equivalence of stacks
\begin{equation}
\label{STors}
\TORS(S^\prime, G^\prime)^{S} \simeq \TORS(S, f_\ast(G^\prime))
\end{equation}
where $\TORS(S^\prime, G^\prime)^{S}$ denotes a full subcategory of $\TORS(S^\prime, G^\prime)$ whose objects are $G^\prime$-torsors $P$ over $S^\prime$, such that there exists a refinement $R$ in $S$ satisfying the following condition: for any $Y\in R$, the restriction of $P$ to $f^{-1}(Y)$ is trivial (see \cite[V. Prop. 3.1.1]{Giraud}.)

Clearly, the vanishing of $R^1f_\ast(G^\prime)$ implies several particularly nice properties that can be summarized as follows:
\begin{propos}
\label{ShapiroLExt} If $R^1f_\ast(G^\prime) =\{\ast\}$, then the following is satisfied:
\begin{itemize}
\item $\TORS(S^\prime, G^\prime) \simeq \TORS(S, f_\ast(G^\prime))$.
\item $f_\ast(\TORS(S^\prime, G^\prime)) \simeq \TORS(S, f_\ast(G^\prime))$.
\end{itemize}
\end{propos}
\begin{proof} Let $P$ be a $G^\prime$-torsor. Then $f_\ast(P)$ is an $f_\ast(G^\prime)$-torsor, if and only if the class of $P$ in $H^1(S^\prime, G^\prime)$ maps to the marked point in $H^0(S, R^1f_\ast(G^\prime))$ (see \cite[V. 3.1.3.1]{Giraud}). Therefore, if $R^1f_\ast(G^\prime) =\{\ast\}$, the natural functor $\TORS(S^\prime, G^\prime)^S\to  \TORS(S^\prime, G^\prime)$ is an equivalence. Then the first claim follows by combining this with  (\ref{STors}.)

In order to prove the second claim, we notice that $f_\ast(\TORS(S^\prime, G^\prime))$ is a gerbe, since it is locally non-empty by definition and it follows from the vanishing of $R^1f_\ast(G^\prime)$ that any local section of $f_\ast(\TORS(S^\prime, G^\prime))$ is $S$-locally isomorphic to the trivial $f_\ast(G^\prime)$-torsor. The rest follows from the fact that $\TORS(S, f_\ast(G^\prime))$ is a unique maximal subgerbe of $f_\ast(\TORS(S^\prime, G^\prime))$ (see \cite[V. Prop. 3.1.6]{Giraud}).
\end{proof}

\subsection{Vanishing of $R^1f_\ast(G^\prime)$ for smooth $G^\prime$ and finite $f$} Although we are going to apply Proposition \ref{ShapiroLExt} in a very particular case, it can be useful to have a more general statement.
\begin{propos}
\label{VanishingR1}
 Let $f: S^\prime \to S$ be a morphism of $\fppf$-sites of schemes induced by a finite morphism of the corresponding schemes and let $G^\prime$ be a sheaf of smooth groups over $S^\prime$. Then $R^1f_\ast(G^\prime) = \{\ast\}$.
\end{propos}
\begin{proof}
We remark that the proof follows the lines of \cite[Exp. VIII, par. 5]{SGA4t2} where the derived direct images of abelian sheaves in \'{e}tale topology are treated.

We claim that $R^1 f_\ast(G^\prime)=\{\ast\}$ as an $\fppf$-sheaf. Recall that the $\fl$-site over a scheme has enough points (see \cite[\href{https://stacks.math.columbia.edu/tag/06VW}{Tag 06VW}]{stacks-project}), that is there exists a conservative family of points $\{p_i\}_{i\in I}$
such that a morphism of sheaves is an isomorphism whenever it holds fiberwise for every $p_i$. In our case the family of $S_{\fl}$-local $S$-schemes defines such a conservative family of fibre functors (see Theorem 2.3 in \cite{GabberKelly}). Since $\fl$ topology is finer than \'{e}tale topology we see that any $S_{\fl}$-local scheme $\Spec(R)\to S$ is in particular strictly henselian; moreover, all its residue fields are algebraically closed.

Thus we are reduced to the computation of fibres of $R^1 f_\ast(G^\prime)$ at points given by $\fl$-local schemes $\overline{S}$. We let $\overline{S^\prime}$ denote the fibre product $S^\prime \times_{S} \overline{S}$ and  let $\overline{G^\prime}$ denote the pullback of $G^\prime$ to $\overline{S^\prime}$. Then we have an isomorphism of sheaves of pointed sets
\begin{equation}
R^1 f_\ast(G^\prime)_{\overline{S}} \simeq H^1_{\fl}(\overline{S^\prime}, \overline{G^\prime})\simeq
H^1_{\et}(\overline{S^\prime}, \overline{G^\prime})
\end{equation}
where the second isomorphism follows from the smoothness of $\overline{G^\prime}$ over $\overline{S^\prime}$. Applying the finiteness of $S^\prime\to S$ we see that $\overline{S^\prime}\to \overline{S}$ is finite, thus $\overline{S^\prime}$ is a finite product of strictly henselian schemes, and over such schemes any torsor under a smooth group is trivial. Therefore
\begin{equation}
R^1 f_\ast(G^\prime)_{\overline{S}} = \{\ast\} \text{ for any } \overline{S} \ \fl\text{-local},
\end{equation}
which proves the claim.
\end{proof}

\subsection{The case of finite field extensions}
\label{ShapiroSubs}
Let $K/F$ be a finite extension of fields and let $G^\prime$ be a smooth algebraic group over $K$. Setting $S^\prime = \Spec(K)_{\fppf}$ and $S = \Spec(F)_{\fppf}$, we see that the assumptions of Proposition \ref{VanishingR1} are satisfied and moreover $f_\ast(G^\prime)$ is representable by a smooth algebraic group $\Res_{K/F}(G^\prime)$ over $F$ (see \cite[A.~Prop.~5.2]{CGP}.) Therefore, we have an equivalence
\begin{equation}
f_\ast(\TORS(K_{\fl}, G^\prime))\to \TORS(F_{\fl}, \Res_{K/F}(G^\prime))
\end{equation}
and a natural functor
\begin{equation}
\label{WeilRestrGerbe}
\theta:  \mathcal{E}\to \TORS(F_{\fl}, \Res_{K/F}(G^\prime))
\end{equation}
 between $F_{\fl}$-gerbes that is given by sending an object $P\in \Ob(E_U)$ to $f_\ast f^\ast P$ which is an object of $\TORS(F_{\fl}, \Res_{K/F}(G^\prime))_U$.

Note that Propositions \ref{ShapiroLExt} and \ref{VanishingR1} combined with Grothendieck comparison isomorphism \cite[Th. 11.7]{GR1} give us an alternative proof of Shapiro-Oesterl\'{e} lemma (see  \cite[IV. 2.3]{Oes}):
\begin{cor} Let $K/F$ be a finite field extension and let $G^\prime$ be a smooth group scheme over $K$. Then the natural functor $\TORS(F_{\fppf}, \Res_{K/F}(G^\prime))\to \TORS(K_{\fppf}, G^\prime)$ defined via \cite[V. Cor. 3.1.2]{Giraud} is an equivalence of gerbes that induces an isomorphism
\begin{equation}
H^1(K_{\et}, G^\prime) \simeq H^1(F_{\et}, \Res_{K/F}(G^\prime))
\end{equation}
of pointed sets, which reduces to the usual isomorphism of abelian cohomology, if $G^\prime$ is abelian.
\end{cor}

\subsection{Comparison isomorphism}
\label{ComparisonProof}
Any gerbe over $F_{\fpqc}$ banded by an algebraic group $G$ can be neutralized by a finite extension $K/F$, here we assume for simplicity that $L$ is a band that is globally representable by a sheaf of groups.
It is a consequence of the comparison isomorphism for $H^2(F_{\fpqc}, G)$ and $H^2(F_{\fl}, G)$ proved by Saavedra Rivano \cite[III. Cor. 3.1.6]{Saa} and the fact that $H^2(F_{\fl}, G)$ is a pseudotorsor under $H^2(F_{\fl}, C(G))$ (see \cite[IV. Th. 3.3.3]{Giraud},) where $C(G)$ is the centre of $G$, and the latter is the usual $\fl$-cohomology of $C(G)$.

We claim that any algebraic gerbe over a field with a smooth algebraic band can be neutralized by a finite Galois extension of $F$.

Note that the discussion above reduces this statement to the case $\Char(F)>0$. Moreover, we may assume that $K/F$ is purely inseparable.

Let $\mathcal{E}$ be a gerbe banded by $G$, and let $f: S^\prime=K_{\fl} \to S= F_{\fl}$ be the corresponding morphism of sites. Assume that $\mathcal{E}^\prime=f^\ast(\mathcal{E})$ is a neutral gerbe $\TORS(S^\prime, f^\ast(G))$.
Consider the fibre product
\[
\begin{tikzcd}
\mathcal{E}\times_{\mathcal{E}_\ast}  S \arrow{r} \arrow{d} & S \arrow{d} \\
\mathcal{E}  \arrow{r}{\theta}                                  & \mathcal{E}_\ast = \TORS(S, \Res_{K/F}(G^\prime))
\end{tikzcd}
\]

We note that the vertical arrow provides us with a morphism $\mathcal{E}\times_{\mathcal{E}_\ast} F \to \mathcal{E}$ and we claim that $\mathcal{E}\times_{\mathcal{E}_\ast} F$ is representable by a smooth scheme over $F$. To show this, we observe that the bottom arrow $\theta$ is representable by a smooth scheme.

Indeed, one can check this after pulling back both gerbes to $K$, and thus reducing to the case of a morphism between neutral gerbes
\begin{equation}
 \theta_K : \TORS(S^\prime, G^\prime) \to \TORS(S^\prime,\Res_{K/F}(G^\prime)).
\end{equation}

By definition of representability, we need to check that for any point
\begin{equation*}
u: U \to \TORS(S^\prime,\Res_{K/F}(G^\prime)),
\end{equation*}
where $U$ is a scheme over $\Spec(K)$, the fibre product
\begin{equation}
\label{fibresgerbes}
\begin{tikzcd}
\mathcal{E}\times_{\mathcal{E}_\ast} U \arrow{r} \arrow{d} & U\arrow{d}{u} \\
\TORS(S^\prime, G^\prime) \arrow{r}{\theta_K}                                  & \TORS(S^\prime, f^\ast\Res_{K/F}(G^\prime))
\end{tikzcd}
\end{equation}
is representable by a smooth scheme.

The fibre is given by an $f^\ast\Res_{K/F}(G^\prime)$-torsor $Q$ over $U$ that lies in the image of $\theta_K$. Since $f^\ast\Res_{K/F}(G^\prime)$ is a smooth group scheme, $Q$ becomes trivial over an \'{e}tale cover  of $K$, therefore, we may assume that $Q$ is a trivial torsor $f^\ast\Res_{K/F}(G^\prime)$. We can describe the functor $\theta_K$ using the following remark.
\begin{remark}
If $u: A\to B$ is a monomorphism of groups, then there is a natural functor $\TORS(S, A) \to \TORS(S,B)$ given by
$P\mapsto  {}^{u}P = P\wedge^{A}B$. There is a natural morphism from $P$ to ${}^{u}P$ that can be composed with with the quotient by $A$, i.e. we have a morphism
\begin{equation}
P \to {}^{u}P \to {}^{u}P/A
\end{equation}
which gives us a section $q: e \to {}^{u}P/A$ since the morphism above factors through $P/A$ (locally isomorphic to $e$, the final object of the site we are working on). Moreover, $\Tors(S, A)$ is equivalent to the category of pairs $(Q,q)$ where $Q$ is a $B$-torsor, and $q \in \Hom(e,Q/A)$.
\end{remark}

Note that $\theta_K$ is defined by adjunction
\begin{equation}
f^\ast P \mapsto f^\ast f_\ast f^\ast P
\end{equation}
where $P$ is an object of a fibre of $\mathcal{E}$, and $f^\ast P$ is the associated $G^\prime$-torsor. On the other hand, we also have a natural monomorphism of sheaves of sets
\begin{equation}
f^\ast P \to f^\ast f_\ast f^\ast P
\end{equation}
that is compatible with the monomorphism of sheaves of groups $G^\prime \to f^\ast \Res_{K/F}(G^\prime)$. In this case we have an isomorphism of $f^\ast \Res_{K/F}(G^\prime)$-torsors (see \cite{Giraud} Prop. III.1.4.6(iii))
\begin{equation}
f^\ast f_\ast f^\ast P \simeq f^\ast P \wedge^{G} f^\ast \Res_{K/F}(G^\prime),
\end{equation}
 which allows us to apply the remark above. It follows that the fibre product (\ref{fibresgerbes}) is given by the quotient of a smooth group scheme $f^\ast\Res_{K/F}(G^\prime)$ by $G^\prime$, and therefore, is representable by a smooth scheme.

We obtain a cartesian functor from the $\fl$-site over the smooth scheme $\Res_{K/F}(G^\prime)/G^\prime$ to the gerbe $\mathcal{E}$, furthermore, $\Res_{K/F}(G^\prime)/G^\prime$ has sections \'{e}tale-locally. This gives us an \'{e}tale section of $\mathcal{E}$.

The previous argument shows that any $G$-gerbe $\mathcal{E}$ over $F_{\fl}$ trivializes over an \'{e}tale extension $K/F$. Taking the inverse image of $\mathcal{E}$ along $F\to K$ we get the category of $G\times_F K$-torsors over $F_{\fl}$, which is equivalent to the category of $G\times_F K$-torsors over $F_{\et}$. Therefore, any two local sections of $\mathcal{E}$ are \'{e}tale-locally isomorphic.

\begin{remark}
A remark of Conrad in \cite[App. B.3]{Con} says that any gerbe $\mathcal{E}$ banded by an abelian $\fl$ $S$-group scheme $G$ is actually an Artin stack, thus it would imply that $\mathcal{E}$ can be neutralized by an \'{e}tale extension by noticing that $H^2_{\fl}(S, G)$ is a pseudo-torsor under $H^2_{\fl}(S, C(G))$.
\end{remark}

\section{Definition of Kottwitz gerbes}
\label{DefnOfKtG}
\subsection{Local fields} Let $F$ be a local field and let $K/F$ be a finite Galois extension with the Galois group $\Gamma_{K/F}$. Then there is an isomorphism between Tate cohomology groups
\begin{equation}
H^0(\Gamma_{K/F}, \zz) \MapsTo H^2(\Gamma_{K/F}, K^\times)
\end{equation}
given by the cup product with the canonical class of local class field theory
\begin{equation*}
\alpha_{K/F} \in H^2(\Gamma_{K/F},\Hom(\zz,K^\times)).
\end{equation*}Using the fact that the latter cohomology group also classifies the extensions of $\Gamma_{K/F}$ by $K^\times$, we define $\Kt_{K/F}$ as an extension
\begin{equation}
1\to K^\times \to \Kt_{K/F} \to \Gamma_{K/F} \to 1
\end{equation}
corresponding to the class $\alpha_{K/F}$. Hilbert's Theorem 90 ensures that $\Kt_{K/F}$ is essentially well-defined, namely, it is defined up to the conjugation by an element of $K^\times$.

In order to define $\Kt_F$, Kottwitz considered the limit of $\Kt_{K/F}$ over finite Galois extensions $K/F$. More accurately, consider $F\subset K\subset L$ such that $L$ is finite Galois over $F$, then the canonical class $\alpha_{K/F}$ can be considered as an element of $H^2(\Gamma_{L/F}, L^\times)$ via inflation, to which one associates a gerbe $\Kt_{K/F}^{\infl}$ (see Section 2.10 in \cite{Kot}). Moreover, there is a relation
\begin{equation}
\infl(\alpha_{K/F}) = [L:K]\alpha_{L/F},
\end{equation}
which implies the existence of a homomorphism $\Kt_{L/F} \to \Kt_{K/F}^{\infl}$ that is given by the map $p_{L/K}: x\mapsto x^{[L:K]}$ when restricted to $L^\times$.

Therefore, $\Kt_{F}$ corresponds to the element $\alpha_{F}\in\varprojlim_{K/F\text{ fin. Galois}} H^2(\Gamma_{K/F}, K^\times)$ defined by the limit of canonical classes $\alpha_{K/F}$.

Applying  Hilbert's Theorem 90, we see that inflation $H^2(\Gamma_{K/F}, K^\times)\to H^2(F_{\et},\Gm)$ is injective, this way $\alpha_{F}$ defines a unique class in $\varprojlim H^2(F,\Gm)$, where the transition maps on cohomology are induced by the maps $p_{L/K}$. Finally, applying Hilbert's Theorem 90 we obtain that $\alpha_F$ can be considered as an element of $H^2(F_{\fpqc}, \widetilde{\Gm})$, where $\widetilde{\Gm}$ is the pro-torus over $F$ whose module of characters is given by $\mathbb{Q}$.

Consequently, the construction of Kottwitz defines an equivalence class of gerbes over $F_{\fpqc}$ banded by $\widetilde{\Gm}$.

\begin{remark} Note that $H^1(F_{\fpqc}, \widetilde{\Gm})\neq0$, despite the fact that $H^1(F_{\fpqc}, \Gm)=0$, namely this cohomology group is given by $\varprojlim^1 F^\times$ with the transition maps given by raising to $n$-th powers for any natural $n$. In general, such transition maps are not surjective and the corresponding inverse system fails Mittag-Leffler condition.
\end{remark}

\subsection{Global fields} Let $F$ be a global field and let $K/F$ be a finite Galois extension. Here we define the global Kottwitz classes and explain that the construction of Kottwitz gives us a well-defined class in $H^2(F_{\fpqc},\mathbb{D}_F)$, that is an equivalence class of gerbes banded by a pro-torus $\mathbb{D}_{F}$ introduced below (see Section \ref{DefOfDF}). Although the situation is quite similar to the local case, there are several differences, for instance, one is forced to prove several vanishing theorems for the first cohomology of various (pro-)tori, that are in general no longer split over $F$.

\label{TateKalethaKottwitzCond}
Let us first recall Tate's construction from \cite{Tate}. Let $K/F$ be a finite Galois extension. Let $S \subset V_F$ be a finite subset of the set of places of $F$, we define $S_K$ as the set of places in $K$ lying over $S$,
and $\dot{S}_K \subset S_K$ denotes a set of lifts for the places in $S$ (that is, over each $v \in S$ there is a
unique $w \in  \dot{S}_K)$. We assume that the pair $(S, \dot{S}_K)$ satisfies the following list of conditions
\begin{itemize}
\item[(T1)] $S$ contains all archimedean places and all places that ramify in $K$.
\item[(T2)] Every ideal class of $K$ has an ideal with support in $S_K$.
\item[(K1)] For every $w\in V_K$  there exists $w^\prime \in S_K$ such that $\Stab(w,\Gamma_{K/F}) = \Stab(w^\prime,\Gamma_{K/F})$.
\item[(K2)] For every $\sigma \in \Gamma_{K/F}$ there exists $\dot{v} \in \dot{S}_K$ such that $\sigma \dot{v} = \sigma{v}$.
\end{itemize}

The conditions (T1) and (T2) were introduced by Tate in \cite{Tate} and are sufficient for the construction of Tate-Nakayama classes, the conditions (K1) and (K2) appeared in a paper by Kaletha (see Conditions 3.3.1 in \cite{Kal}) and are sufficient for proving that $H^1(\Gamma, \mathbb{D}_{K,S}(\bar{F})) = 0$ and an injectivity statement, consequences of which we will recall in Lemma \ref{InjKaletha}.
\subsection{Tate's construction of the canonical classes}
\label{SemiLocal}
Here, $K/F$ is a finite Galois extension of global fields. Let (A) be the following exact sequence of $\Gamma_{K/F}$-modules
\begin{equation}
\tag{A}
0\to \OO_{K,S}^\times \xrightarrow{a^\prime} \mathbb{A}^\times_{K,S} \xrightarrow{a} C_{K,S} \to 1,
\end{equation}
where
\begin{itemize}
\item $\OO_{K,S}^\times \subset K^\times$ is the group of $S$-units.
\item $\mathbb{A}_{K,S}^\times\subset \mathbb{A}_K^\times$ is the group of $S$-id\`{e}les of $K$.
\item $C_{K,S}$ is the group of $S$-id\`{e}le classes of $K$.
\end{itemize}

The second short exact sequence considered by Tate is
\begin{equation}
\tag{B}
0 \to \zz[S_K]_0 \xrightarrow{b^\prime} \zz[S_K] \xrightarrow{b} \zz\to 0,
\end{equation}
where
\begin{itemize}
\item $\zz$ is a trivial $\Gamma_{K/F}$-module.
\item $\zz[S_K]$ is the free abelian group on the finite set of places $S_K$, the $\Gamma_{K/F}$-action being induced by the natural $\Gamma_{K/F}$-action on $V_K$.
\item $b: \sum_{v\in S_K} n_v v \mapsto \sum_{v\in S_K} n_v$.
\item $\zz[S_K]_0$ is the kernel of $b$, and $b^\prime$ the canonical inclusion.
\end{itemize}

Tate also considers the group $\Hom(B,A)$ consisting of all triples
\begin{equation*}
(f_3,f_2,f_1) \in \Hom(\zz[S_K]_0, \OO_{K,S}^\times)\times\Hom(\zz[S_K], \mathbb{A}^\times_{K,S})\times \Hom(\zz, C_{K,S})
\end{equation*}
such that
\begin{displaymath}
\xymatrix{
  0 \ar[r]^{} & \zz[S_K]_0 \ar[d]_{f_3} \ar[r]^{b^\prime} & \zz[S_K] \ar[d]_{f_2} \ar[r]^{b} & \zz \ar[d]_{f_1} \ar[r]^{} & 0  \\
  0 \ar[r]^{} & \OO_{K,S}^\times \ar[r]^{a^\prime} & \mathbb{A}^\times_{K,S} \ar[r]^{a} & C_{K,S} \ar[r]^{} & 0   }
\end{displaymath}
commutes. In order to prove the existence of cohomology classes $\alpha_i^r(K,S)$  for $i=1,2,3$ that fit into the following commutative diagram
\begin{displaymath}
\xymatrix{
  \ldots \ar[r]^{} &H^r(\Gamma_{K/F},\zz[S_K]_0) \ar[d]_{\alpha^r_3} \ar[r]^{b^\prime} & H^r(\Gamma_{K/F},\zz[S_K]) \ar[d]_{\alpha^r_2} \ar[r]^{b} & H^r(\Gamma_{K/F},\zz) \ar[d]_{\alpha^r_1}  \ar[r]^{} & \ldots \\
  \ldots \ar[r]^{} &H^{r+2}(\Gamma_{K/F}, \OO_{K,S}^\times) \ar[r]^{a^\prime} & H^{r+2}(\Gamma_{K/F},\mathbb{A}_{K,S}^\times) \ar[r]^{a} & H^{r+2}(\Gamma_{K/F},C_{K,S}) \ar[r]^{} & \ldots  }
\end{displaymath}
where all the vertical arrows are isomorphisms, one must construct a single canonical class in $H^2(\Gamma_{K/F},\Hom(B,A))$, satisfying several compatibility conditions under maps $u_i$ sending a class $\theta\in H^2(\Gamma_{K/F},\Hom(A,B))$ to its $i$th component for $i=1,2,3$.

Moreover, everything boils down to finding $\alpha^r_1$ and $\alpha^r_2$, since we have the following short exact sequence
\begin{equation}
\label{DefTNSeq}
0\to \Hom(B,A) \xrightarrow{(u_1, u_2)} \Hom(\zz,C_{K,S}) \times \Hom(\zz[S_K],\mathbb{A}^\times_{K,S})\xrightarrow{(b,1)-(1,a)} \Hom(\zz[S_K], C_{K,S}) \to 0
\end{equation}
and corresponding long exact sequence
\begin{equation}
\label{TateSES}
0 \to H^2(\Hom(B,A)) \xrightarrow{(u_1, u_2)}  H^2(C_{K,S}) \times H^2(\Hom(\zz[S_K],\mathbb{A}^\times_{K,S})) \xrightarrow{} H^2(\Hom(\zz[S_K], C_{K,S}))
\end{equation}
Note that we used the fact that $H^1(\Hom(\zz[S_K], C_{K,S}))$ vanishes, which follows from Shapiro's lemma and global class field theory. Moreover this establishes the uniqueness of canonical class in $H^2(\Gamma_{K/F},\Hom(B,A))$. Tate defines this canonical in several steps.

Firstly, $\alpha_1(K,S) \in H^2(\Gamma_{K/F}, \Hom(\zz,C_{K,S})) \simeq H^2(\Gamma_{K/F},C_{K,S})$ is given by the canonical class of global class field theory.

The class $\alpha_2=\alpha_2(K,S)$ is given as follows. Note that there is a local description of the group $H^2(\Gamma_{K/F}, \Hom(\zz[S_K],\mathbb{A}^\times_{K,S}))$, namely, there is an isomorphism
\begin{equation}
\label{IsomHJK}
j_P\circ \res : H^2(\Gamma_{K/F}, \Hom(\zz[S_K],\mathbb{A}^\times_{K,S})) \simeq \bigoplus_{w\in \dot{S}_K} H^2(\Gamma_{K/K^w}, \mathbb{A}^\times_{K,S}).
\end{equation}
Thus, $\alpha_2$ can be defined by its local components. More precisely, Tate takes $\alpha_2(w)$ to be the image of the local fundamental class $\alpha(K_w/F_u) \in H^2(\Gamma_{K/K^w}, K_w^\times)$ under the map $i_w: K_w^\times \hookrightarrow \mathbb{A}^\times_{K,S}$. The pair $(\alpha_1,\alpha_2)$ maps to $0$ in $H^2(\Hom(\zz[S_K],C_{K,S}))$. Consequently, $(\alpha_1,\alpha_2)$ lies in the image of inclusion in (\ref{DefTNSeq}) and its projection to $H^r(\Gamma_{K/F}, \Hom(\zz[S_K]_0, \OO_{K,S}))$ defines the canonical class $\alpha_3^r = \alpha_3^r(K,S)$.

\begin{definition}
\label{KottTorus}
The canonical class $\alpha_{K,S} = \alpha_3^2(K,S) \in H^2(\Gamma_{K/F}, \mathbb{D}_{K,S}(K))$ is called the Kottwitz class corresponding to the triple $(F,K,S)$ where $F$, $K$, and $S$ are as above. If $S=V_K$ we denote the corresponding class by $\alpha_{K/F}$.
\end{definition}
\subsection{Global pro-torus $\mathbb{D}_F$}
\label{DefOfDF}
Let $L\supset K \supset F$ be a tower of finite Galois extensions.
We define a map of $\Gamma_{L/F}$-modules $p_{L/K}:\mathbb{Z}[V_K]_0\to \mathbb{Z}[V_L]_0$ by restricting the map
\begin{equation}
\zz[V_K] \to \zz[V_L]: v \in V_K \mapsto \sum_{w|v} [L_w:K_v] w,
\end{equation}
to the kernels of the corresponding degree maps (see Section \ref{SemiLocal}.) Analogously, one defines $\Gamma_{L/F}$-homomorphisms $p_{L/K,S^\prime,S}:\mathbb{Z}[S_K]_0\to\mathbb{Z}[S^\prime_L]_0$, where $S\subseteq S^\prime$ are sets of places in $F$ (see notation in Section \ref{TateKalethaKottwitzCond}.)

Using $p_{L/K}$ as the transition maps, we define $X_F=\varinjlim_{K}\mathbb{Z}[V_K]_0$, where the colimit is being taken over the directed set of finite Galois extensions $K/F$ contained in a fixed separable closure $F^{\sep}$ of $F$. Then the pro-torus over $F$ having $X_F$ as the module of characters is denoted by $\mathbb{D}_F$.

\subsection{Vanishing theorems}
The following lemma is a consequence of Lemmata 3.1.9 and 3.1.10 in \cite{Kal}.
\begin{lemma}
\label{InjKaletha}
Assume (T1-2) and (K1-2), then we have the following vanishing results:
\begin{itemize}
\item The inflation map
\begin{equation}
H^i(\Gamma_{K/F}, \Hom(\zz[S_K]_0,\OO_{K,S}^\times))\to H^i(\Gamma, \Hom(\zz[S_K]_0, \bar{F}^\times)) = H^i(F, \mathbb{D}_{K,S})
\end{equation}
is injective for $i=1, 2$.
\item Moreover, $H^1(F,\mathbb{D}_{K,S}) = 0$ and $H^2(F_{\fpqc},\mathbb{D}_F) = \varprojlim_{K,S} H^2(F,\mathbb{D}_{K,S})$.
\end{itemize}
\end{lemma}
In order to prove the latter statement, one uses Grothendieck spectral sequence for the composition of derived limit and derived functor of global sections (see more in Section \ref{CohomologicalGeneralities}).

\subsection{Kottwitz gerbes on finite levels}
 \label{KTClass}
 Let us consider the canonical class $\alpha_{K/F}$. It belongs to $H^2(\Gamma_{K/F}, \mathbb{D}_{K/F}(K))$, where $\mathbb{D}_{K/F}$ is the pro-torus over $F$ with characters $\zz[V_K]_0$, note that $\mathbb{D}_{K/F}$ is split over $K$. The Kottwitz gerbe for $K/F$ is defined as an extension
\begin{equation}
1\to \mathbb{D}_{K/F}(K) \to \Kt_{K/F} \to \Gamma_{K/F} \to 1
\end{equation}
corresponding to the class of $\alpha_{K/F}$. It follows from Lemma \ref{InjKaletha} that the extension $\Kt_{K/F}$ is defined up to conjugation by an element of $\mathbb{D}_{K/F}(K)$.

\subsection{An $F_{\fpqc}$-gerbe attached to the Kottwitz class} Using the equality of cohomology classes  $p_{L/K,S^\prime,S}^\ast(\alpha_{L,S^\prime}) = \alpha_{K,S}^{\infl}\in H^2(\Gal(L/F),\mathbb{D}_{K,S}(L))$ proved by Kottwitz \cite[Lemma 8.3]{Kot} and
passing to the limit in Lemma \ref{InjKaletha}, we obtain an injective map
\begin{equation}
\varprojlim H^2(\Gamma_{K/F}, \Hom(\zz[S_K]_0,\OO_{K,S}^\times)) \to  H^2(F, \mathbb{D}_F),
\end{equation}
which shows that the Kottwitz class defined as a projective limit of the Tate-Nakayama classes $\alpha_{K,S}$ gives rise to a well-defined class in $H^2(F_{\fpqc},\mathbb{D}_F)$, that can be represented by an $\fpqc$-gerbe which we denote by $\Kt_F$. Note that the gerbe $\Kt_F$ is well-defined only up to a non-unique isomorphism, since $H^1(F_{\fpqc}, \mathbb{D}_F)\neq 0$. However, equivalent gerbes in the class of $\Kt_F$ give rise to equivalent Tannakian categories $\Rep(\Kt_F)$.

\section{The category of representations of the Kottwitz gerbes}
\subsection{Abstract definition}
\label{AbstrReps}
Let $S=\Spec(F)$, where $F$ is a field, and let $\LOCLIB(S)$ denote the stack of locally free sheaves of finite rank over $S_{\fpqc}$, i.e. $\LOCLIB(S)_U =\LocLib_R$ for $R = \Gamma(U, \mathcal{O}_U)$, if $U$ is affine.
If $\mathcal{E}$ is an affine gerbe over $S_{\fpqc}$, we define $\REP(\mathcal{E})$ as a stack of cartesian functors
\begin{equation}
\mathcal{E} \to \LOCLIB(S),
\end{equation}
we define $\Rep_S(\mathcal{E})$ or simply $\Rep(\mathcal{E})$ as a groupoid $\REP(\mathcal{E})_S$.

If $\mathcal{E}$ is a neutral gerbe $\TORS(S, G)$ where $G$ is an affine group, we have an equivalence of categories
\begin{equation}
\Rep(\mathcal{E}) \simeq \Rep(G)
\end{equation}
which is given by sending every cartesian functor $\Phi: \mathcal{E}_U \to \LOCLIB(S)_U$ to $\Phi(G_d)$, where $G_d$ is a trivial $G$-torsor. Since $\Aut(G_d)\simeq G$, we see that $G$ acts on the locally free sheaf of finite rank $\Phi(G_d)$ by functoriality of $\Phi$.

\subsection{Representations of a bitorsor attached to a gerbe}
\label{RepsBitor}
Here we recall some results of Deligne contained in \cite{Del}.
As described in Section \ref{LocDescrG}, any affine gerbe $\mathcal{E}$ over $S_{\fpqc}$ with a cartesian section $x\in \mathcal{E}_U$, where $U$ is an affine scheme over $S$, gives rise to a bitorsor $(s,t):E=\Isom(p_1^\ast(x),p_2^\ast(x)) \to U^2$ endowed with a cocycle isomorphism $\psi$ (see (\ref{CocyclePsi})), such that $\psi$ satisfies the coherence condition over $U^4$. Moreover, the pullback of $E$ along the diagonal $\Delta: U \to U^2$ is given by the sheaf of groups $G= \Aut_U(x)$. In terminology of \cite{Breen} and \cite{Ulbrich}, the triple $(G, E, \psi)$ is called a bitorsor cocycle.

Let $V$ be a locally free sheaf of finite rank over $U$. We define the bitorsor $\mathcal{G}_V$ as a scheme over $U^2$ representing the functor that sends any affine scheme $(f_1,f_2): T\to U^2$ to the set $\Isom_{T}(f_1^\ast V, f_2^\ast V)$. Since $V$ is a sheaf over $U$, its natural descent data defines a cocycle isomorphism $\psi_V$ of the form (\ref{CocyclePsi}) satisfying the coherence condition over $U^4$. Note that $\Delta^\ast\mathcal{G}_V = \GL_V$.

\begin{definition}A representation of $(G, E,\psi)$ is a pair $(V,\rho)$, where $V$ is a locally free sheaf of finite rank on $U$ and $\rho$ is a $U^2$-morphism of bitorsors
\begin{equation}
\label{RhoMM}
\rho:E \to \mathcal{G}_V
\end{equation}
such that the restriction of $\rho$ to the diagonal induces a morphism of group schemes $\tilde{\rho}: G \to \GL(V)$ over $U$ and $\rho$ is compatible with the cocycle isomorphisms $\psi$ and $\psi_V$ over $U^3$ and corresponding coherence conditions over $U^4$. A morphism between representations $(V,\rho)$ and $(V^\prime,\rho^\prime)$ is a morphism of sheaves $f: V \to V^\prime$ such that for any affine $U^\prime\xrightarrow{q} U$ and any section $g\in E(U^\prime)$ the following diagram
\begin{equation}
\begin{tikzcd}
(p_1 q)^\ast V \arrow{r}{\rho(g)} \arrow{d}[swap]{(p_1 q)^\ast f} & (p_2 q)^\ast V \arrow{d}{(p_2 q)^\ast f}\\
(p_1 q)^\ast V^\prime \arrow{r}{\rho^\prime(g)} & (p_2 q)^\ast V^\prime
\end{tikzcd}
\end{equation}
is commutative.
\end{definition}
\begin{comment}
for any $S$-scheme $T$ and any $g\in E(T)$, there is an isomorphism of quasi-coherent sheaves over $T$
\begin{equation}
\rho(g): (s\circ g)^\ast V \simeq (t\circ g)^\ast V,
\end{equation}
such that $\rho(\psi(g,h)) = \rho(g) \rho(h)$, whenever $t\circ g = s\circ h$, and such that $\rho(g)$ is the identity automorphism of $s^\ast V$, when $g$ is the identity automorphism $\epsilon\circ s$ of the section $s\in U(T)$. Here, $\epsilon\in E(U)$ is given by the composition of the identity section $U\to \Delta^\ast E$ with the map $\Delta^\ast E\to E$. We also require that $\rho$ is compatible with base changes $T^\prime \to T$.
\end{comment}

We let $\Rep(U:E)$ denote the category of representations of a bitorsor cocycle $(G, E, \psi)$. Its tensor structure is given by $(V,\rho)\otimes (V^\prime,\rho^\prime) = (V\otimes V^\prime, \rho\otimes\rho^\prime)$, where $\rho\otimes\rho^\prime$ denotes the morphism $E \to \mathcal{G}_{V\otimes V^\prime}$ induced by $\rho$ and $\rho^\prime$. The internal hom-functor $\HomI((V,\rho),(V^\prime,\rho^\prime))$ is given by the sheaf $\Hom(V,V^\prime)$ with an action of $E$ defined for any $U^\prime\xrightarrow{q} U$ and any $g\in E(U^\prime)$  by sending $(p_1 q)^\ast f\in (p_1 q)^\ast \Hom(V,V^\prime)$ to $\rho^\prime(g)\circ (p_1 q)^\ast f\circ \rho^\prime(g)^{-1}\in  (p_2 q)^\ast \Hom(V,V^\prime)$.   The tensor structure and internal hom-functor that we described turn $\Rep(U:E)$ into a Tannakian category.
The following result gives a more explicit description of the tannakian category $\Rep(\mathcal{E})$.
\begin{propos}[Deligne, \cite{Del}]
Let $\mathcal{E}$ be an affine gerbe over $\Spec(F)_{\fpqc}$, where  $F$ is field. Let $E$ be the bitorsor attached to $\mathcal{E}$ and its local section over an affine scheme $U$. Then the category $\Rep(U:E)$ defined above is equivalent to $\Rep(\mathcal{E})$.
\end{propos}

Applying this proposition to the case when $\mathcal{E}$ is a smooth algebraic gerbe having a section over a finite Galois extension $K/F$ (see Theorem \ref{ComparisonNonAb}), we get a morphism between the corresponding extensions of $\Gamma_{K/F}$ by the $K$-points of sheaves that locally represent the respective bands (see Section \ref{PassingToGalois}). Namely, a representation $(V,\rho)$, where $V$ is a finite-dimensional vector space over $K$ and $\rho$ is a morphism $E\to\mathcal{G}_V$ as in (\ref{RhoMM}) gives rise to the following diagram
\begin{equation}
\label{diagramFunctor}
\begin{tikzcd}
1 \arrow{r} & G(K) \arrow{r} \arrow{d}{\widetilde{\nu}} & E(K) \arrow{r}\arrow{d}{\rho(K)} & \Gamma_{K/F} \arrow{r} \arrow[equal]{d}& 1\\
1 \arrow{r} & \GL_V(K) \arrow{r}          & \mathcal{G}_V(K) \arrow{r}   & \Gamma_{K/F} \arrow{r}& 1
\end{tikzcd}
\end{equation}
where $\rho(K)$ is a homomorphism and $\widetilde{\nu}$ is induced by an algebraic morphism $\nu: G \to \GL_V$ over $K$.

\subsection{Explicit definition}
\label{ExplDefRep}
Let $K/F$ be a finite Galois extension.  We define the category of representations of a finite level Kottwitz gerbe $\Kt_{K/F}$ (see Section \ref{KTClass}) as the category of pairs $(V,\rho)$, where $V$ is a finite-dimensional vector space over $K$ and $\rho$ is a homomorphism $\Kt_{K/F} \to \mathcal{G}_V(K)=\GL_V(K)\rtimes\Gamma_{K/F}$, such that
\begin{itemize}
\item $\rho$ induces the identity on $\Gamma_{K/F}$.
\item the restriction of $\rho$ to the kernels is induced by an algebraic morphism of $K$-group schemes $\nu: \mathbb{D}_{K/F} \to \GL_V$.
\end{itemize}
A morphism between two representations $(V,\rho)$ and $(V^\prime,\rho^\prime)$ is defined as a morphism of $K$-vector spaces $V\to V^\prime$ such that for any $g\in \Kt_{K/F}$ the following diagram
\begin{equation}
\begin{tikzcd}
p_1^\ast V \arrow{r}{\rho(g)} \arrow{d}[swap]{p_1^\ast f} & p_2^\ast V \arrow{d}{p_2^\ast f}\\
p_1^\ast V^\prime \arrow{r}{\rho^\prime(g)} & p_2^\ast V^\prime
\end{tikzcd}
\end{equation}
is commutative. We denote this category $\Rep_{\Gamma}(\Kt_{K/F})$ in order to distinguish it from the category $\Rep(\Kt_{K/F})$ defined in Section \ref{AbstrReps}. Endowing $\Rep_{\Gamma}(\Kt_{K/F})$ with the tensor structure and internal hom-functor defined as in Section \ref{RepsBitor}, we see that it is a Tannakian category over $F$ neutralized by a finite Galois extension $K/F$.

Remark that if $F$ is a global field, then we could analogously define $\Rep_{\Gamma}(\Kt_{K,S})$ for any finite set of places $S$ satisfying the conditions of Section \ref{TateKalethaKottwitzCond}. It follows from the definitions of $\alpha_{K,S}$ and $\alpha_{K/F}$ that $\Rep_{\Gamma}(\Kt_{K/F}) \simeq \varinjlim_{S} \Rep_{\Gamma}(\Kt_{K,S})$.

Let $F\subset K \subset L$, where $L$ and $K$ are finite Galois extensions. Following Kottwitz one can also define $L/F$-Galois gerbes $\Kt_{K/F}^{\infl}$ that correspond to inflated canonical classes $\infl(\alpha_{K/F}) \in H^2(\Gal(L/F), \mathbb{D}_{K/F}(L))$. It follows by descent that $\Rep_{\Gamma}(\Kt_{K/F})\simeq \Rep_{\Gamma}(\Kt_{K/F}^{\infl})$. It is shown in Section 8 in \cite{Kot} that there exists a homomorphism $\Kt_{L/F}\to \Kt_{K/F}^{\infl}$ that gives rise to the functor $\Rep_{\Gamma}(\Kt_{K/F}^{\infl})\to \Rep_{\Gamma}(\Kt_{L/F})$. Consequently, we obtain a functor
\begin{equation}
\label{RepsTransition}
\Rep_{\Gamma}(\Kt_{K/F}) \to \Rep_{\Gamma}(\Kt_{L/F}).
\end{equation}
It follows formally that (\ref{RepsTransition}) is fully faithful, hence we define
\begin{equation}
\Rep_{\Gamma}(\Kt_F) = \varinjlim_{K/F\text{ fin. Galois}} \Rep_{\Gamma}(\Kt_{K/F}).
\end{equation}
\begin{propos}
\label{RepsEqui}There is an equivalence of Tannakian categories
$\Rep_{\Gamma}(\Kt_F)\approx\Rep(\Kt_F)$.
\end{propos}
\begin{proof} Note that any representation of $\Kt_F$ in the stack of locally free sheaves of finite rank factors through a representation of an algebraic gerbe $\Kt_{K,S}$ for sufficiently large $K$ and $S$ satisfying the conditions of Section \ref{TateKalethaKottwitzCond}. Therefore, we can rewrite $\Rep(\Kt_F)$ as $\varinjlim_{K,S}\Rep(\Kt_{K,S})$.

Note that the class of $\Kt_{K,S}$ lies in $H^2(\Gal(K/F),\mathbb{D}_{K,S})$ and therefore it is neutralized by the finite Galois extension $K/F$.
As explained in Section \ref{RepsBitor}, we have a functor $\theta^{K,S}: \Rep(\Kt_{K,S}) \to \Rep_{\Gamma}(\Kt_{K,S})$ that sends a representation $(V,\rho)$ of the bitorsor cocycle attached to $\Kt_{K,S}$ to the pair $(V,\rho(K))$ (see diagram (\ref{diagramFunctor}). We claim that $\theta^{K,S}$ is an equivalence of Tannakian categories. Indeed, $\theta^{K,S}$ induces a morphism between corresponding bands $\theta^{K,S}_\ast: \mathbb{D}_{K,S} \to \mathbb{D}_{K,S}$ such that, for any finite Galois extension $L/F$ containing $K$, the homomorphism $\theta^{K,S}_\ast(L): \mathbb{D}_{K,S}(L) \to \mathbb{D}_{K,S}(L)$ is the identity homomorphism of groups of points. Passing to the limit of $\theta^{K,S}_\ast$ and noticing that the transition functors defining $\Rep(\Kt_{F})$ and $\Rep_{\Gamma}(\Kt_F)$ are compatible, we obtain the desired equivalence.
\end{proof}
\subsection{Algebraic $1$-cocycles and morphisms of gerbes}
\label{KtBGandMorphisms}
Let $K$ be a finite Galois extension of $F$, let $G$ be a linear algebraic group over $F$, and let $\mathbb{D}_{K/F}$ be the band of the gerbe $\Kt_{K/F}$. In Section 2.3 of \cite{Kot}, Kottwitz defined the set $Z^1(\Kt_{K/F}, G(K))$ as a set of pairs $(x,\nu)$, where $x$ is an abstract $1$-cocycle of $\Kt_{K/F}$ with values in $G(K)$, i.e. $x$ is a map $w\mapsto x_w$ satisfying
\begin{equation}
\label{cocyclecond}
x_{w_1 w_2} = x_{w_1} w_1(x_{w_2}),
\end{equation}
where $\Kt_{K/F}$ acts on $G(K)$ via projection to $\Gamma_{K/F}$, and $\nu: \mathbb{D}_{K/F}(K) \to G(K)$ is induced by an algebraic morphism $\mathbb{D}_{K/F}\to G$ over $K$ such that $x_d = \nu(d)$ for any $d\in \mathbb{D}_{K/F}$.

Note that $\Gamma_{K/F}$ acts on the set of homomorphisms $\mathbb{D}_{K/F}(K)\to G(K)$ via the formula $\sigma(\nu)(d) = \sigma(\nu(\sigma^{-1}(d)))$. By $\Int(x)$ for $x\in G(K)$ we denote the inner automorphism defined by $g\mapsto x g x^{-1}$.

The cocycle condition (\ref{cocyclecond}) then implies that $\Int(x_w) \circ \sigma(\nu) = \nu$ whenever $w$ maps to $\sigma$.

There is an obvious action of $G(K)$ on the set of algebraic $1$-cocycles: $g\in G(K)$ transforms $(x,\nu)$ into $(w\mapsto g x_w w(g)^{-1}, \Int(g) \circ \nu)$. Then a pointed set $H^1_{\alg}(\Kt_{K/F}, G(K))$ is defined as the quotient of $Z^1(\Kt_{K/F},G(K))$ by the action of $G(K)$.

 Let $\mathcal{G}$ be a neutral gerbe $\TORS(S, G)$. The category of representations of $\mathcal{G}$ is equivalent to $\Rep_F(G)$. Moreover, for any finite Galois extension $K/F$ the gerbe $\mathcal{G}$ can be represented by a $K/F$-Galois gerbe given by the split extension
\begin{equation}
\xymatrix{
  1 \ar[r]^{} & G(K) \ar[r]^{} & \mathcal{G}_{K/F} \ar[r]^{} & \Gamma_{K/F} \ar[r]^{} & 1   }
\end{equation}
that is $\mathcal{G}_{K/F} \simeq G(K) \rtimes \Gamma_{K/F}$ with the natural action of $\Gamma_{K/F}$ on the set of $K$-points of $G$.

Recall the product formula for $G(K) \rtimes \Gamma_{K/F}$
\begin{equation}
(g, \sigma) \bullet (h, \tau) = (g \sigma(h), \sigma \tau),
\end{equation}
which also gives the formula for the action of $G(K)$ by conjugation, namely
\begin{equation}
\label{Conjugation}
\Int(g,1) \circ (h, \tau) = (g h \tau(g)^{-1}, \tau).
\end{equation}

We deduce from (\ref{cocyclecond}) and the definition of $\Kt_{K/F}$ that an algebraic $1$-cocycle $(x,\nu)$ defines a homomorphism of extensions
\begin{equation}
\label{HomsKt}
\begin{tikzcd}
1\ar{r} & \mathbb{D}_{K/F}(K) \ar{r} \ar{d}{\nu} & \Kt_{K/F} \ar{d}{x} \ar{r} & \Gamma_{K/F} \ar{r} \ar[equal]{d} & 1\\
1\ar{r} & G(K) \ar{r} & \mathcal{G}_{K/F} \ar{r} & \Gamma_{K/F} \ar{r} & 1
\end{tikzcd}
\end{equation}
that agrees with the situation described in Section \ref{PassingToGalois}.

\begin{comment}Therefore an algebraic $1$-cocycle in the sense of Kottwitz defines a $K$-homomorphism of the corresponding $K/F$-Galois gerbes, and thus a morphism between $\Kt_{K/F}$ and $\mathcal{G}$. Additionally, the formula (\ref{Conjugation}) shows that if $x, x^\prime \in Z^1_{\alg}(\Kt_{K/F}, G(K))$ become equivalent in $H^1_{\alg}(\Kt_{K/F}, G(K))$, then the morphisms between gerbes defined by $x$ and $x^\prime$ are connected by a $2$-morhism, and thus define the same element in the pointed set $\Hom(\Kt_{K/F}, \mathcal{G})/\Hom^2(\Kt_{K/F}, \mathcal{G})$.
\end{comment}

\section{Representations of the Kottwitz gerbes with $G$-structure}
\label{GStructure}
Here we compare the Kottwitz set $B(F,G)$ and the set of isomorphism classes of tensor functors $\Rep(G) \to \Rep(\Kt_F)$. Namely, we prove the following
\begin{propos}
\label{RepsGStruct}
For any linear algebraic group $G$ over $F$, the set of isomorphism classes of tensor functors $\Rep(G)\to\Rep(\Kt_F)$ can be identified with the pointed set $B(F,G)$ defined by Kottwitz.
\end{propos}

We use Tannakian duality to express everything in terms of gerbes. Namely, let $\mathcal{G}$ be a neutral gerbe attached to $G$, then the category $\Hom^{\otimes}_F(\Rep(G), \Rep(\Kt_F))$ is equivalent to the category $\Hom_{F_{\fpqc}}(\Kt_F, \mathcal{G})$ consisting of cartesian functors between the corresponding gerbes over $F_{\fpqc}$. The heart of the comparison lies in the analysis of the gerbes $\HOM_{\nu}(\Kt_{F},\mathcal{G})$ introduced in the following section.

\subsection{Definition and main properties of the gerbe $\HOM_{\nu}(\mathcal{E}, \mathcal{H})$}
Here we work in a more abstract setting. Let $\mathcal{E}$ and $\mathcal{H}$ be algebraic gerbes over $S$ having $L$ and $M$ as their bands respectively. Here $S$ is the site of affine schemes over $F$, where $F$ is field, endowed with the $\fl$ topology.

\begin{comment}Recall that the stack of bands over $S$ is defined as a stack associated to the  follows. The objects are the sheaves of groups and a morphism between two bands $L$ and $M$ is given by the global sections of the quotient sheaf $\Hom(L,M)/ M$ with $M$ acting by inner automorphisms.

To any gerbe $\mathcal{E}$ one can canonically attach a band $L$, namely
\begin{equation}
\label{LocalBand}
\restr{L}{U} \MapsTo \band(\restr{\Aut(x)}{U}),\text{\quad where } x\in \Ob(\restr{\mathcal{E}}{U})
\end{equation}
for any cover $U$ in $S$.

A morphism of gerbes $f: \mathcal{E} \to \mathcal{H}$ is banded by a morphism of bands $\nu: L \to M$ if the following diagram commutes
\[
\begin{tikzcd}
\band(\restr{\Aut(x)}{U}) \arrow{r}{f}  & \band(\restr{\Aut(f(x))}{U})  \\
\restr{L}{U} \arrow{r}{\nu} \arrow{u} & \restr{M}{U} \arrow{u}
\end{tikzcd}
\]
where the vertical arrows are the isomorphisms from (\ref{LocalBand}). Furthermore, any morphism of gerbes is banded by a unique morphism of bands.
\end{comment}
Let $\nu: L \to M$ be a morphism of bands. The category $\Hom_\nu(\mathcal{E},\mathcal{H})$ is defined as a full subcategory of $\Hom(\mathcal{E},\mathcal{H})$ whose objects are $\nu$-morphisms of corresponding gerbes. It is also useful to define a fibred category $\HOM_\nu(\mathcal{E},\mathcal{H})$ whose objects over $U$ are the morphisms of gerbes
\begin{equation}
\restr{\mathcal{E}}{U} \to \restr{\mathcal{H}}{U}
\end{equation}
that are banded by $\restr{\nu}{U}$.

Let us recall the main properties of $\HOM_\nu(\mathcal{E},\mathcal{H})$:

\begin{itemize}
\label{HOMnuProps}
\item  $\HOM_\nu(\mathcal{E},\mathcal{H})$ is a gerbe over $S$.
\item  The band of $\HOM_\nu(\mathcal{E},\mathcal{H})$ is identified with $C_\nu$, the centralizer of $\nu$.
\item  If $L=M$ and $u = \id_L$, then $\HOM_\nu(\mathcal{E},\mathcal{H})$ is banded by the centre of $L$.
\item  There exists a morphism $\mathcal{E} \to \mathcal{H}$ banded by $\nu$  if and only if $\HOM_\nu(\mathcal{E},\mathcal{H})$ is a neutral gerbe.
\end{itemize}

Note that $\Hom_{\nu}(\mathcal{E},\mathcal{H})$ is given by the groupoid $\HOM_{\nu}(\mathcal{E},\mathcal{H})_{F}$.

For more details see IV.2 in \cite{Giraud}.

\subsection{Basic example} Consider the gerbe $\HOM_{1}(\Kt_F, \mathcal{G})$ that is fibered in groupoids of the morphisms between $\Kt_F\times U$ and $\mathcal{G}\times U$ banded by the trivial morphism of bands over $U$, in this situation we have an equivalence of gerbes
\begin{equation}
\label{BasicExEq}
\HOM_{1}(\Kt_F, \mathcal{G}) \simeq \mathcal{G},
\end{equation}
and the latter is equivalent to the stack $\TORS(S,G)$. Passing to the fibre over $F$, we obtain the groupoid $\Tors(S,G)$ of $G$-torsors over $S$. Noticing that under the equivalence (\ref{BasicExEq}) $2$-morphisms between $\Kt_F$ and $\mathcal{G}$ banded by the trivial morphism correspond to morphisms between $G$-torsors, we obtain an isomorphism of pointed sets
\begin{equation}
\Hom_1(\Kt_F,\mathcal{G})/_{\sim} \simeq H^1(F, G).
\end{equation}
\subsection{Reduction steps}
Recall that every morphism of gerbes is banded by a unique morphism of bands, and that every morphism of bands is $\fpqc$-locally representable by a morphism of sheaves, thus we can rewrite $\Hom_{F_{\fpqc}}(\Kt_F, \mathcal{G})$ as a disjoint union of $\Hom_\nu(\Kt_F,\mathcal{G})$ over the set of representatives of the morphisms of bands over $F^{\alg}$
\begin{equation}
\label{Stratification}
\Hom(\Kt_F,\mathcal{G}) =\bigsqcup_{\nu:\ \mathbb{D}_F \to G \text{ over } F^{\alg}}  \Hom_\nu(\Kt_F,\mathcal{G})
\end{equation}

\begin{comment}
\begin{remark}
Notice that we are forced to consider morphisms $\mathbb{D}_F \to G$ over $F^{\alg}$, since $H^1(F^{\sep}_{\fpqc},\mathbb{D}_F)$ is given by the derived limit $\varprojlim^{1}_{K,S} \mathbb{D}_{K,S}(F^{\sep})$ and is not trivial, therefore, a description of the groupoid $\HOM(\Kt_F,\TORS(G))_{F^{\sep}}$ requires considering all the inner forms of $\mathbb{D}_F$.
\end{remark}
\end{comment}

However, any $\nu : \mathbb{D}_{F} \to G$ factors through an algebraic quotient of the protorus $\mathbb{D}_{F}$, that is through some $\mathbb{D}_{K,S}$ where $K=K(\nu)$ is finite Galois, and $S=S(\nu)$ is a finite set of places that can be taken sufficiently large (so that $H^1(F, \mathbb{D}_{K,S}) = 0$). Using ``surjective functoriality'' of morphisms of bands (see \cite{Giraud} Proposition IV.2.3.18), we obtain $\Hom_\nu(\Kt_F,\mathcal{G}) \simeq \Hom_\nu(\Kt_{K,S},\mathcal{G})$. Thus, we reduced the computation of $\Hom(\Kt_F,\mathcal{G})$ to the question about morphisms between algebraic gerbes $\Kt_{K,S}$ and $\mathcal{G}$, although we note that we are still working over $F_{\fpqc}$.

Consider the gerbe $\mathcal{E} = \HOM_{\nu}(\Kt_{K,S},\mathcal{G})$. As mentioned in Section  \ref{HOMnuProps}, this gerbe is banded by the centralizer of $\nu$ in $G$, which we denote by $C_{\nu}$. It is representable by an algebraic group scheme over a finite extension of $F$, because $\nu: \mathbb{D}_{K,S}\to G$ is a morphism between algebraic groups over $F^{\alg}$. Since $\Kt_{K,S}$ is banded by the torus $\mathbb{D}_{K,S}$ and $G$ is a smooth group, we see that $C_{\nu}$ is smooth (see Corollaire 2.4 Exp.\,XI in \cite{SGA3}). Thus we are in the situation of the comparison isomorphism of Proposition \ref{ComparisonNonAb}, which shows that $\mathcal{E}$ is a gerbe over $F_{\sm}$. Using this we see that there exists a finite separable extension of $F$ over which $\mathcal{E}$ has a section and any two \'{e}tale-local sections are \'{e}tale-locally isomorphic.

Let $f: \Kt_{K,S} \to \mathcal{G}$ be a $\nu$-morphism, then as described in Section \ref{PassingToGalois}, for a sufficiently large finite Galois extension $L/F$, we obtain an algebraic $1$-cocycle in $Z^1_{\nu}(\Kt_{K,S}(L), \mathcal{G}(L))$.
\begin{propos}
\label{GerbesGaloisGerbes}
The map from $\Hom_\nu(\Kt_{K,S}, \mathcal{G})$ to $\varinjlim_{L/F} Z^1_{\nu}(\Kt_{K,S}(L),\mathcal{G}(L))$ described above induces a bijection on the corresponding pointed sets of isomorphism classes. Moreover, any element in $\varinjlim_{L/F} Z^1_{\nu}(\Kt_{K,S}(L),\mathcal{G}(L))$ defines a $\nu$-morphism of gerbes $\Kt_{K,S} \to \mathcal{G}$.
\end{propos}
\begin{proof}
As follows from Section \ref{HOMnuProps}, the sheaf $J_\nu = \Aut(f)$ over $F_{\sm}$ is an $F$-form of $C_{\nu}$ and the groupoid $\Hom_\nu(\Kt_{K,S}, \mathcal{G})$ is then identified with the groupoid of $J_\nu$-torsors over $F_{\sm}$. Therefore, the pointed set of isomorphism classes of $\nu$-morphisms between the corresponding gerbes is given by the pointed set $H^1(F_{\sm},J_\nu)$.

On the other hand, the image of $f$ in $\varinjlim_{L/F} Z^1_{\nu}(\Kt_{K,S}(L),\mathcal{G}(L))$ and the compatibility of descent data on $C_v$ defined by $f$ and its image in the Kottwtiz set, allow us to identify $B_\nu(F,G)$ with the pointed set $H^1(\Gal(F^{\sep}/F), J_{\nu})$ (see Section 3.5 in \cite{Kot97}). Since $J_\nu$ is smooth, the latter is isomorphic to $H^1(F_{\sm},J_\nu)$.

In order to show that both sets of isomorphism classes of objects in $\Hom_\nu(\Kt_{K,S},\mathcal{G})$ and $B_\nu(F,G)$ are simultaneously non-empty, we use Tannakian duality. Namely, it follows from Section \ref{KtBGandMorphisms} that any element in $Z^1_\nu(\Kt_{K,S}(L), \mathcal{G}(L))$ induces a functor between the categories of representations of $L/F$-Galois gerbes $\mathcal{G}(L)$ and $\Kt_{K,S}(L)$ defined in Section \ref{ExplDefRep}. By Proposition \ref{RepsEqui}, the latter categories are equivalent to $\Rep(\mathcal{G})$ and $\Rep(\Kt_{K,S})$, respectively. Finally, we apply Tannakian duality to obtain a functor $\Kt_{K,S} \to \mathcal{G}$.
\end{proof}

 Consequently, $B_\nu(F,G)$ is identified with $\Hom_\nu(\Kt_{K,S},\mathcal{G})/_{\sim}$.   Combining this fact with the decomposition (\ref{Stratification}), we see that $\Hom(\Kt_{K,S},\mathcal{G})/_{\sim}$ is identified with $B(F, G)$, which finishes the proof of Proposition \ref{RepsGStruct}.

\section{Representations over non-archimedean local fields}
\label{NonArchDescr}
 We write $\Rep(\Kt_F)$ to denote the Tannakian category defined in Section \ref{ExplDefRep}.  We let $\breve{F}$ denote  the maximal unramified extension of $F$, if $F$ is of characteristic $0$, and $\breve{F} = F\otimes_{\mathbb{F}_q} \overline{\mathbb{F}}_q$, if $\Char(F)$ is a local field of characteristic $p>0$ with the residue field $\mathbb{F}_q$. Note that in both cases $\Gal(\breve{F}/F)$ is a pro-cyclic group generated by the Frobenius automorphism $\Frob_{\bar{\mathbb{F}}_q}$. In this section we prove the following theorem.

 \begin{theorem}
 \label{TheoremNonArch}
 Let $F$ be a non-archimedean local field. Then the category of representations of the Kottwitz gerbe $\Rep(\Kt_F)$ is equivalent to the category of pairs $(V,\rho)$, where $V$ is a finite-dimensional vector space over $\breve{F}$, and $\rho:V\MapsTo V$ is a $\Frob_{\bar{\mathbb{F}}_q}$-linear automorphism.
 \end{theorem}

\subsection{Reduction to the unramified case}
\label{reductionunram}
Let $K$ and $L$ be two finite extensions of $F$ such that $[K:F] =  [L:F] = n$, $L$ is Galois and $K$ is unramified. Then we have two functors
\begin{displaymath}
\xymatrix{
  \Rep(\Kt^{\infl}_{K/F})  \ar[r]^-{\Theta_{K,K.L}}
                & \Rep(\Kt_{K.L/F})   \\
                & \Rep(\Kt^{\infl}_{L/F})   \ar[u]^{\Theta_{L,K.L}}          },
\end{displaymath}
and it follows from local class field theory (cf. \cite{Neu1}, Theorem (5.2)) that the cohomology groups $H^2(\Gamma_{L/F}, L^\times)$ and $H^2(\Gamma_{K/F},K^\times)$ coincide as subgroups of the group $H^2(\Gamma_{K.L/F}, (K.L)^\times)$. More accurately, the inclusions
\begin{equation}
H^2(\Gamma_{L/F}, L^\times), H^2(\Gamma_{K/F}, K^\times) \hookrightarrow H^2(\Gamma_{K.L/F}, (K.L)^\times))
\end{equation}
are defined via inflations of the extensions classified by these cohomology groups. Moreover, inflated extensions corresponding to the fundamental classes of $K$ and $L$ are equivalent.
 It follows that in the limit  $\Rep(\Kt_{F})$ every representation $\rho: \Kt_{L/F} \to \mathcal{G}_V$ becomes equivalent to a representation of the Kottwitz gerbe of an unramified extension.

Thus it remains to describe $\Rep(\Kt_{K/F})$ where $K$ is an unramified extension.

\subsection{Unramified case}
Let $K/F$ be a finite unramified extension of degree $n$. Consider a representation of the Kottwitz gerbe of $K/F$, i.e. $\rho: \Kt_{K/F} \to \mathcal{G}_V$, where $V$ is a finite dimensional vector space over $K$, and $\mathcal{G}_V$ is a $K/F$-Galois gerbe defined in Section \ref{ExplDefRep}.

The Galois group of $K$ is naturally isomorphic to the Galois group of corresponding residue field extension, which is cyclic and generated by the Frobenius element $\sigma_K$. Consider any set-theoretic section $d^K(\sigma_K)\in \Kt_{K/F}$, after composing  $\rho$ and the projection $\mathcal{G}_V\to \Gal(K/F)$, we see that $\Phi:=\rho(d^K(\sigma_K))$ defines a $\sigma_K$-linear automorphism of $V$.

Note that the following diagram commutes
\begin{displaymath}
\xymatrix{
  1  \ar[r]^{} & \Gm(K) \ar[d]_{\nu} \ar[r]^{i} & \Kt_{K/F} \ar[d]_{\rho} \ar[r]^{p} & \langle\sigma_{K/F}\rangle \ar@{=}[d]_{} \ar[r]^{} & 1 \\
  1 \ar[r]^{} & \GL(V) \ar[r]^{i^\prime} & \mathcal{G}_V \ar[r]^{p^\prime} & \langle\sigma_{K/F}\rangle \ar[r]^{} & 1  \rlap{\ ,} }
\end{displaymath}
and the morphism $\nu$ gives us a slope decomposition $V = \bigoplus V^m$, where $V^m$ corresponds to the subspace on which $\Gm(K)$ acts through the character $\lambda \mapsto \lambda^m$.

Since $\sigma_K^n = \id$, we see that
\begin{equation}
\Phi^n(v) = \rho(x)(v),
\end{equation}
where $x= d_K^K(\sigma_K)^n\in K^\times$.

It follows from the proposition below that $\restr{\Phi^n(v)}{V^m} = \pi_F^m$.

\begin{propos}
Let $K/F$ be a finite unramified extension of local fields of degree $n$, and $\sigma_{K}\in \Gal(K/F)$ be a Frobenius automorphism. Then, for any set-theoretic section $d^K: \Gal(K/F)\to \Kt_{K/F}$, the element $x = d^K(\sigma_K)^n$ lies in $F^\times$ and its image in $F^\times/N_{K/F}(K^\times)$ can be identified with $\pi_F$, a uniformizer of $F$.
\end{propos}
\begin{proof}
To simplify notation, we let $\sigma$ denote $\sigma_K$ throughout this proof; we also set $d(\sigma) = d^K(\sigma)$.

We see that
\begin{equation}
d(\sigma) d(\sigma)^n = d(\sigma) x= d(\sigma)^n d(\sigma) = x d(\sigma),
\end{equation}
that is $x\in F^\times$.

Let $d^\prime(\sigma) = z d(\sigma)$ with $z\in K^\times$ be another section, then
\begin{multline}
x^\prime = d^\prime(\sigma)^n = (z d(\sigma))^n  = z \sigma(z) d(\sigma)d(\sigma) z d(\sigma)\ldots =\\
 z \sigma(z) \sigma^2(z)\ldots\sigma^{n-1}(z) d(\sigma)^n = x \prod_{0\leq i \leq n-1} \sigma^i(z) = x N_{K/F}(z)
\end{multline}
and we see that $x$ is well-defined modulo $N_{K/F}(K^\times)$.

It follows from the class field theory that the group $F^\times/N_{K/F}(K^\times)$ is isomorphic to the group $\pi_F^{\zz}/ \pi_F^{[K:F] \zz}$; namely, it follows from the fact that $H^q(\Gamma_{K/F}, U_K)= 1$ for every unramified extension $K/F$ and for every $q$, if $q=0$ then we obtain an isomorphism $U_F \simeq N_{K/F}(U_K)$, where $U_F$ and $U_K$ denote the groups of units in $F^\times$ and $K^\times$ resp.

On the other hand, $\Kt_{K/F}$ corresponds to the fundamental class of local class field theory, namely, to a generator of $H^2(\Gamma_{K/F}, K^\times)$; thus $x$ can be identified with $\pi_F$, the uniformizer of $F$.
\end{proof}

\subsection{Functor from $\Rep(\Kt_{K/F})$ to $\Isoc(\bar{\mathbb{F}}_p)$}
\label{FunctorToIsoc}
Let $\breve{F}$ denote the fraction field of the ring of Witt vectors over $\bar{\mathbb{F}}_p$, and let $\sigma$ denote the induced Frobenius automorphism of $\check{F}$. Then after the base change every representation $(V,\Phi)$ defines  an isocrystal over $\bar{\mathbb{F}}_p$, namely we have a functor
\begin{equation}
\label{IsocFuncKt}
\mathcal{E}^{\breve{F}}_K: \Rep(\Kt_{K/F}) \to \Isoc(\bar{\mathbb{F}}_p),\quad (V, \Phi) \mapsto (V \otimes \breve{F}, \Phi \otimes \sigma).
\end{equation}

We claim that $\mathcal{E}^{\breve{F}}_K$ is fully faithful; its faithfulness follows immediately and it remains to show that the functor is full.

General theory of isocrystals reduces the question to studying spaces of morphisms between isoclinic isocrystals of the same slope. By Dieudonn\'{e}-Manin classification, every isoclinic isocrystal over $\bar{\mathbb{F}}_p$ of slope $s/r$ decomposes as a direct sum of simple objects $E_{s/r}$, moreover, given a module defined by relation $\Phi^n - \pi^m$, it is of dimension $n$ as a vector space and it is isomorphic to the sum of $d=(m,n)$ copies of $E^{s_0/{r_0}}$, where $r_0 d = n$ and $s_0d = m$. Our goal is to show that similar fact holds true for representations of the Kottwitz gerbe for $K/F$.

Consider a representation $\rho: \Kt_{K/F} \to \mathcal{G}_V$, where $V$ is a finite-dimensional vector space over $K$. Let $V^m$ and $\Phi$ be as above, and let $v$ be a non-zero vector in $V^m$. Then $\Phi^n(v) = \pi^m v$, by $V_{m,n}$ we denote a submodule $V_m$  generated by $\Phi^i(v)$, where $i=0,\ldots,n-1$.

Note that we have the following decomposition
\begin{equation}
\label{decompospoly}
\Phi^n - \pi^m = (\Phi^{\frac{n}{d}} - \pi^{\frac{m}{d}})(\Phi^{\frac{n}{d} (d-1)} + \pi^{\frac{m}{d}}\Phi^{\frac{n}{d} (d-2)}+\ldots + \pi^{\frac{m}{d} (d-1)}) = (\Phi^{\frac{n}{d}} - \pi^{\frac{m}{d}}) \Psi ,\quad d =(m,n).
\end{equation}

\begin{propos}
\label{Decomposition}
Let $V_{m,n}$ be a module defined as above, then it decomposes as a direct sum of $d$ invariant subspaces of the operator $\Phi^{n/d}/\pi^{m/d}$, such that their images under $\mathcal{E}^{\breve{F}}_K$ are simple isocrystals of slope $m/n$.
\end{propos}
\begin{proof}
It is well known that if $d=1$, then $V_{m,n}$ is a simple object.

Suppose $d>1$. Then it follows from (\ref{decompospoly}) that $\Psi(a v)$ is invariant under the action of $\Phi^{n/d}/\pi^{m/d}$ for any $a\in K$. On the other hand, commuting $\Phi$ with coefficients, we get the following vector
\begin{equation}
\Psi(av) = \sigma_K^{\frac{n}{d} (d-1)}(a) \Phi^{\frac{n}{d} (d-1)} v + \pi^{\frac{m}{d}}\sigma_K^{\frac{n}{d} (d-2)}(a)\Phi^{\frac{n}{d} (d-2)}v+\ldots + \pi^{\frac{m}{d} (d-1)} a v.
\end{equation}

We set $r= n/d\in\zz$. Since $K$ is an unramified extension of $F$ of degree $n$ and $n = d r$, $K$ contains a subfield $K^\prime$ which is an unramified extension of $F$ of degree $r$, and $[K:K^\prime] = d$. Let $a_i$ with $i=1,\ldots,d$ be a basis of $K$ over $K^\prime$.

Consider the vectors $\Psi(a_i v)$, $i=1,\ldots,d$. It remains to prove that these vectors are linearly independent over $K$. This is equivalent to showing that the determinant
\begin{equation}
\begin{vmatrix}
a_1 & a_2 & \ldots & a_d \\
\sigma_{K^{\prime}}(a_1) &\sigma_{K^{\prime}}(a_2) & \ldots & \sigma_{K^{\prime}}(a_d)\\
\vdots & \vdots & \ddots & \vdots \\
 \sigma^{d-1}_{K^{\prime}(a_1)} & \sigma^{d-1}_{K^{\prime}}(a_d) & \ldots & \sigma^{d-1}_{K^{\prime}}(a_d)
\end{vmatrix}, \quad \text{where } \sigma_{K^\prime}  \text{ generates } \Gal(K^\prime/F),
\end{equation}
is non-zero, which is indeed the case since it is non-zero modulo maximal ideal (its power factors as a product of all linear combinations of $\overline{a}_i$ with coefficients in the residue field of $K^\prime$, see \cite{DeLu}, Example 2.4).

This shows that $V_{m,n}$ is isomorphic to a direct sum of subspaces $V^i_{m,n}$ defined as follows
\begin{equation}
V^{i}_{m,n}= \langle \Psi(a_i v), \Phi(\Psi(a_i v)), \ldots, \Phi^{n/d -1}(\Psi(a_i v)) \rangle, \quad i=1,\ldots,d
\end{equation}
and every $V^{i}_{m,n}$ is isomorphic (after base change) to a simple isocrystal $E_{s/r}$ with $s= m/d$, $r=n/d$.
\end{proof}

Since every representation of $\Kt_{K/F}$ is a direct sum of $V^{i}_{m/n}$, it is enough to show that natural injective morphism
\begin{equation}
\label{fullness}
\End_{\Rep(\Kt_{K/F})}(V^{i}_{m,n},\Phi) \to \End_{\Isoc(\bar{\mathbb{F}}_p)}(V^{i}_{m,n}\otimes \breve{F}, \Phi\otimes\sigma)
\end{equation}
is an isomorphism.
Note that both sides of (\ref{fullness}) define central division algebras over $F$. After choosing cyclic bases both algebras can be identified with a subalgebra of $r\times r$ matrices with coefficients in an unramified extension of $F$ of degree $r$, where $r=n/(m,n)$. This follows immediately from the proof of Proposition \ref{Decomposition} for the left-hand side, and it is done, for instance, in \cite{Saa}, Lemme 3.3.2.2 for the right-hand side. Thus (\ref{fullness}) is an isomorphism.

\subsection{Equivalence} Combining the reduction argument of section \ref{reductionunram} and the description of $\Rep(\Kt_{K/F})$ for $K/F$ finite and unramified (see section \ref{FunctorToIsoc}), we see that that the functors (\ref{IsocFuncKt}) define a fully faithful functor from the category of representations of the Kottwitz gerbe $\Kt_{F}$ to the category of isocrystals over $\bar{\mathbb{F}}_p$. Moreover, this functor is essentially surjective, since for every simple object $E_{s/r}$ in $\Isoc(\bar{\mathbb{F}}_p)$ there exists a representation of Kottwitz gerbe for $K/F$ with $K$ unramified of degree $r$ over $F$.

\section{Archimedean cases}
\label{SecArchCases}
\subsection{Representations over $\mathbb{C}$}
Consider the exact sequence of groups
\begin{equation}
1 \to \Gm(\CC) \to \Kt_{\CC} \to  \Gamma_{\CC/\CC}\to 1,
\end{equation}
since the Galois group of $\CC$ is trivial, we see that $\Kt_{\CC} = \Gm(\CC)$, thus its category of representations $\Rep(\Kt_{\CC})$ is equivalent to the category $\mathscr{C}$ of $\zz$-graded complex vector spaces  (see \cite[Example 2.30]{DM} or \cite[A.8.8]{CGP}  for the proof of general statement).

Note that the grading on $V= \bigoplus V^m \in Ob(\mathscr{C})$ is defined as follows:
\begin{equation}
\label{gradingdefinition}
V^m = \{v\in V \mid \rho(g) v = m(g) v \text{ for every } g\in \Gm\},
\end{equation}
where $m(g) \in \Hom(\Gm, \Gm) = \zz$ is the character $\lambda \mapsto \lambda^m$.

\subsection{The case of the Kottwitz gerbe over $\RR$}
In this case, we are given the exact sequence
\begin{equation}
\label{KottRealExt}
1 \to \Gm(\CC) \xrightarrow{i} \Kt_{\RR} \xrightarrow{p} \Gamma_{\CC/\RR} \to 1
\end{equation}
representing the non-trivial class in $H^2(\Gamma{\mathbb{C}/\mathbb{R}}, \mathbb{C}^\times)$.
\begin{theorem}
\label{KtRealRep}
The category of representations of real Kottwitz gerbe $\Rep(\Kt_\RR)$ is equivalent to the category $\mathscr{F}$ of finite dimensional $\zz$-graded complex vector with semi-linear automorphism $\alpha$, such that $\alpha^2$ acts on $V^m$ by $(-1)^m$.
\end{theorem}
\begin{proof}
Firstly, we describe $\Kt_\RR$ explicitly.

Let $\Gamma=\Gamma_{\CC/\RR}= \langle 1,\,\sigma\rangle$, where $\sigma$ is the complex conjugation. Since the kernel of extension (\ref{KottRealExt}) is $\Gm$, there is a well-defined action of $\Gamma$ on $\Gm$, namely for any set-theoretic section $\omega : \Gamma \to \Kt_{\RR}$, we have the action $i(g) \mapsto \omega(\sigma) i(g) \omega(\sigma)^{-1}= \sigma(i(g)) = i(\bar{g})$.

Since $p(\omega(\sigma)) = p(\omega(\sigma^2))$, there exists an element $x\in \CC^\times$ such that $\omega(\sigma)^2 = x$, but $x$ is known to be a $2$-cocycle of $\Gamma$ in $\Gm$, therefore the equation $\omega(\sigma)^2 \omega(\sigma) = \omega(\sigma) (\omega(\sigma))^2$ is satisfied. We deduce that $x = \bar{x}$, i.e. $x\in\RR$.

We have computed $Z^2(\Gamma, \CC^\times)$, in order to compute the coboundary we substitute $\omega(\sigma)$ with $\omega^\prime= z \omega(\sigma)$.

On the one hand $(\omega^\prime )^2 = x^\prime$, where $x^\prime \in \RR$. On the other hand, we have $(\omega^\prime )^2 = |z|^2 x$. Thus $x$ and $x^\prime$ differ by the positive real number, and we have $B^2(\Gamma, \CC^\times) = \RR_{>0}\subset \RR = Z^2(\Gamma,\CC^\times)$.
The corresponding second cohomology group can be identified with the group $\RR\slash \RR_{>0} =\{\pm1\}$, where $-1$ is the class of the nontrivial extension of $\Gamma_{\CC/\RR}$ by $\mathbb{C}^\times$.

Our goal is to describe the category of representations of $\Kt_{\RR}$ in  terms of linear algebra.

Given a representation $\rho: \Kt_{\RR} \to \mathcal{G}_V$ (see Section \ref{ExplDefRep}),  it defines a grading on $V$ by restricting $\rho$ to $\Gm(\CC)$ and noticing that the morphism on kernels is by definition induced by an algebraic morphism $\Gm \to \GL(V)$.

We let $\alpha$ denote the element of $\mathcal{G}_V$ defined by $\rho(\omega(\sigma))$.
Applying the definition \eqref{gradingdefinition} to $g=-1 = \omega(\sigma)^2$ we obtain that $\alpha^2: V^m \to V^m$ acts by $m(-1) = (-1)^m$.

By definition we have $p^\prime \circ \rho = p$, thus $\alpha$ defines a $\sigma$-linear automorphism of $V$.

Thus we defined a functor $\Rep(\Kt_\RR) \to \mathscr{F}$, that is actually an equivalence, since we can recover a representation of $\Gm$ from the knowledge of grading on $V$ and also recover the image of $w(\sigma)$ from a given semi-linear map $\alpha: V\MapsTo V$.
\end{proof}

\section{Fibre functors}
\label{FibreFunctorsExist}
In this section we show that $\Rep(\Kt_F)$ have fibre functors over any extension of $F$ of cohomological dimension $\leq 1$ (see \cite[Chap.III.2]{SerreCG}.)
More precisely, an extension $\widetilde{F}/F$ is of cohomological dimension $\leq1$, if and only if $\Br(L)= 0$ for every algebraic extension $L$ of $\widetilde{F}$. It is, in particular, equivalent to the assertion that, for every finite algebraic Galois extension $L/\widetilde{F}$, the $\Gal(L/\widetilde{F})$-module $L^\times$ is cohomologically trivial (for Tate cohomology).

Recall that If $F$ is a function field of a curve over $\mathbb{F}_q$, then by Tsen's theorem $F\otimes_{\mathbb{F}_q}\overline{\mathbb{F}}_q$ is of dimension $\leq 1$.
If $F=\QQ$, then any totally imaginary field such that for every ultrametric place the corresponding local degree is $\infty$ (as a supernatural number) is of dimension $\leq 1$ (see Proposition 9 and Corollary, Chap. II, \S3 in \cite{SerreCG}). In particular, $\QQ^{\cycl}$ satisfies this condition.
The main result of this section is stated as follows.

 \begin{theorem}
 \label{FibreFunctors}
 \begin{itemize}
\item $\Rep(\Kt_{\mathbb{Q}})$ has a fiber functor in finite-dimensional vector spaces over $\mathbb{Q}^{\cycl}$, where $\mathbb{Q}^{\cycl}$ is the extension of $\mathbb{Q}$ obtained by adjoining all roots of unity.
 \item If $F$ is the function field of an algebraic curve over $\mathbb{F}_q$, then $\Rep(\Kt_F)$ has a fiber functor in finite-dimensional vector spaces over $\breve{F} = F\otimes_{\mathbb{F}_q}\overline{\mathbb{F}}_q$.
 \end{itemize}
 \end{theorem}

%One may also consider any purely imaginary extension of the $\widehat{\zz}$-extension of $\mathbb{Q}$, the latter is defined as the fixed field of the maximal torsion subgroup of $\Gal(\QQ^{\cycl}/\QQ)$ (see Proposition (5.1), Chap. VI in \cite{Neu2}).

\subsection{Rephrasing the existence of a fiber functor in cohomological terms}Gerbes over $F_{\fpqc}$ banded by $B$ (which we assume to be an abelian affine group scheme) are classified by $H^2(F_{\fpqc},B)$. Thus the existence of a fiber functor of $\Rep(\mathcal{G})$, where $\mathcal{G}$ is such a gerbe, over $\widetilde{F}\supseteq F$ is equivalent to the vanishing of the class of $\mathcal{G}$ in $H^2(\widetilde{F}_{\fpqc},B\times\Spec(\widetilde{F}))$. In particular, the vanishing of the latter cohomology group implies the existence of a fiber functor.

The problem we face working with the Kottwitz gerbes for global fields is that even the bands of $\Kt_{K/F}$ are pro-algebraic groups. Consequently, we must be careful with the cohomology groups mentioned above, since projective limits do not commute with $\fpqc$ cohomology.

In the following sections, we establish exact sequences that relate $\fpqc$ cohomology groups of a projective limit with the projective limit of $\fpqc$-cohomology groups of its algebraic quotients; we also compare cohomology of $\mathbb{D}_{F}$ with that of a simpler pro-torus $\mathbb{T}_{F}$, and deduce the vanishing of $H^2(\widetilde{F}_{\fpqc},\mathbb{D}_F\times \Spec{\widetilde{F}})$ from these two results.
\subsection{Cohomological generalities}
\label{CohomologicalGeneralities}
In Section 3 of \cite{BS} Bhatt and Scholze studied replete topoi, and established several properties of limits in such topoi. We only need some of those results in the special case of the topos of $\fpqc$ sheaves on the category of schemes, that is shown to be replete in Example 3.1.7. of \cite{BS}. We let $\mathcal{X}$ denote the mentioned topos.
\begin{propos}\label{repletelimit} If $\mathcal{X}$ is a replete topos and $F: \mathbf{N}^{\op} \to \Ab(\mathcal{X})$ is a diagram with $F_{n+1}\to F_{n}$ surjective, then $\lim F_n \simeq \Rlim F_n$.
\end{propos}
Also, derived limit commutes with derived functor of global sections, namely
\begin{equation}
\label{DerivedCommutativity}
\Rlim \RGamma(X, F_n) \simeq \RGamma(X, \Rlim F_n).
\end{equation}

Suppose that the conditions of Proposition \ref{repletelimit} are satisfied, then we obtain
\begin{equation}
 \RGamma(X, \varprojlim F_n)\simeq\R\varprojlim \RGamma(X, F_n).
\end{equation}

In order to compute the right-hand side one applies Grothendieck spectral sequence
\begin{equation}
\R^p\varprojlim\circ \R^q \Gamma (X, F_n) \Rightarrow \R^{p+q}(\varprojlim \circ \Gamma)(X,F_n),
\end{equation}
that after being combined with (\ref{DerivedCommutativity}) produces the following short exact sequences
\begin{equation}
\label{GrSpSequence}
0\to \R^1 \varprojlim \R^{i-1}\Gamma(X,F_n) \to \R^i(\Gamma(X,\varprojlim F_n)) \to \varprojlim \R^i(X,F_n)\to 0, \ i\geq 1.
\end{equation}

In what follows we will apply the results of this section in the case $X=X_{\fpqc}$, where $X$ is the spectrum of a field and $\{F_n\}$ is a pro-system of abelian sheaves on $X_{\fpqc}$, then the middle term in (\ref{GrSpSequence}) for $i=2$ is the desired $\fpqc$ cohomology group classifying gerbes banded by $\varprojlim_n F_n$.

We will be using freely several comparison isomorphisms of Grothendieck and Saavedra.
\begin{theorem}[Grothendieck, Saavedra Rivano]
\label{GrSaaTh}
Let $X$ be a scheme and $p: X_{\fl}\to X_{\et} $ be the canonical morphism of sites. Let $G$ be a smooth algebraic group on $X$. Then the canonical morphisms
\begin{equation}
H^i(X_{\et}, p_\ast(G)) \to H^i(X_{\fl}, G)
\end{equation}
are isomorphisms for $i=0,1$ (for all $i\geq0$ if $G$ is abelian, respectively).
If $X=\Spec{k}$ where $k$ is a field, then the canonical morphisms
\begin{equation}
H^i(X_{\fl}, G) \to H^i(X_{\fpqc}, G) \quad \text{for } i=1,2
\end{equation}
are isomorphims.
\end{theorem}

The following statement generalizes Shapiro's lemma and will be useful when we deal with cohomology of $\Res_{L/K} \Gm$ for a finite separable field extension $L/K$.
\begin{theorem}
\label{VarShapiro}
Let $X^\prime \to X$ be a finite \'{e}tale map of schemes, and let $G^\prime$ be a commutative finitely presented and relatively affine $X^\prime$-group scheme. Then the finitely presented and relatively affine $X$-group scheme $G=\Res_{X^\prime/X}(G^\prime)$ defines a sheaf for the \'{e}tale topology on $X$, and there is a natural isomorphism
\begin{equation}
H^i(X, G) \simeq H^i(X^\prime, G^\prime).
\end{equation}
\end{theorem}
\subsection{Reduction to a simpler band}
In the case of Kottwitz gerbes for global fields, $\mathbb{D}_F$ is defined to be the projective limit of protori
\begin{equation}
\mathbb{D}_F = \varprojlim_{K/F} \mathbb{D}_{K/F}
\end{equation}
taken over the directed set of finite Galois extensions of $F$. Similarly, we let $\mathbb{T}_F$ denote the protorus $\varprojlim_{K/F} \mathbb{T}_{K/F}$, where $X^\ast(\mathbb{T}_{K/F}) = \mathbb{Z}[V_K]$ and the transition maps will be specified below (see also Section \ref{SemiLocal}).

The duality for groups of multiplicative type applied to exact sequence (B) introduced in Section \ref{SemiLocal} gives us a short exact sequence of pro-tori
\begin{equation}
\tag{C}
0 \to \varprojlim_{K}\Gm \xrightarrow{} \varprojlim_{K}\mathbb{T}_{K/F} \xrightarrow{} \varprojlim_{K}\mathbb{D}_{K/F}\to 0,
\end{equation}
note that $\varprojlim^1_{K}\Gm$ vanishes since the transition maps are surjective. Then we extend scalars to $\widetilde{F}$ and consider the long exact sequence of $\fpqc$-cohomology
\begin{equation}
\label{longexactforlim}
 \ldots \to H^i(\widetilde{F}, \varprojlim_{K}\Gm) \to H^i(\widetilde{F}, \varprojlim_{K}\mathbb{T}_{K/F}) \to H^i(\widetilde{F},\varprojlim_{K}\mathbb{D}_{K/F})\to \\ H^{i+1}(\widetilde{F},\varprojlim_{K}\Gm)\to\ldots
\end{equation}
where we omitted the base change for corresponding sheaves of abelian groups.

We apply short exact sequence (\ref{GrSpSequence}) to $H^i(\widetilde{F}, \varprojlim_{K}\Gm)$ with $i=1,2,3$. It follows immediately that $H^1(\widetilde{F}, \varprojlim_{K}\Gm)=\varprojlim^1 H^0(\widetilde{F},\Gm)$ since $\varprojlim H^1(\widetilde{F},\Gm)$ vanishes by Hilbert's Theorem 90; $H^2(\widetilde{F}, \varprojlim_{K}\Gm)=0$ since $\varprojlim_K H^2(\widetilde{F}, \Gm)$ vanishes by assumption on $\widetilde{F}$ and  we also have $\varprojlim^1 H^1(\widetilde{F}, \Gm)=0$ by Hilbert's 90; similarly, $H^3(\widetilde{F}, \varprojlim_{K}\Gm)$ vanishes by cohomological dimension reasons.

It follows from the vanishing results above that long exact sequence (\ref{longexactforlim}) induces an isomorphism
\begin{equation}
H^2(\widetilde{F}, \mathbb{T}_F\times\Spec{\widetilde{F}})\simeq H^2(\widetilde{F}, \mathbb{D}_F\times\Spec{\widetilde{F}}).
\end{equation}

\subsection{The pro-torus $\mathbb{T}_{F}$}
Recall that we denoted by $\mathbb{T}_{F}$ the inverse limit of the protori $\mathbb{T}_{K/F}$ defined by their character modules $\zz[V_K]$. The transition maps are given as the maps that are dual to the following homomorphisms of character modules
\begin{equation}
\label{TransMaps}
\zz[V_K] \to \zz[V_L]: v \in V_K \mapsto \sum_{w|v} [L_w:K_v] w, \quad F\subseteq K \subseteq L.
\end{equation}
Then $X^\ast(\mathbb{T}_{F}) = \varinjlim_{K} \zz[V_K]$, the limit is being taken over the set of finite Galois extensions of $F$ contained in the fixed separable closure $\overline{F}$.
We can rewrite $\zz[V_K]$ as a direct sum of finite dimensional $\Gal(K/F)$-modules $X_{v,K}$; namely
\begin{equation}
\label{LocalTorus}
X_{v,K} = \bigoplus_{w\in V_K, w|v} \zz w.
\end{equation}
Let us denote by $\mathbb{T}^v_{K/F}$ the groups of multiplicative type dual to $X_{v,K}$, thus
\begin{equation}
\mathbb{T}_{K/F} = \prod_{v\in V_F} \mathbb{T}^v_{K/F}
\end{equation}
and the pro-torus $\mathbb{T}_F$ can be written as a projective limit of finite products of algebraic groups of the form $\mathbb{T}^v_{K/F}$ indexed by a countable directed set $I$, then one applies the short exact sequence (\ref{GrSpSequence}) to this limit
\begin{equation}
0\to  \R^1\varprojlim_{i\in I} H^1(\widetilde{F}, \mathbb{T}^i \times \Spec(\widetilde{F})) \to H^2(\widetilde{F}, \mathbb{T}_F \times \Spec(\widetilde{F})) \to \varprojlim_{i\in I} H^2(\widetilde{F}, \mathbb{T}^i \times \Spec(\widetilde{F})) \to 0
\end{equation}
Using the fact that cohomology commutes with finite products, we can pass to cohomology of $\mathbb{T}^v_{K/F}\times\Spec{\widetilde{F}}$.
\subsection{Vanishing of $H^1$ and $H^2$ of algebraic quotients}
Let $w$ be a place of $K$ lying over $v\in V_F$, and let $K^w$ denote the decomposition field of $w\in V_K$, that is the field of invariants for the decomposition group $G_w$ at the place $w$, which is the subgroup of $G=\Gal(K/F)$ acting trivially on $w$.

Consider the group of characters of $\mathbb{T}^v_{K/F}$ given by $G$-module $X_{v,K}$ in (\ref{LocalTorus}). The set of places $w|v$ is in bijection with a system of representatives of the cosets in $G/G_w$, and there is an isomorphism of $G$-modules
\begin{equation}
X^\ast(\mathbb{T}^v_{K/F}) = X_{v,K} \simeq \bigoplus_{\sigma \in G/G_w} \zz \sigma(w),
\end{equation}
where the latter is endowed with the natural $G$-structure.

It is clear then that $\mathbb{T}^v_{K/F}$ is isomorphic to $\Res_{K^w/F} \Gm$, the Weil restriction for $\Gm$ over $K^w$ (see Lemma 12.61 in \cite{MilneAG}). Using the compatibility of Weil restriction with base change (see Section 7.6 in \cite{BLR}), we obtain
\begin{equation}
\label{WeilResBaseChange}
\Res_{K^w/F} (\Gm \times\Spec(K^w))\times \Spec(\widetilde{F}) \simeq \Res_{K^w\otimes_F\widetilde{F}/\widetilde{F}} (\Gm \times \Spec{K^w} \times \Spec(\widetilde{F}))
\end{equation}
with all the fiber products taken over $\Spec(F)$.
Note that $\Spec(K^w) \times_{\Spec(F)} \Spec(\widetilde{F})\simeq \coprod  \Spec(\widetilde{K}_j)$, where the fields $\widetilde{K}_j$ are extensions of $\widetilde{F}$ defined by the following decomposition
\begin{equation*}
K^w \otimes_F \widetilde{F} \simeq \prod \widetilde{K}_j.
\end{equation*}
It follows then from the universal property defining Weil restriction applied to the right-hand side of (\ref{WeilResBaseChange}) that the latter is isomorphic to
\begin{equation}
\label{DecompositionOfALocalFactor}
\prod_j \Res_{\widetilde{K}_j/\widetilde{F}}(\Gm\times_{\Spec(\widetilde{F})} \Spec(\widetilde{K}_j)).
\end{equation}
Consequently, in order to prove vanishing of $H^i(\widetilde{F}, \mathbb{T}^v_{K/F})$ for $i=1,2$ it is enough to show this for every factor in (\ref{DecompositionOfALocalFactor}).

By Theorem \ref{VarShapiro}, we have
\begin{equation}
H^i(\widetilde{F}_{\et}, \Res_{\widetilde{K}_j/\widetilde{F}}(\Gm\times_{\Spec(\widetilde{F})} \Spec(\widetilde{K}_j))) \simeq
H^i((\widetilde{K}_j)_{\et}, \Gm).
\end{equation}

We conclude by noticing that $H^1(\widetilde{K}_j, \Gm)=0$ by Hilbert's 90, and $H^2(\widetilde{K}_j, \Gm)=0$ since $\widetilde{K}_j/\widetilde{F}$ is a finite algebraic extension of a field of cohomological dimension $\leq1$.

\section{Local-global principle}
\label{LocalGlobalPrinSec}
In this section $F$ is a global field, $\bar{F}$ denotes its separable closure, and $\Gamma_F = \Gal(\bar{F}/F)$. We usually omit $F$ from notation, when the base field is clear. For any Galois extension $K/F$ we denote $\Gal(K/F)$ by $\Gamma_{K/F}$.

\subsection{Ad\`{e}lic cohomology}
\label{AdCoh}
For a finite set of primes $S$ in $F$, $\mathbb{A}^S$ is
the restricted product of $F_v$ for $v\notin S$, and for a finite extension $K/F$, $\mathbb{A}_K^S = \mathbb{A}^S \otimes_F K$. When $S$ is empty, we omit it from notation.

For a torus $T$ over $F$ define
\begin{equation*}
H^r(\mathbb{A}^S,T) = \varinjlim_{K} H^r(\Gamma_{K/F}, T(\mathbb{A}^S_K)),
\end{equation*}
where the limit is taken over the finite Galois extensions $L/F$ contained in a fixed separable closure of $F$.

For the proof of the next result we refer to \cite[Prop. 2.13]{MilneGerbes}.
\begin{propos}\label{AdelicCoh} Let $T$ be a torus over $F$. For all $r\geq1$, there is a natural isomorphism
\begin{equation}
H^r(\mathbb{A}^S, T) \simeq \bigoplus_{v\notin S} H^r(F_v, T).
\end{equation}
Moreover, a short exact sequence of tori gives rise to a long exact sequence of ad\`{e}lic cohomology.
\end{propos}
\subsection{The principle}
\label{SecPrinciple}
Let $K/F$ be a finite Galois extension of global fields, and let $S\subset V_F$ be a finite subset of the set of places of $F$. We define $S_K$ as the set of all places of $K$ lying over $S$, and let $\dot{S}_K$ denote a set of lifts of elements of $S$ to places of $K$.

Recal that there is an exact sequence of Galois modules
\begin{equation}
0 \to \zz[S_K]_0 \xrightarrow{b^\prime} \zz[S_K] \xrightarrow{b} \zz\to 0,
\end{equation}
and an exact sequence of algebraic tori dual to it, namely,
\begin{equation}
\label{FiniteTori}
0 \to \Gm \xrightarrow{b} \mathbb{T}_{K,S} \xrightarrow{b^\prime}\mathbb{D}_{K,S} \to 0.
\end{equation}
Applying exact sequence (\ref{GrSpSequence}) to $\mathbb{D}_F$, we obtain
\begin{equation}
0\to R^1\varprojlim H^1(F,\mathbb{D}_{K,S}) \to H^2(F,\mathbb{D}_F) \to \varprojlim H^2(F,\mathbb{D}_{K,S})\to 0.
\end{equation}

Since $\mathbb{D}_{K,S}$ is a smooth affine algebraic group, we have an isomorphism  
\begin{equation*}
H^i(F,\mathbb{D}_{K,S})\simeq H^i(\Gamma, \mathbb{D}_{K,S}(\bar{F})).
\end{equation*}
It was shown in Section 3.8 of \cite{KalTai19}, that $R^i \varprojlim H^1(\Gamma, \mathbb{D}_{K,S}(\bar{F}))=0$ for $i=0,1$. Therefore the sequence of cohomology groups above is reduced to the isomorphism
\begin{equation}
H^2(F,\mathbb{D}_F) \MapsTo \varprojlim H^2(F,\mathbb{D}_{K,S}).
\end{equation}

Next we deal with the terms $H^2(F,\mathbb{D}_{K,S})$. First, note that (\ref{FiniteTori}) induces the long exact sequence of cohomology
\begin{equation}
\to H^1(F,\mathbb{T}_{K,S})\to H^1(F,\mathbb{D}_{K,S})\to H^2(F,\Gm) \to H^2(F,\mathbb{T}_{K,S}) \to H^2(F,\mathbb{D}_{K,S}) \to H^3(F,\Gm)\to
\end{equation}

We can simplify this sequence using the following lemma.

\begin{lemma}
\begin{itemize}
\item $H^1(F,\mathbb{T}_{K,S}) =0$.
\item $H^2(F,\mathbb{T}_{K,S}) = \bigoplus_{v \in S} \Br(K^{\dot{v}})$, where $\dot{v}\in S_K$ is an arbitrary lift of $v\in S$, and $K^{\dot{v}}$ is the decomposition field of $\dot{v}$.
\item $H^3(F,\Gm) = 0$.
\item The map $H^2(F,\Gm) \to H^2(F,\mathbb{T}_{K,S})$ in the sequence above coincides with the product of restrictions $\res_v: \Br(F)\to \Br(K^{\dot{v}})$.
\end{itemize}
\end{lemma}
\begin{proof}
The first two statements follow from the \'{e}tale variant of Shapiro's lemma (see Theorem \ref{VarShapiro}) combined with Hilbert's Theorem 90.

The vanishing of $H^3(F,\Gm)$ follows from the comparison of \'{e}tale and Galois cohomology and Theorem 14 in \cite{ArtinTate}, in which it was shown that
\begin{equation*}
\varinjlim_{K} H^3(\Gal(K/F), K^\ast) = 0,
\end{equation*}
where the limit is being taken over the set of finite Galois extensions $F\subset K\subset \bar{F}$.

The last statement follows from Proposition 1.6.5 in \cite{Neu1}.
\end{proof}

It follows from the lemma above that
\begin{equation}
H^2(F,\mathbb{D}_{K,S})= \coker(\bigoplus_{v\in S} \res_v: \Br(F) \to \bigoplus_{v\in S} \Br(K^{\dot{v}})).
\end{equation}

On the other hand, applying Albert–Brauer–Hasse–Noether theorem, we obtain the following commutative diagram

\begin{equation}
\label{LocGlobSnake}
\begin{tikzcd}
 0\arrow{r}&   \Br(F) \arrow{r} \arrow{d}{\bigoplus_{v\in S}\res_v } & \bigoplus_{w\in V_F} \Br(F_w) \arrow{d} \arrow{r} & \QQ/ \zz\arrow{r} \arrow{d}{\bigoplus \deg(K^{\dot{v}}:F)} &0\\
 0\arrow{r}&   \bigoplus_{\dot{v}\in \dot{S}_K} \Br(K^{\dot{v}})   \arrow{r} &  \bigoplus_{\dot{v}\in \dot{S}_K}\bigoplus_{w\in V_{K^{\dot{v}}}} \Br(K^{\dot{v}}_w) \arrow{r} & \bigoplus_{v\in S} \QQ / \zz \arrow{r} &0
\end{tikzcd}
\end{equation}

We let $B_{K,S}$ denote $\ker( \QQ/\zz \to \bigoplus_{v\in S} \QQ/\zz )$. Note that $B_{K,S}$ is a finite subgroup of $\QQ/\zz$ of order $\gcd\{\deg(K^{\dot{v}}:F) \mid v\in S\}$.
\begin{remark} If $K$ is a cyclic extension of $F=\mathbb{Q}$ and $S$ contains an inert prime $v$ (that exists in this case), then $B_{K,S}$ vanishes, since $\deg(K^v:\QQ) =1$.
\end{remark}

\begin{remark}\label{AdelicCohom}
 The middle column in diagram (\ref{LocGlobSnake}) can be reinterpreted in terms of ad\`{e}lic cohomology. Namely $H^2(\mathbb{A}, \Gm) = \oplus_{v\in V_F} \Br(F_v)$ and analogously
 \begin{equation*}
 H^2(\mathbb{A}, \mathbb{T}_{K,S}) =  \bigoplus_{\dot{v}\in \dot{S}_K}\bigoplus_{w\in V_{K^{\dot{v}}}} \Br(K^{\dot{v}}_w).
 \end{equation*}
 Moreover, since $H^3(\mathbb{A},\Gm)=0$, we have an isomorphism
 \begin{equation}
 H^2(\mathbb{A},\mathbb{D}_{K,S}) \simeq \coker\big(H^2(\mathbb{A},\Gm) \to H^2(\mathbb{A},\mathbb{T}_{K,S})\big).
 \end{equation}
\end{remark}

Let us assume that the inverse system $B_{K,S}$ satisfies the Mittag-Leffler condition and $\varprojlim B_{K,S} = 0$. We will specify the transition maps below. Snake lemma applied to the diagram above gives us an obvious exact sequence, part of which is as follows
\begin{displaymath}
\xymatrix@C=0.5cm{
  \ldots \ar[r] & B_{K,S} \ar[rr]^{\delta_{K,S}\quad} && H^2(F,\mathbb{D}_{K,S}) \ar[rr]^{\eta_{K,S}} && H^2(\mathbb{A},\mathbb{D}_{K,S}) \ar[rr]^{} && \ldots  }
\end{displaymath}
that leads us to the short exact sequence
\begin{equation}
0 \to \coker{\delta_{K,S}} \to H^2(F,\mathbb{D}_{K,S}) \to \im{\eta_{K,S}}\to 0,
\end{equation}
to which we apply the inverse limit functor and obtain $\varprojlim H^2(F,\mathbb{D}_{K,S}) \simeq \varprojlim \im{\eta_{K,S}}$. On the other hand, $\varprojlim \im{\eta_{K,S}}\hookrightarrow \varprojlim H^2(\mathbb{A},\mathbb{D}_{K,S})$, from which we deduce a local-global principle for Tannakian categories banded by $\mathbb{D}_F$, that is stated as follows.
\begin{theorem}
\label{LocalGlobalTheorem}
 Let $F$ be a global field and let $\mathbb{D}_F$ be the Kottwitz pro-torus. Then the natural map
\begin{equation}
\label{LocalGlobal}
H^2(F,\mathbb{D}_F) \to \varprojlim H^2(\mathbb{A},\mathbb{D}_{K,S}).
\end{equation}
is injective.
\end{theorem}

We will see in Lemma \ref{MittagLeff} below that the inverse system of abelian groups $B_{K,S}$ satisfies the desired properties. Let us firstspecify the transition maps. If  $S^\prime$ is a finite set of places of $F$ containing $S$, the map $B_{K,S^\prime} \to B_{K,S}$ is induced by the following diagram
\begin{displaymath}
\xymatrix@R+1pc@C+6pc{
  \QQ/\zz \ar[d]_{\id} \ar[r]^{\bigoplus_{v\in S^\prime} \deg(K^{\dot{v}}:F)} & \bigoplus_{v\in S^\prime} \QQ/\zz \ar[d]^{\pi} \\
  \QQ/\zz \ar[r]^{\bigoplus_{v\in S} \deg(K^{\dot{v}}:F)} & \bigoplus_{v\in S} \QQ/\zz   }
\end{displaymath}
where $\pi$ is an obvious projection. It is clear that $B_{K,S} = B_{K,S^\prime}$ if $S$ is large enough, namely, if $S$ contains places $w$ corresponding to all possible degrees of decomposition fields $K^w/F$. Note that it is a finite condition, since the latter correspond to decomposition subgroups in $\Gamma_{K/F}$. Therefore, it is harmless to assume that $S$ satisfies this condition.

Let $E$ be a finite Galois extension of $K$, in this situation we have the set $S_E$ and a map $\zz[S_K] \to \zz[S_E]$ defined as follows: $v\in S_K \mapsto \sum_{w\mid v} \deg(E_w:K_v) w$. Composing this map with the corresponding degree maps, we see that it induces the multiplication by $\deg(E:K)$ on $\zz$. Note that after passing to the tori dual to these Galois modules and computing their cohomology, this map corresponds to the multiplication by $\deg(E:K)$ on $\Br(F)$, that induces the map $\QQ/\zz \xrightarrow{\deg(E:K)} \QQ/\zz$, and we have the following commutative diagram
\[\begin{tikzcd}
    B_{E,S} \arrow[hook]{r} \arrow{d} & \QQ/\zz \arrow{d}{\cdot \deg(E:K)} \\
    B_{K,S} \arrow[hook]{r} & \QQ/\zz
\end{tikzcd}\]
that shows that the transition map $B_{E,S}\to B_{K,S}$ factors through the multiplication by $\deg(E:K)$.

\begin{lemma}
\label{MittagLeff}
Let $\{B_{K,S}\}$ be the inverse system defined above. Then $B_{K,S}$ satisfies the Mittag-Leffler condition and its inverse limit vanishes.
\end{lemma}
\begin{proof}
From the discussion above it follows that it is sufficient to show that for any pair $(K,S_K)$ there exists a finite Galois extension $E/K/F$ such that $\deg(E:K)$ is divisible by $|B_{E,S}|$.
Suppose that $l=|B_{K,S}|$ is co-prime to the characteristic of $F$. We can always assume that $K$ contains enough roots of unity and can also replace $K$ with its Hilbert class field, then we can set $E=K(\sqrt[l]{y})$ where $y$ is the product of divisors in $\dot{S}_K$, then every prime from $S_K$ remains non-split in $E$ and $\deg(E:K) = |B_{E,S}| = |B_{K,S}|$.

In the case $|B_{K,S}|$ is divisible by $p^n$ where $p$ is the characteristic of $F$, we may take sufficiently large unramified extension of $K$ and consider an extension $E=K(y)$, where $y$ is a root of the additive polynomial $x^q - x = a$, where $q=p^n$, and $a$ is the inverted product of divisors in $\dot{S}_K$. Then $\deg(E:F) = p^n$ and $|B_{E,S}|=|B_{K,S}|$ since every local degree at $w\in S_K$ is equal to $p^n$ (see Theorem 4.1 in \cite{SP}). Now the argument of the previous part can be applied to deal with the co-prime to $p$ part of $B_{K,S}$.
\end{proof}

\section{$\varphi$-spaces}
\subsection{Drinfeld's classification of $\varphi$-spaces}
\label{IsoshtukasIntro}
Let $F$ be a function field with the field of constants $\mathbb{F}_q$.
We are going to introduce the category of $\varphi$-spaces (or isoshtukas) defined by Drinfeld in \cite{Dr}.
\begin{definition}
Let $B$ be an algebraic extension of $\mathbb{F}_q$. A $\varphi$-space over $B$ is a finite dimensional vector space $V$ over $F\otimes B$ with an $\id\otimes \Frob_B$-linear map $\varphi:V \MapsTo V$. A morphism between two $\varphi$-spaces $(V_1,\varphi_1)$ and $(V_2,\varphi)$ is an $F\otimes B$-linear map $\alpha: V_1\to V_2$ such that $\alpha\circ \varphi_1 = \varphi_2 \circ \alpha$.
\end{definition}

The $\varphi$-spaces over $B$ form an $F$-linear Tannakian $\otimes$-category $\varphi\text{-Vect}(B)$. In order to prove that this category is semi-simple, when $B=\overline{\F}$, and to describe the set of simple objects, Drinfeld introduces the notion of a $\varphi$-pair.
\begin{definition}
A $\varphi$-pair is a pair $(L,a)$, where $L$ is a commutative finite-dimensional $F$-algebra, and where $a\in L^\times\otimes \qq$ satisfies the following property: for any proper $F$-subalgebra $L^\prime$ of $L$, $a\notin (L^\prime)^\times\otimes \qq$.
\end{definition}

Notice that the latter condition in the definition of a $\varphi$-pair ensures that the group of automorphisms of a given pair is trivial, it also implies that $L$ is separable, and the set of isomorphisms of $\varphi$-pairs is identified with the set of orbits of $(F^{\text{sep}})^\times\otimes\qq$ under the action of $\Gal(F^{\text{sep}}/F)$. On the other hand, it follows from global class field theory that
\begin{equation}
(F^{\text{sep}})^\times\otimes\qq \simeq \varinjlim_{E/F}\Div^0(E)\otimes\qq.
\end{equation}

To every $\varphi$-space $(V,\varphi)$, Drinfeld associates a $\varphi$-pair $(L_{(V,\varphi)}, a_{(V,\varphi)})$, and notices that if $(V,\varphi)$ is a simple object in the category of $\varphi$-spaces, then $L_{(V,\varphi)}$ is a field. In particular. $a_{(V,\varphi)}$ is just an element in $L_{(V,\varphi)}^{\times}\otimes \mathbb{Q}$.

Let us briefly recall the construction of a $\varphi$-pair associated to a (simple) $\varphi$-space $(V,\varphi)$. Firstly, we find a pair $(V^\prime,\varphi^\prime)$, where $V^\prime$ is a finite-dimensional vector space defined over $F\otimes_{\F_q}\F_{q^n}$ for some natural $n$, and $\varphi$ is a $1\otimes\Frob_{\F_{q^n}}$-linear endomorphism of $V^\prime$ such that
\begin{equation}
(V,\varphi) \simeq (V^\prime, \varphi^\prime)\otimes_{\F_{q^n}}\overline{\F}_q.
\end{equation}

Then since $\varphi^n$ is a linear automorphism of $V^\prime$, it extends to a linear automorphism of $V$, namely $(\varphi^\prime)^n\otimes 1$, we denote it by $\Pi$. At the same time, $\Pi$ is an endomorphism of $(V,\varphi)$ as a $\varphi$-space, since it obviously commutes with $\varphi$. Thus $\Pi^N$ for $N\geq 1$ generates a commutative finite-dimensional algebra  $L_N=F[\Pi^N]\subseteq\End(V,\varphi)$.

Our next step is to consider an algebra
\begin{equation}
L = \bigcap_{N\geq1} L_N,
\end{equation}
which does not depend on the choice of $\varphi^\prime$.
\begin{remark}
\label{divisbilitystable}
Since all $L_N$ are finite-dimensional, there exists an index $N$ such that $L=L_N$, moreover, since for any $l,m\geq 1$ such that $l | m$ we have an inclusion $L_m \subseteq L_l$, we see that $L_N = L_{N^\prime}$ for any $N^\prime\geq 1$ divisible by $N$.
\end{remark}
To proceed, observe that $\Pi^N \in L^\times$ is not a zero divisor, and denote by $w_N$ its image in $L^\times\otimes\QQ$. Finally, we set \begin{equation}
a = a_{(V,\varphi)} = (w_N)^{\frac{1}{n^\prime N}},
\end{equation}
and note that $a$ does not depend on the choices of $\varphi^\prime$ and $N$. The pair $(L,a)$ is a $\varphi$-pair associated to $(V,\varphi)$.

The following theorem is proved in \cite{Dr}. Note that $\deg_x(a)$ for $a\in \Div^0(L)\otimes \QQ$ denotes the product of the degree of the residue field of the place $x$ and the multiplicity of $x$ in the rational divisor  $a$.
\begin{theorem}
\label{DrinfeldClassification}
\begin{enumerate}
\item
The category of $\varphi$-spaces over $\overline{\F}_q$ is abelian and semi-simple.
\item
The map $(V,\varphi) \mapsto (L_{(V,\varphi)}, a_{(V,\varphi)})$ induces a bijection between the set of isomorphism classes of irreducible $\varphi$-spaces and the set of isomorphism classes of $\varphi$-pairs $(L,a)$, where $L$ is a field.
\item Let $(V,\varphi)$ be an irreducible $\varphi$-space corresponding to a pair $(L,a)$, then
\begin{equation}
\dim_{F\otimes_{\F_q} \overline{\F}_q}(V) = [L:F]\cdot d(a),
\end{equation}
where $d(a)$ denotes the common denominator of $\deg_{\tilde{x}}(a)\in \qq$. Moreover, $\End(V,\varphi)$ is a central division algebra over $L$ of dimension $d(a)^2$ with local invariants $-\deg_{\tilde{x}}(a) \pmod \zz$.
\end{enumerate}
\end{theorem}

We wish to reinterpret this theorem as a description of a Tannakian category $\varphi\text{-Vect}(\overline{\F}_q)$. As a preparatory step we should prove the following statement
\begin{propos}
Let $(V_1, \varphi_1)$ and $(V_2,\varphi_2)$ be two isotypic $\varphi$-spaces over $\overline{\F}_q$ of types $a_1, a_2 \in (F^{\text{sep}})\otimes \qq$ respectively, then their tensor product $(V_1\otimes V_2, \varphi_1\otimes\varphi_2)$ is an isotypic $\varphi$-space of type $a_1 a_2$.
\end{propos}
\begin{proof}
Since the category of $\varphi$-spaces is semi-simple, we may assume that $(V_i,\varphi_i)$ are simple objects.

Let $(\tilde{V},\tilde{\varphi})$ denote the tensor product $(V_1\otimes V_2, \varphi_1\otimes\varphi_2)$.

Following the construction, we define $(V^\prime_i,\varphi^\prime_i)$, $i=1,2$, to be $\varphi$-spaces over $\F_{q^{n_i}}$ such that
\begin{equation*}
(V_i,\varphi_i) = (V^\prime_i,\varphi^\prime_i) \otimes_{\F_{q^{n_i}}} \overline{\F}_q, i=1,2.
\end{equation*}
Note that $\varphi_i^{n_i}$ is a linear bijective map $V_i \to V_i$. We can find an integer $n^\prime$ such that both spaces $(V^\prime_i,\varphi^\prime_i)$($i=1,2$) are defined over $\F_{q^{n^\prime}}$. We denote by $(L_i,a_i)$ the corresponding $\varphi$-pairs, where $L_i= F[\Pi_i^{N_i}]$ is a separable extension of $F$ for $i=1,2$, and $\Pi_i = (\varphi_i^{\prime\ n^\prime}\otimes \id)$, we have also
\begin{equation}
a_i= (w_{N_i})^{\frac{1}{n^\prime N_i}}\in L^\times_i\otimes \QQ.
\end{equation}

Notice that
\begin{equation}
(V_1\otimes V_2, \varphi_1 \otimes \varphi_2) \simeq (V_1^\prime \otimes V_2^\prime, \varphi_1^\prime\otimes\varphi_2^\prime)\otimes_{\F_{q^{n^\prime}}} \overline{\F}_q,
\end{equation}
and that $\Pi =(\varphi_1^\prime\otimes\varphi_2^\prime)^n\otimes 1$ is a linear bijective endomorphism of $(\widetilde{V},\widetilde{\varphi})$, let $L= L_N$ denote the intersection of finite-dimensional commutative algebras $F[\Pi^N]$ with $N\geq1$.

Then using Remark \ref{divisbilitystable}, we deduce that $N_1$, $N_2$, and $N$ can be chosen to be equal, and $\Pi^N = \Pi_1^N \otimes \Pi_2^N$ as an element in $\End(\widetilde{V},\widetilde{\varphi})\simeq \End(V_1,\varphi_1)\otimes_F\End(V_1,\varphi_1)$; furthermore, $a \in L^\times \otimes \QQ$ can be identified with its image in $(L_1 L_2)^\times \otimes \QQ$, which is the same as $a_1 a_2$. Thus $(\widetilde{V},\widetilde{\varphi})$ is an isotypic $\varphi$-space of type $a_1 a_2$.
\end{proof}

\subsection{The band of the category of $\varphi$-spaces} In this section, we show that the band of the category of $\varphi$-spaces can be identified with the global Kottwitz protorus $\mathbb{D}_F$.

 Notice that in order to compute the band of a semisimple Tannakian category $\mathcal{C}$ over $F$ it is enough to do so for the category $\mathcal{C}\otimes \widetilde{F}$ for any $\widetilde{F}/F$ such that $\mathcal{C}$ has a fibre functor in $\Vect_{\widetilde{F}}$, in which case $\mathcal{C}\otimes\widetilde{F}$ becomes a neutral Tannakian category over $\widetilde{F}$, and therefore is identified with the category of representations of its band $\Rep_{\widetilde{F}}(B)$, where we assume $B$ to be represented by an affine scheme over $\widetilde{F}$. Moreover, if any simple object of $\mathcal{C}\otimes\widetilde{F}$ is of rank $1$, then $B$ is identified with $D(\Sigma)$, the diagonalizable group dual to the abelian group of isomorphism classes of simple objects with the group structure given by tensor product (see Proposition 2.22 in \cite{Milne1994}).

In the case when $\mathcal{C}$ is the category of $\varphi$-spaces, we may take $\widetilde{F} = F^{\sep}$. Clearly, $\varphi$-spaces has a fibre functor over $F^{\sep}$, and therefore $\varphi$-spaces becomes a neutral Tannakian category after tensoring with $F^{\sep}$. We claim that every simple object in $\mathcal{C}\otimes F^{\sep}$ is of rank $1$, and the set of isomorphism classes of simple objects is given by the abelian group $(F^{\sep})^\times\otimes\mathbb{Q}$.

Let $(V,\varphi)$ be a simple object in $\mathcal{C}$, and let $(L,a)$ be its $\varphi$-pair. Since $\End(V,\varphi)$ is a central division algebra over $L$, we see that $\End(V\otimes_F F^{\sep},\varphi\otimes \id)$ is a central division algebra over
\begin{equation}
L\otimes_F F^{\sep} = \prod_{j: L\hookrightarrow K} F^{\sep}
\end{equation}
and therefore $(V\otimes_F F^{\sep},\varphi\otimes\id)$ splits into the sum of subobjects given by $(V_j,\varphi_j) = (V\otimes_{L,j} F^{\sep},\varphi\otimes(\id\circ j))$. Moreover, $\End(V_j,\varphi_j)$ is a central simple algebra over $F^{\sep}$, hence it is isomorphic to $M_{d(a)}(F^{\sep})$, where $d(a) = \deg_L(\End(V,\varphi))$ is as in Theorem \ref{DrinfeldClassification}. On the other hand
\begin{equation*}
\dim_L(V) = \dim_{F^{\sep}}(V_j) = d(a),
\end{equation*}
and therefore the centre of $\End(V_j,\varphi_j)$ provides us with idempotents $e_1,\ldots,e_{d(a)}$ splitting $(V_j,\varphi_j)$ into the sum of $d(a)$ simple isotypic objects $(e_{i}V_j,e_i\varphi_j)$,  for $i=1,\ldots, d(a)$, of rank $1$ and slope $a\in (F^{\sep})^{\times}\otimes\mathbb{Q}$.

\subsection{Drinfeld spaces and Dieudonn\'{e} modules}
Let $u$ be a place of $F$, and let $F_u$ denote the completion of $F$ at $u$.

We have an exact and $F$-linear functor
\begin{equation}
\Loc_u: (V,\varphi)\mapsto (V_u,\varphi_u) = (V\hat{\otimes}_F F_u, \varphi\hat{\otimes}_F F_u).
\end{equation}

The following theorem is due to Drinfeld (see also Proposition (B.4) in \cite{LRS}).
\begin{theorem}
\label{phiSpacesLocal} Fix a place $u\in V_F$.
Let $(V,\varphi)$ be an irreducible $\varphi$-space  and let $(L,a)$ be the corresponding $\varphi$-pair. For every place $w\in L$ dividing $u$ we let
\begin{equation}
(V_w, \varphi_w) = L_w \otimes_L (V,\varphi).
\end{equation}
The canonical splitting $F_u \otimes_F L \simeq \prod_{w|u} L_w$ induces a splitting
\begin{equation}
\label{ShapiroIsoc}
(V_u, \varphi_u) = \oplus_{w|u} (V_w, \varphi_w)
\end{equation}
of $(V_u,\varphi_u)$ as a Dieudonn\'{e} $F_u$-module. Then for each $w|u$, $(V_w,\varphi_w)$ is non-canonically isomorphic to
\begin{equation}
(N_{d_w, r_w}, \psi_{d_w,r_w})^{s_w}
\end{equation}
where the integers $d_w$, $r_w$, and $s_w$ are determined uniquely by the following relations
\begin{equation}
\begin{cases}
d_w, s_w \geq 1\\
(d_w, r_w) = 1\\
r_w/d_w = \deg_w(a)/ \deg(L_w:F_u)\\
d_w s_w = d(a)  \deg(L_w:F_u)
\end{cases}
\end{equation}
\end{theorem}

\section{Comparison over function fields}
In this section we prove the following theorem.
\begin{theorem}
\label{FunctionFields}
Let $F$ be a function field of a smooth algebraic curve over a finite field. Then the category of representations of the Kottwitz gerbe $\Rep(\Kt_F)$ is equivalent to the category of Drinfeld isoshutkas (see Section \ref{IsoshtukasIntro}).
\end{theorem}

\subsection{Ad\`{e}lic Kottwitz class}
The definition of the Kottwitz class as an element of $H^2(F, \mathbb{D}_F)$ is quite inexplicit, but we are able to compute its local components in $\varprojlim H^2(\mathbb{A}, \mathbb{D}_{K,S})$.

It follows from the isomorphism (\ref{IsomHJK}) and Remark \ref{AdelicCohom} that the natural map
\begin{equation}
H^i(\Gamma_{K/F}, \Hom(\zz[S_K],\mathbb{A}_{K,S}^\times))\to H^i(\Gamma, \Hom(\zz[S_K], \mathbb{A}_K^\times)) \to H^i(\mathbb{A}, \mathbb{T}_{K,S})
\end{equation}
is injective. Therefore, the limit of semi-local classes $\alpha_2(K,S)\in H^i(\Gamma_{K/F}, \Hom(\zz[S_K],\mathbb{A}_{K,S}^\times))$, defines a unique element in $\varprojlim H^2(\mathbb{A},\mathbb{T}_{K,S})$. 

Furthermore, on every finite level we have an equality of Tate-Nakayama classes $b^\prime \alpha_2(K,S) = a^\prime \alpha_3(K,S)$ coming from natural maps
\begin{equation}
\label{classesfinlevel}
\begin{tikzcd}
                          \quad &                                H^2(\Gamma_{K/F}, \Hom(\zz[S_K], \mathbb{A}^\times_{K,S})) \arrow{d}{b^\prime} \\
H^2(\Gamma_{K/F}, \Hom(\zz[S_K]_0, \OO^\times_{K,S})) \arrow{r}{a^\prime}   &  H^2(\Gamma_{K/F}, \Hom(\zz[S_K]_0, \mathbb{A}^\times_{K,S}))
\end{tikzcd}
\end{equation}
\subsection{Inflations of classes defined for finite extensions} The following diagrams
\[
\begin{tikzcd}
H^2(\Gamma_{K/F}, \Hom(\zz[S_K]_0, \OO^\times_{K,S})) \arrow{r}{a^\prime} \arrow{d}   &  H^2(\Gamma_{K/F}, \Hom(\zz[S_K]_0, \mathbb{A}^\times_{K,S})) \arrow{d}\\
H^2(F, \mathbb{D}_{K,S}) \arrow{r}{}   &  H^2(\mathbb{A}, \mathbb{D}_{K,S})
\end{tikzcd}
\]
and
\[
\begin{tikzcd}
H^2(\Gamma_{K/F}, \Hom(\zz[S_K], \mathbb{A}^\times_{K,S})) \arrow{r}{b^\prime} \arrow{d}   &  H^2(\Gamma_{K/F}, \Hom(\zz[S_K]_0, \mathbb{A}^\times_{K,S})) \arrow{d}\\
H^2(\mathbb{A}, \mathbb{T}_{K,S}) \arrow{r}{b^\prime}   &  H^2(\mathbb{A}, \mathbb{D}_{K,S})
\end{tikzcd}
\]
commute. This basically follows from the same computation as Proposition \ref{AdelicCoh}.
\subsection{Local components of the Kottwitz class}
\label{LocalComponents}
We note that the maps $H^2(\mathbb{A}, \mathbb{T}_{K,S}) \to H^2(\mathbb{A}, \mathbb{D}_{K,S})$ are surjective. They induce a surjective map between the corresponding derived limits, since the kernel of every such projection is a quotient of $H^2(\mathbb{A},\Gm)$, that is a direct sum of $\QQ/\zz$ and $\frac{1}{2}\zz$, and has vanishing $R^1\varprojlim$. Therefore, the cokernel of
\begin{equation}
\varprojlim H^2(\mathbb{A}, \mathbb{T}_{K,S}) \to \varprojlim H^2(\mathbb{A}, \mathbb{D}_{K,S})
\end{equation}
is zero.

We have a diagram analogous to (\ref{classesfinlevel}), namely
\begin{equation}
\label{GlobalAdelicAdelic}
\begin{tikzcd}
                          \quad &                                \varprojlim H^2(\mathbb{A}, \mathbb{T}_{K,S}) \arrow[twoheadrightarrow]{d}{b^\prime} \\
H^2(F, \mathbb{D}_F) \arrow[hook]{r}{a^\prime}   &  \varprojlim H^2(\mathbb{A}, \mathbb{D}_{K,S})
\end{tikzcd}
\end{equation}
and two uniquely defined classes $\varprojlim \alpha_2(K,S)$ and $\varprojlim \alpha_3(K,S)$, such that their images coincide in $\varprojlim H^2(\mathbb{A}, \mathbb{D}_{K,S})$.

The images of $\alpha_2(K,S)$ in $H^2(\mathbb{A}, \mathbb{T}_{K,S})$ can be described fairly easily via isomorphism (\ref{IsomHJK}). More precisely, $\alpha_2(K,S)$ corresponds to the element $(\alpha_{K^{\dot{v}}/F})_{v\in S}$ where each  $\alpha_{K^{\dot{v}}/F}$ is the image of the local fundamental class in $H^2(\Gamma_{K/K^{\dot{v}}}, K_{\dot{v}}^\times)\simeq \frac{1}{n_v}\zz$ where $n_v = \deg(K:K^{\dot{v}}) = \deg(K_{\dot{v}}:F_v)$.

We can say more about the limit of $\alpha_2(K,S)$, since
\begin{equation}
H^2(\mathbb{A},\mathbb{T}_{K,S}) \simeq \bigoplus_{\dot{v}\in \dot{S}_K}\bigoplus_{w\in V_{K^{\dot{v}}}} \Br(K^{\dot{v}}_w),
\end{equation}
and we notice that $\Br(K^{\dot{v}}_w) = \Br(F_v)$ for any place $w$ lying over $v$. Moreover, the transition maps $\mathbb{T}_{E,S^\prime} \to \mathbb{T}_{K,S}$ are defined locally and for any place $w^\prime \in S_E$ dividing $w$ induce the map
\begin{equation}
\cdot \deg(E_{w^\prime}: K_w): \Br(E^{\dot{v}}_{w^\prime}) = \Br(F_v) \to \Br(K^{\dot{v}}_w) = \Br(F_v).
\end{equation}
 We see that the $v$-component of the inverse limit of local canonical classes with such transition maps coincides with the cohomology class  $\alpha_v\in H^2(F_v,\widetilde{\Gm})$ corresponding to the category of isocrystals. Here $\widetilde{\Gm}$ denotes the inverse limit of $\Gm$ with $X^\ast(\widetilde{\Gm}) =\QQ$. This implies that $\varprojlim \alpha_2(K,S)$ is given by the image of the product of the $\alpha_v$'s under the inclusion
\begin{equation}
\prod_{v\in V_F} H^2(F_v, \widetilde{\Gm}) \hookrightarrow \varprojlim H^2(\mathbb{A}, \mathbb{T}_{K,S}).
\end{equation}

\subsection{$\Div^0(K)$ and $\zz[V_K]_0$}
\label{isomCharacters}
 Note that the degree map defined on the abelian group $\Div(K)$ is defined as follows
\begin{equation}
\deg: \Div(K) \to \zz,\quad
\sum_{P\in V_K} n_P P  \mapsto \sum_{P\in  V_K} \deg(P) n_P,
\end{equation}
where $\deg(P)$ is the residue degree at the place $P$. Thus one has an isomorphism between $\Div(K)\otimes\QQ$ and $\zz[V_K]\otimes\QQ$ preserving kernels of the corresponding degree maps, which is given by
\begin{equation}
\Div(K)\otimes\QQ\to\zz[V_K]\otimes\QQ, \deg(P)^{-1} P \mapsto w,
\end{equation}
where $w$ and $P$ are the same places of $K$, and under such an identification $n_w = n_P \deg(P)$.

\subsection{Local cohomology classes via duality}
\label{ReductionStep}
 There is a sequence of natural maps on cohomology
\begin{equation}
H^2(F,\mathbb{D}_F) \MapsTo \varprojlim H^2(F,\mathbb{D}_{K,S}) \hookrightarrow \varprojlim \bigoplus_{u\in V_F} H^2(F_u, \mathbb{D}_{K,S}) \hookrightarrow
\prod_{u\in V_F} \varprojlim H^2(F_u, \mathbb{D}_{K,S}),
\end{equation}
where we omitted obvious base change maps. Notice that the inclusion in the middle is given by the local-global principle (see Theorem \ref{LocalGlobalTheorem}).

It follows from Lemma 14.4 in \cite{Kot} and the fact that $\mathbb{D}_{K,S}\times \Spec(F_u)$ is split over $K_w$ for $w\in V_K$ lying above $u$, that $\varprojlim^1 H^1(F_u, \mathbb{D}_{K,S}\times \Spec(F_u))=0$, and therefore
\begin{equation}
\label{LocalLimit}
H^2(F_u, \mathbb{D}_F\times \Spec(F_u)) \simeq \varprojlim_{K,S} H^2(F_u, \mathbb{D}_{K,S}\times \Spec(F_u)).
\end{equation}

Taking a product of the localization functors $\Loc_u$ over the set of places of $F$ defines a unique class in $\prod_{u\in V_F} H^2(F_u, \mathbb{D}_F)$.

 Note that Tate-Nakayama duality over the local field $F_u$ can be rephrased as a classification of Tannakian categories banded by tori. The following proposition generalizes a result of Saavedra Rivano (see VI.Proposition 3.5.3) to the case of Tannakian categories over local fields that are banded by tori that are not necessarily split over the base field.

\begin{propos}
 \label{TannakaClassificationLocal}
 Let $F_u$ be a local field, and let $\mathcal{C}$ be a Tannakian category over $F_u$ banded by an algebraic torus $T$. Then the class of $\mathcal{C}$ in $H^2_{\fppf}(F_u, T)$ is given by an element of $$\Hom(X^{\ast}(T)^{\Gamma_u}, \Br(F))$$
 defined by sending every simple object $E_{\chi}$, for $\chi \in X^\ast(T)^{\Gamma_u}$, to the class of $\End(E_{\chi})$ in the Brauer group of $F_u$.
\end{propos}
\begin{proof} Recall that $H^2$ of tori over local fields can be described fairly easily. Namely, given an algebraic torus $T$ with the character group $X^\ast(T)$, the cup-product gives us bilinear maps
\begin{equation}
\theta_i: H^i(F_u, X^\ast(T)) \times H^{2-i}(F_u, T) \to H^2(F_u, \Gm)\quad(i=0,1,2)
\end{equation}
such that $\theta_0$ defines a duality between the compact group $H^0(F_u, X^\ast(T))^{\wedge}$ (the completion of the abelian group $H^0(F_u, X^\ast(T))$ for the topology given by subgroups of finite index) and the discrete group $H^2(F_u,T)$, in particular, $\theta_0$ induces an isomorphism $H^2(F_u, T)  \simeq \Hom_{\zz}(X^{\ast}(T)^{\Gamma_u},\Br(F_u))$ (see II.\S5 Theorem 6 in \cite{SerreCG}). Namely, any character $\chi \in H^0(F_u, X^\ast(T)) = X^\ast(T)^{\Gamma_u}$ defines an $F_u$-epimorphism from $T$ to $\Gm$, which in its turn induces the morphism on cohomology
\begin{equation}
\chi^{\ast}: H^2(F_u, T) \to H^2(F_u, \Gm),
\end{equation}
such that for any $\alpha\in H^2(F_u, T)$, we have
\begin{equation}
\theta_0(\chi, \alpha) = \chi^\ast(\alpha) \in H^2(F_u, \Gm).
\end{equation}

Let $\mathcal{E}_{\alpha}$ and $\mathcal{E}_{\chi^\ast(\alpha)}$ be the gerbes associated to $\alpha$ and $\chi^\ast(\alpha)$, respectively. It follows that we have a morphism of gerbes
\begin{equation}
\label{functorOnGerbes}
\mathcal{E}_{\alpha} \to \mathcal{E}_{\chi^\ast(\alpha)}
\end{equation}
that is banded by the epimorphism of bands given by $\chi$. Applying Tannakian duality to (\ref{functorOnGerbes}), we obtain a fully faithful functor
\begin{equation}
\Rep(\mathcal{E}_{\chi^{\ast}(\alpha)}) \to \Rep(\mathcal{E}_{\alpha})
\end{equation}
that can be viewed as an embedding of the Tannakian category spanned by the simple object $E_{\chi}$ of $\Rep(\mathcal{E}_{\alpha})$ corresponding to the character $\chi$. To finish the proof, we notice that the class of such subcategory is defined by an element in $\Hom(\mathbb{Z}, \Br(F_u))$ obtained naturally by computing the endomorphisms of $E_{\chi}$, and in our case is given by $\chi^{\ast}(\alpha)$ (see more on Brauer groups and gerbes in Ch.V, 4.2 in \cite{Giraud}).
\end{proof}

We apply Proposition \ref{TannakaClassificationLocal} to the isomorphism (\ref{LocalLimit}) to obtain
\begin{equation}
H^2(F_u, \mathbb{D}_F) \simeq \varprojlim \Hom(X^\ast(\mathbb{D}_{K,S})^{\Gamma_u},\Br(F_u)) =  \Hom(X^\ast(\mathbb{D}_F)^{\Gamma_u}, \Br(F_u)).
\end{equation}

Consequently, the equality of classes in $H^2(F_u, \mathbb{D}_F)$ would follow from the equality of homomorphisms $X^\ast(\mathbb{D}_F)\to \Br(F_u)$.
\subsection{Localization of the Kottwitz category}
\label{LocalMapsDesc}
This section is based on Section 7 in \cite{Kot}, where Kottwitz defined localization maps for finite layers of $B(F,G)$.

The localization of the Kottwitz category is defined via the sequence of maps of character modules
\begin{equation}
X^\ast(\mathbb{D}_F) \to X^\ast(\mathbb{T}_F) \to \QQ,
\end{equation}
that arises as a direct limit of
\begin{equation}
\mu_w^\prime: X^\ast(\mathbb{D}_{K,S}) \xrightarrow{b^\prime} X^\ast(\mathbb{T}_{K,S}) \xrightarrow{\mu_w} \zz = X^\ast(\Gm),
\end{equation}
where the first map is an inclusion and the second is given by $\mu_w(\sum n_v v) = n_w$, where $w$ is a chosen place in $S_K$ lying above $u$.
Recall that $H^2(\Gal(K_w/F_u), \Gm(K_w)) =  \frac{1}{[K_w:F_u]} \zz/ \zz \hookrightarrow \QQ/\zz=H^2(F_u, \Gm)$. Consequently, the canonical class $\alpha(K_w/F_u)$ corresponds to the homomorphism
\begin{equation*}
n_w \mapsto \frac{n_w}{[K_w:F_u]} \pmod{\zz}.
\end{equation*}

Thus the homomorphism $X^\ast(\mathbb{D}_{K,S})\to \Br(F_u)$ is given by the composition $\alpha(K_w/F_u)\circ \mu_w^\prime$. If $a= \sum n_v v \in X^\ast(\mathbb{D}_{K,S})$ then
\begin{equation}
\label{LocalKottwitzClass}
\alpha(K_w/F_u)\circ \mu_w^\prime(a) = \frac{n_w}{[K_w:F_u]}  \pmod{\zz}.
\end{equation}

\subsection{The local class corresponding to $\varphi$-spaces}
Let $\Div^0(K,S)\otimes \QQ$ denote the subgroup of degree $0$ rational divisors with supports in $S_K$, and let $\widetilde{\mathbb{D}_{K,S}}$ be its dual protorus. Assume that $S$ satisfies the condition (K1) (see section \ref{TateKalethaKottwitzCond}). Then local duality applies, since such a module of characters can be represented as a direct limit of $\Div^0(K,S)$ with transition maps being multiplication by natural numbers. Then the same argument as we used to prove the isomorphism (\ref{LocalLimit}) shows that $H^2(F_u, \widetilde{\mathbb{D}_{K,S}}) = \Hom(\Div^0(K,S)^{\Gamma_u}\otimes \QQ, \Br(F_u))$.

Consequently, Theorem \ref{phiSpacesLocal} provides us with an element in $H^2(F_u,\mathbb{D}_F)$, since the rational number $r_w/d_w$ from its statement is a slope of the corresponding isocrystal over $F_u$, the endomorphism algebra of which has the invariant $r_w/d_w \pmod{\zz}$ in $\Br(F_u)$. Strictly speaking, such maps were only defined for the subset of rational divisors of degree $0$ in $K$ that satisfy the primitivity condition imposed on $\varphi$-pairs, but they extend to the whole group of rational divisors via transition maps.
\subsection{Conclusion}
Recall that section \ref{ReductionStep} reduces the comparison of the category of Drinfeld $\varphi$-spaces over $\mathbb{F}_q$ and the category of representations of the Kottwitz gerbe $\Rep(\Kt_F)$ for a global function field $F$ to the comparison of their localizations for every place $u$ in $F$. The latter is reduced to establishing the equality of maps on characters of bands.

Taking into account that the formula (\ref{LocalKottwitzClass}) and the expression for $r_w/d_w$ in Theorem \ref{phiSpacesLocal} are the same modulo the isomorphism described in Section \ref{isomCharacters}, we conclude that the categories in question are equivalent.

\subsection{Isoshtukas with $G$-structure} Let $G$ be a linear algebraic group over $F$. One defines an isoshtuka with $G$-structure as an exact tensor functor from $\Rep(G)$ to the category of isoshtukas. Since the latter is equivalent to $\Rep(\Kt_F)$ we can apply the results of Section \ref{GStructure} to obtain the following.
\begin{theorem}
\label{BFG}
Let $F$ be a global function field, then $B(F,G)$ is isomorphic to the set of isomorphism classes of isoshtukas with $G$-structure.
\end{theorem}

\section{More on fiber functors}
\label{MoreonFiber}
In this section, we prove the following result.
\begin{theorem}
\label{FineFiberFunctors}
For any finite Galois extension $K/\mathbb{Q}$ and any finite set of places of $\mathbb{Q}$ satisfying the conditions of Section \ref{TateKalethaKottwitzCond}, the class $\alpha_{K,S}$ of the Kottwitz gerbe is split by a cyclotomic extension of the form $\mathbb{Q}(\zeta_q)$, where $q$ is a prime number.
\end{theorem}
This is done in two steps: in Section \ref{ReductionToEx}, we reduce the question about the vanishing of restrictions of the class $\alpha_{K,S}$ to a question about the vanishing of the corresponding class in adelic cohomology, which in its turn provides us with a criterion of splitting formulated purely in terms of the local behaviour of the Galois extension to which $\alpha_{K,S}$ is restricted. The second step is performed in Section \ref{ExistenceTheorem2}, where we show that there exist infinitely many integer primes $q$, such that $\mathbb{Q}(\zeta_q)$ has the desired local behaviour.

\subsection{Reduction to an existence theorem}
\label{ReductionToEx}Let $K/F$ be a finite Galois extension of number fields and let $S$ be a finite set of places in $F$, satisfying the usual conditions (see Section \ref{TateKalethaKottwitzCond}). We consider the Tate-Nakayama class $\alpha_{K,S} \in H^2(F,\mathbb{D}_{K,S})$, where $\mathbb{D}_{K,S}$ is the algebraic torus over $F$ introduced in  Definition \ref{KottTorus}.

For any finite Galois extension $L/F$ we can consider the natural restriction map
\begin{equation}
H^2(F,\mathbb{D}_{K,S}) \xrightarrow{\res_L} H^2(L, \mathbb{D}_{K,S} \times_{\Spec(F)} \Spec(L))
\end{equation}
and ask whether  $\res_L(\alpha_{K,S}) = 0$. Since we do not have an explicit description of $\alpha_{K,S}$, but know the images of this class under localization maps, we may consider the following commutative diagram
\begin{equation}
\label{LocalizEx}
\begin{tikzcd}
H^2(F,\mathbb{D}_{K,S}) \arrow{r}{\res_L} \arrow{d}{\loc_F} &  H^2(L, \mathbb{D}_{K,S} \times_{\Spec(F)} \Spec(L)) \arrow{d}{\loc_L}\\
H^2(\mathbb{A}_F, \mathbb{D}_{K,S}) \arrow{r}{\Res_L} &  H^2(\mathbb{A}_L, \mathbb{D}_{K,S} \times_{\Spec(F)} \Spec(L))
\end{tikzcd}
\end{equation}
where $\loc_F$ and $\loc_L$ are the natural maps to adelic cohomology (see Section \ref{AdCoh}), and $\Res$ is the restriction map induced on adelic cohomology.

In order to show that $\res_L(\alpha_{K,S}) = 0$, it is enough to show that $\Res_L(\loc_F(\alpha_{K,S}))=0$ and that $\res_L(\alpha)$ lies in the kernel of $\loc_L$ if and only if it is $0$. We know that the localization map is not injective in general, however we showed in Section \ref{SecPrinciple} that the projective limit of the kernels vanishes.

Assume that $L \cap K = F$, in this case we are basically mimicking the argument from Section \ref{SecPrinciple}, namely we have an isomorphism
\begin{equation}
\mathbb{D}_{K,S} \times_{\Spec(F)} \Spec{L}= \big(\prod_{\dot{v}\in \dot{S}} \Res_{K^{\dot{v}}L/L} \mathbb{G}_{m,K^{\dot{v}}L}\big) / \mathbb{G}_{m,L},
\end{equation}
 moreover, it follows from the vanishing of $H^1(L,\mathbb{D}_{K,S})$ (see Lemma \ref{InjKaletha}) and from the vanishing of $H^3(L,\Gm)$ (see Corollary I.4.21 in \cite{ADT}) that $H^2(L, \mathbb{D}_{K,S}) = (\oplus_{\dot{v}\in \dot{S}} \Br(K^{\dot{v}}L))/\Br(L)$. Again, applying the snake lemma to the diagram similar to (\ref{LocGlobSnake}), we obtain that the kernel of the localization map $\loc_L$ is a finite subgroup of $\mathbb{Q}/\mathbb{Z}$ of order $\gcd_{\dot{v}\in\dot{S}} \deg(K^{\dot{v}}L/L)$. By disjointness of $L$ and $K$ over $F$, the latter is equal to $\gcd_{\dot{v}\in\dot{S}} \deg(K^{\dot{v}}/F)$.
 
 Note that by the local-global principle, there exists a finite Galois extension $E/K/F$ and a finite set of places $S^\prime$ in $E$ lying above $S$, such that if $\res_L(\alpha_{S^\prime, E})$ is in the kernel of the localization map $H^2(L,\mathbb{D}_{S^\prime,E}) \to H^2(\mathbb{A}_L,\mathbb{D}_{S^\prime,E})$, then  $\res_L (\alpha_{K,S}) = \pi_{S^\prime,E}(\res_L(\alpha_{S^\prime, E})) = 0$.

 Consequently, after replacing $(K,S)$ with $(E,S^\prime)$ in the diagram (\ref{LocalizEx}), we see that it is enough to find a finite Galois extension $L$ linearly disjoint from $E$ and $K$ that satisfies the following condition:
\begin{itemize}
\item $\Res_L(\alpha_{S^\prime,E}) = 0 \in H^2(\mathbb{A}_L,\mathbb{D}_{S^\prime,E})$.
\end{itemize}

By construction of $\alpha_{S^\prime, E}$, the class $\loc_F(\alpha_{S^\prime,E})$ is the image of the semi-local canonical class $\alpha_2(S^\prime,E)$ under the map $H^2(\mathbb{A}_F,\mathbb{T}_{S^\prime,E}) \to H^2(\mathbb{A}_F,\mathbb{D}_{S^\prime, E})$ (see Section \ref{LocalComponents}). Therefore, we are reduced to proving that $\Res_L(\alpha_2(S^\prime,E))=0$. Using the fact that $\alpha_2(S^\prime,E)$ is defined very explicitly (see Section \ref{SemiLocal}, in particular, the isomorphism (\ref{IsomHJK})), we obtain that $\Res_L(\alpha_{2}(S^\prime,E))$ vanishes if $\alpha(E_{\dot{v}}/F_v) \in \Br(E_{\dot{v}}/F_v)$ is killed by the map
\begin{equation}
H^2((E^{\dot{v}})_v,\Gm)=\Br(F_v) \xrightarrow{\res_{L_w}} \oplus_{\tilde{w}\in V_{E^{\dot{v}}L}, \tilde{w}|w} H^2((E^{\dot{v}}L)_{\tilde{w}}, \Gm) = \oplus_{\tilde{w}|w} \Br((E^{\dot{v}}L)_{\tilde{w}})
\end{equation}
where $w$ is a place in $L$, such that $w\mid v$. On the left-hand side we used that $E^{\dot{v}}$ is the decomposition field of a place lying over $v$. Consider an induced map $\Br(F_v) \to \Br((E^{\dot{v}}L)_{\tilde{w}})$, where $\tilde{w}$ is as above. Since it comes from the restriction map $\res_{L_w}$, it follows from local class field theory that this map can be identified with the multiplication by $\deg((E^{\dot{v}}L)_{\tilde{w}}/F_{v})$ on $\mathbb{Q}/\mathbb{Z}$.

In order to compute the degree of the extension of local fields $(E^{\dot{v}}L)_{\tilde{w}}/F_{v}$, we notice that $L\otimes_FE^{\dot{v}}\otimes_F  \otimes F_v = (\prod_{w|v, w\in V_{L}}L_w)^g$, where $g$ is the number of places of $E$ lying over $v$, therefore the desired degree is given by $\deg(L_w/F_v)$.

Combining all reduction steps, we see that any extension $L/F$ such that
\begin{enumerate}
\item $L\cap E = F$,
\item for any $w\in V_L$ lying above $v\in S$, the local degree $\deg(L_w/F_v)$ is divisible by $\deg(E_{v^\prime}/F_v$), where $v^\prime \in S_E$ is any lift of $v$,
\end{enumerate}
splits the canonical class $\alpha_{K,S}$.

In the following sections, we show that infinitely many extensions $\mathbb{Q}(\zeta_q)$, where $q$ is a prime, satisfy the second condition, therefore, the first condition poses no problem. We also notice that $\mathbb{Q}(\zeta_q)$ is a purely imaginary extension of $\mathbb{Q}$, thus the local degrees of $\mathbb{Q}(\zeta_q)/\mathbb{Q}$  at the infinite primes are all equal to $2$. This way we may restrict ourselves to finite primes of $S$.
%Indeed, we have $\deg(E L/ K^{\dot{v}}L) = \deg(E/K^{\dot{v}})=|\res_L(B_{S^\prime,E})|$, where the latter denotes the order of the kernel of the localization map
%\begin{equation*}
%\loc_L: H^2(L,\mathbb{D}_{S^\prime,E})\to H^2(\mathbb{A}_L, \mathbb{D}_{S^\prime,E})
%\end{equation*}

\subsection{Auxiliary result: Kummer extensions of cyclotomic fields}

First, we recall a result from theory of Kummer extensions of cyclotomic number fields that will be crucial in Section \ref{ExistenceTheorem2}.
\begin{theorem}[Theorem 1 in \cite{PerS}]
\label{PerSTheorem}
Let $K$ be a number field and let $G$ be a finitely generated subgroup of $K^{\times}$. Consider a Kummer extension $K_{n,n} = K(\zeta_n,\sqrt[n]{G})$ of the $n$th cyclotomic extension $K_n=K(\zeta_n)$. If $G$ is of strictly positive rank $s$, then there is an integer $c\geq 1$ (depending only on $K$ and $G$) such that for all integers $n\geq 1$ we have
\begin{equation}
\frac{n^s}{\Gal(K_{n,n}:K_n)} \mid c
\end{equation}
where $\frac{n^s}{\Gal(K_{n,n}:K_n)}$ is a positive integer.
\end{theorem}

%Let $c$ be the constant corresponding to the case where $K=\mathbb{Q}$ and $G$ is generated by rational prime number $p_1,\ldots,p_s$.
\subsection{Existence theorem}
\label{ExistenceTheorem2} The proof of the following theorem will occupy the rest of this section.
\begin{theorem}
Let $S=\{p_1,\ldots ,p_s\}$ be a fixed finite set of primes in $\mathbb{Q}$. Then there exists a constant $c\in\zz$ (that depends only on $S$) such that for any integer $r$ satisfying $c\mid r$, there exists infinitely many prime numbers $q \notin S$ with the following properties:
\begin{itemize}
\label{ExThSt}
\item[($\star$)] $q\equiv 1\pmod r$, $\gcd(\frac{q-1}{r}, r) = c$.
\item[($\star\star$)] let $L$ denote $\mathbb{Q}(\zeta_q)$, then for every $p_i\in S$, and for any prime ideal $P_i$ in $\mathcal{O}_L$ lying above $p_i$ the degree of the local extension $L_{P_i}/\mathbb{Q}_{p_i}$ is divisible by $\frac{r}{c}$.
\end{itemize}
\end{theorem}
\begin{comment}Let $r=\lcm(l_1,\ldots,\l_s)$, then it follows from the fact that $\Gal(\mathbb{Q}(\zeta_q)/\mathbb{Q}) = \mathbb{Z}/(q-1)\mathbb{Z}$ that $r$ must divide $q-1$.  Consequently, we can restrict ourselves to the subfield $L_r\subset L$ defined as a fixed subfield of the subgroup of $\Gal(L/\mathbb{Q})$ generated by $\sigma^{r}$, where $\sigma$ is a generator of $\Gal(L/\mathbb{Q})$. If we assume further that $\gcd(r, \frac{q-1}{r})=1$, then by the Chinese remainder theorem
\begin{equation*}
\Gal(L/\mathbb{Q}) = \sfrac{\mathbb{Z}}{(r)}\oplus \sfrac{\mathbb{Z}}{(\frac{q-1}{r})},
\end{equation*}
where the first summand corresponds to $\Gal(L_r/\mathbb{Q})$. Since local degrees are multiplicative in towers of Galois extensions, we see that under our assumptions the conditions ($\star$) are satisfied by $L$ and $L_r$ simultaneously.
\end{comment}
To simplify notation, we denote $r_{\ast} = \frac{r}{c}$, where $c$ is an integral divisor of $r$. We will specify this constant below.

\begin{comment}It follows from Chebotarev denstity theorem, that there are infinitely many prime numbers $q\in\mathbb{Z}$ satisfying the conditions $r\mid q-1$ and $\gcd(r^2, q-1) = rc$. More precisely, consider the tower of extensions $\mathbb{Q}\subset\mathbb{Q}(\zeta_r)\subset\mathbb{Q}(\zeta_{r^2})$. Then $r\mid q-1$ is equivalent to the condition that $\Frob_q=1$ in $\Gal(\mathbb{Q}(\zeta_r)/\mathbb{Q})$, and the second condition is equivalent to the statement that $\Frob_q$ generates the cyclic subgroup $\sfrac{\mathbb{Z}}{(r_{\ast})}$ in $\Gal(\mathbb{Q}(\zeta_{r^2})/\mathbb{Q}(\zeta_r))$. Chebotarev density theorem implies that there are infinitely many primes in the conjugacy class of the generator of $\sfrac{\mathbb{Z}}{(r_{\ast})}$.
\end{comment}

We rephrase the condition ($\star\star$) as follows.
Note that since $q\notin S$, every extension $L_{P_i}/\mathbb{Q}_{p_i}$ as above is unramified of degree $f_i$, which is computed as the order of $p_i$ in $(\sfrac{\mathbb{Z}}{q\mathbb{Z}})^{\times}$, more concretely, $f_i$ is the smallest natural number such that $p_i^{f_i} \equiv 1 \pmod q$ (see Theorem 2.13 in \cite{Wash}). Thus, the degree of  $L_{P_i}/\mathbb{Q}_{p_i}$ is divisible by $r_{\ast}$ if and only if $f_i$ is divisible by $r_{\ast}$.

Focusing on the prime $p=p_i$, we see that its order in $(\sfrac{\mathbb{Z}}{q\mathbb{Z}})^{\times}$ can be computed by analysing the extension $\mathbb{Q}(\zeta_r,p^{1/r})/\mathbb{Q}(\zeta_r)$, in which $q$ is unramified. Since $r\mid q-1$, the degree of the local extension over $q$ is equal to the degree of the extension $\mathbb{F}_q(p^{1/r})/\mathbb{F}_q$. Using Kummer theory (see Chapter IV, Theorem 3.3 in \cite{Neu2}), we obtain that $\Frob_q\in \Gal(\mathbb{F}_q(p^{1/r})/\mathbb{F}_q)$ has order equal to the order of $p^{\frac{(q-1)}{r}}$ in $\mathbb{F}_q^\times$. More precisely, in this case $\mathbb{F}_q^\times/(\mathbb{F}_q^{\times})^{{r}} \simeq \sfrac{\mathbb{Z}}{(r)}$ is isomorphic to the group of characters $\Hom(\Gal(\overline{\mathbb{F}}_q/\mathbb{F}_q), \mu_{r})$ via the map
\begin{equation*}
p \mapsto \chi_{p}(\Frob_q) = \frac{\Frob_q(p^{1/r})}{p^{1/r}}= p^{\frac{q-1}{r}}\in \mu_{r}.
\end{equation*}
Thus, the order of $p$ in $\mathbb{F}_q^\times/(\mathbb{F}_q^{\times})^{r}$ is equal to the order of $\chi_{p}(\Frob_q)$, and the latter is equal to the order of $\Frob_q \in \Gal(\mathbb{F}_q(p^{1/r})/\mathbb{F}_q)$. In particular, the order of $\Frob_q\in \Gal(\mathbb{F}_q(p^{1/r})/\mathbb{F}_q)$ divides the order of $p$ in $\mathbb{F}_q^{\times}$.

In order to combine the local conditions for different $p_i$'s, we consider the following tower of extensions

\begin{center}%
 \begin{tikzpicture}%
    \node (Q01) at (0,0) {$\mathbb{Q}$};
    \node (Q02) at (0,1) {$\mathbb{Q}_{r}:=\mathbb{Q}(\zeta_r)$};
    \node (Q1) at (0,2) {$E$};
    \node (Q2) at (3,3) {$\mathbb{Q}_{r^2}:=\mathbb{Q}(\zeta_{r^2})$};
    \node (Q3) at (0,4.5) {$\mathbb{Q}_{r^2,r}:=\mathbb{Q}(\zeta_{r^2},\sqrt[r]{p_1},\ldots,\sqrt[r]{p_s})$};
    \node (Q4) at (-3,3) {$\mathbb{Q}_{r,r}:=\mathbb{Q}(\zeta_{r},\sqrt[r]{p_1},\ldots,\sqrt[r]{p_s})$};

    \draw (Q01)--(Q02);
    \draw (Q02)--(Q1);
    \draw (Q1)--(Q2);
    \draw (Q1)--(Q4);
    \draw (Q3)--(Q4);
    \draw (Q2)--(Q3);
\end{tikzpicture}%
\end{center}%
where $E$ denotes the intersection of $\mathbb{Q}_{r,r}$ and $\mathbb{Q}_{r^2}$.

We claim that there are infinitely many rational primes $q$, such that $\Frob_q\in\Gal(\mathbb{Q}_{r^2,r}/\mathbb{Q})$ satisfies the following conditions:
\begin{enumerate}
\item $\Frob_q = 1 \in \Gal(\mathbb{Q}_r/\mathbb{Q})$.
\item $\Frob_q$ generates the subgroup $\sfrac{c\mathbb{Z}}{r\mathbb{Z}}\subseteq\Gal(\mathbb{Q}_{r^2}/\mathbb{Q}_r)$.
\item for any $1\leq i \leq s$ and for every projection $\Gal(\mathbb{Q}_{r,r}/\mathbb{Q}_r)\to \Gal(\mathbb{Q}_r(\sqrt[r]{p_i})/\mathbb{Q}_r)$ the image of $\Frob_q$ in the latter Galois group is at least of order $\frac{r}{c}$.
\end{enumerate}
Note that it follows from the discussion above that the third condition implies that the orders of $p_i$'s in $\mathbb{F}_q^\times$ are divisible by $\frac{r}{c}$.

First, we show that $H=\Gal(\mathbb{Q}_{r,r}/E)$ contains a subgroup $(\sfrac{c\mathbb{Z}}{r\mathbb{Z}})^s$.  Notice that by definition of $E$, the group $H$ is isomorphic to $\Gal(\mathbb{Q}_{r^2,r}/\mathbb{Q}_{r^2})$. Letting $n=r^2$ in Theorem \ref{PerSTheorem}, we obtain
\begin{equation}
\label{kummerlim}
\frac{r^{2s}}{\deg(\mathbb{Q}_{r^2,r^2}/\mathbb{Q}_{r^2})} = \frac{r^s}{\deg(\mathbb{Q}_{r^2,r}/\mathbb{Q}_{r^2})}\cdot \frac{r^s}{\deg(\mathbb{Q}_{r^2,r^2}/\mathbb{Q}_{r^2,r})} \mid c_S
\end{equation}
Since both factors in (\ref{kummerlim}) are integers due to the fact that the corresponding extensions are Kummer extensions, we see that $\frac{r^s}{\deg(\mathbb{Q}_{r^2,r}/\mathbb{Q}_{r^2})}$ is an integer dividing $c_S$. On the other hand, $H$ is a subgroup of $(\sfrac{\mathbb{Z}}{(r)})^s$, thus any generator of the preimage of any factor $\sfrac{\mathbb{Z}}{(r)}$ is at least of order $r/c_S$, hence $(\sfrac{c_S\mathbb{Z}}{r\mathbb{Z}})^s\subseteq H$. Finally, we set
\begin{equation*}
c = c_S.
\end{equation*}

Note that since $\mathbb{Q}_{r^2,r}$, $\mathbb{Q}_{r,r}$, and $\mathbb{Q}_{r^2}$ are abelian extensions over $\mathbb{Q}_r$, the Galois group of $\mathbb{Q}_{r^2,r}$ over $\mathbb{Q}_r$ is a fiber product of $\Gal(\mathbb{Q}_{r,r}/\mathbb{Q}_r)$ and of $\Gal(\mathbb{Q}_{r^2}/\mathbb{Q}_r)$  over $\Gal(E/\mathbb{Q}_r)$. Next, we consider the diagram
\begin{equation}
\begin{tikzcd}
1 \arrow{r} & H \arrow{r} \arrow[equal]{d}& \Gal(\mathbb{Q}_{r^2,r}/\mathbb{Q}_r) \arrow{r} \arrow{d}& \Gal(\mathbb{Q}_{r^2}/\mathbb{Q}_r)=\sfrac{\mathbb{Z}}{(r)} \arrow{r} \arrow{d}& 1 \\
1 \arrow{r} & H \arrow{r} & \Gal(\mathbb{Q}_{r,r}/\mathbb{Q}_r) \arrow{r} & \Gal(E/\mathbb{Q}_r) \arrow{r} & 1
\end{tikzcd}
\end{equation}

Let $\sigma$ denote a generator of $\sfrac{c \mathbb{Z}}{r\mathbb{Z}} \subseteq \Gal(\mathbb{Q}_{r^2}/E)$, and let $x = (c,\ldots,c)\in(\sfrac{c\mathbb{Z}}{r\mathbb{Z}})^s \subseteq H$, then both $\sigma$ and $x$ project to $1$ in $\Gal(E/\mathbb{Q}_r)$, since $\Gal(E/\mathbb{Q}_r)$ is a quotient of $\sfrac{\mathbb{Z}}{(r)}$ and its order divides $c$. It follows that $(x,\sigma) \in \Gal(\mathbb{Q}_{r^2,r}/\mathbb{Q}_r)$. Applying the Chebotarev density theorem to the conjugacy class of $(x,\sigma)$, we see that there are infinitely many primes $q\in \mathbb{Q}_r$ such that $\Frob_q$ lies in the chosen conjugacy class. Notice that by construction $\Frob_q = 1$ in $\Gal(\mathbb{Q}_r/\mathbb{Q})$.

In order to check the third condition, we need to show that for every $i=1,\ldots,s$, the Frobenius element $\Frob_{q,i}$ in $\Gal(\mathbb{Q}_{r}(\sqrt[r]{p_i})/\mathbb{Q}_r)$ has the correct order. Since $q$ is unramified in $\mathbb{Q}_{r^2,r}$, we see that  $\Frob_{q,i} = \pi_i(\Frob_q)$, where $\pi_i$ is defined as the sequence of natural projections
\begin{equation}
\pi_i: \Gal(\mathbb{Q}_{r^2,r}/\mathbb{Q}_r) \to\Gal(\mathbb{Q}_{r,r}/\mathbb{Q}_r) \to \Gal(\mathbb{Q}_{r}(\sqrt[r]{p_i})/\mathbb{Q}_r).
\end{equation}
By construction, $\pi_i(\Frob_q)$ belongs to the conjugacy class of the element
\begin{equation}
c \in \Gal(\mathbb{Q}_{r}(\sqrt[r]{p_i})/\mathbb{Q}_r) \subseteq \sfrac{\mathbb{Z}}{r \mathbb{Z}},
\end{equation}
therefore, the order of $\Frob_{q,i}$ is divisible by $\frac{r}{c}$, which completes the proof of the theorem.

\subsection{End of the proof of Theorem \ref{FineFiberFunctors}} We apply Theorem \ref{ExThSt} to the conditions obtained in the end of Section \ref{ReductionToEx}. Namely, we let $r$ be divisible by the least common multiple of local degrees $l = \lcm_{v\in S}(\deg(E_{v^\prime}/\mathbb{Q}_v))$, where $v\in S$ and $v^\prime\in S_E$ is any lift of $v$. It follows that there exists infinitely many primes $q\notin S$, such that local degrees of $\mathbb{Q}(\zeta_q)$ are divisible by an integer $r/c$, where $c$ does not depend on $r$. Therefore, we can achieve that $r/c$ is divisible by $l$ by simply enlarging $r$ multiplicatively.

\section{Motives over $\mathbb{F}$ and $\Rep(\Kt_\mathbb{Q})$}
\label{ScholzeTate}
Let $\mathbb{F}$ be the algebraic closure of the finite field $\mathbb{F}_p$.  In this section we show that $\Rep(\Kt_\mathbb{Q})$ has a full subcategory that under the assumption of Tate's conjecture over $\mathbb{F}$ can be identified with $\Mot(\mathbb{F})$ the category of motives over $\mathbb{F}$. The idea of describing motives in terms of gerbes and non-abelian cohomology goes back to Grothendieck and was pushed further in the case of $\Mot(\mathbb{F})$ in the work of Langlands and Rapoport (\cite{LR}), where an explicit description of the class of $\Mot(\mathbb{F})$ appeared. It was however corrected and simplified significantly in the following works of Milne (see \cite{Milne1994}, \cite{MilneGerbes}).

\subsection{Weil numbers and their germs}
\label{WeilNumbersGerms}
Let $q=p^n$. Recall that an algebraic number $\pi \in \mathbb{Q}^{\text{al}}$ is said to be a Weil $q$-number of weight $m$ if
\begin{itemize}
\item for every embedding $\rho: \mathbb{Q}[\pi] \hookrightarrow \mathbb{C}$, $\rho(\pi)\cdot\overline{\rho(\pi)} = q^m$.
\item for some $l$, $q^l \pi$ is an algebraic integer.
\end{itemize}
The set of all Weil $q$-numbers in $\mathbb{Q}^{\text{al}}$ is an abelian group denoted by $W(q)$. If $q^\prime = p^{n^\prime}$ and $n\mid n^\prime$, we have a group homomorphism $W(q) \to W(q^\prime)$: $\pi \mapsto \pi^{n^\prime/n}$. Define the group of germs of Weil numbers
\begin{equation}
W = \varinjlim W(q).
\end{equation}

There is an action of the absolute Galois group of $\mathbb{Q}$ on $W$, and $P$ is defined as a protorus over $\mathbb{Q}$ such that
\begin{equation}
X^{\ast}(P) = W.
\end{equation}
In order to represent $P$ as a projective limit of algebraic tori, we rewrite $W$ as a colimit of finitely generated Galois modules. Namely, for any CM-field $K/\mathbb{Q}$, we define $W^K$ as a subset of $\pi \in W$ having a representative in $K$ and such that for any place $v\mid p$ in $K$:
\begin{equation}
\label{germCondition}
\frac{\ord_v(\pi)}{\ord_{v}(p^n)} \cdot [K_v:\mathbb{Q}_p] \in \mathbb{Z}.
\end{equation}

Define $P^K$ as a group of multiplicative type defined by $\Gal(\mathbb{Q}^{\text{al}}/\mathbb{Q})$-module $W^K$.
Then $W = \varinjlim_{K} W^K$ and $P = \varprojlim_K P^{K}$, where the limit is being taken over the directed set of finite CM-extensions of $\mathbb{Q}$.

Let $K^{+}$ denote the maximal totally real subfield of $K$, and let $V_{p,K}$ and $V_{p,K^{+}}$ denote the set of $p$-adic primes in $K$ and $K^{+}$ respectively. Then there is an exact sequence of Galois modules
\begin{equation}
\label{MilneWeil}
0 \to W^K \xrightarrow{\alpha} \bigoplus_{v \in V_{p,K}} \mathbb{Z}\cdot v \bigoplus \mathbb{Z} \xrightarrow{\beta} \bigoplus_{w\in V_{p,K^{+}}}\mathbb{Z}\cdot w \to 0,
\end{equation}
where $\alpha$ is the map
\begin{equation}
\pi \mapsto (\sum_{v \in V_{p,K}} \frac{\ord_v(\pi)}{\ord_v{p^n}}\cdot [K_v:\mathbb{Q}_p], wt(\pi)),
\end{equation}
and $\beta$ is given by the formula
\begin{equation}
(\sum_{v \in V_{p,K}} n_v v, m) \mapsto \sum_{v \in V_{p,K}} n_v\mid_{V_{p,K^{+}}} - [K_v:\mathbb{Q}_p] m \sum {w\in V_{p,K^{+}}}
\end{equation}

Let $K^\prime$ be a CM-field containing $K$. Then $W^K\subset W^{K^\prime}$, and the transition map can be described naturally using the exact sequence (\ref{MilneWeil}). Namely, the transition map is induced by the map
\begin{equation}
\label{MilneWeilTrMaps}
v\mapsto \sum_{v^\prime\mid v}  [K^{\prime}_{v^\prime}:K_v]\cdot v^\prime,
\end{equation}
when $v\mid p$, and the identity map on $\mathbb{Z}$.

The short exact sequence of Galois modules (\ref{MilneWeil}) induces the short exact sequence of the corresponding tori and provides a very useful tool for analysis of the cohomology of $P^K$. The following results are due to Milne (see Section 3 in \cite{MilneGerbes}).
\begin{propos}
\label{MilneCohom}
 Let $K$ be a CM-field, and let $P^K$ be a torus attached to the Galois module $W^K$, then
\begin{itemize}
\item $\varprojlim^i_K H^1(\mathbb{Q}, P^K) = 0$, for $i=0,1$.
\item $H^2(\mathbb{Q},P) = \varprojlim_K H^2(\mathbb{Q}, P^K)$.
\item $H^2(\mathbb{Q},P^K) \to H^2(\mathbb{A}, P^K)$ is injective.
\item $H^2(\mathbb{Q},P) \to H^2(\mathbb{A},P)$ is injective.
\end{itemize}
\end{propos}

\subsection{$W^K$ and the characters of the Kottwitz protorus}
One can also realize $W^K$ as a submodule of the character module of the Kottwitz pro-torus $\mathbb{D}_{K/\mathbb{Q}}$.
\begin{propos}
Let $K$ be a finite CM-extension of $\mathbb{Q}$. Then there is an injective homomorphism of $\Gal(K/\mathbb{Q})$ modules
\begin{equation}
\omega_K: W^K\hookrightarrow X^{\ast}(\mathbb{D}_{K/F}).
\end{equation}
\end{propos}
\begin{proof}
Let $\pi\in W^K$ be a representative of a Weil $p^n$-Number of weight $m$. We may apply the product formula to $\pi$, namely
\begin{equation}
\prod_{v\in V_K} \norm{\pi}_v = 1.
\end{equation}
Moreover, for any non-archimedean $v\notin V_{p,K}$, we have $\norm{\pi}_v = 1$. Notice that for any $v \mid p$, we have $\norm{\pi}_v = p^{-f_v\ord_v(\pi)}$, where $f_v$ is the degree of the extension of the residue field of $K_v$ over $\mathbb{Q}_p$. The latter can be expressed as $[K_v:\mathbb{Q}_p]/e_v$, where $e_v$ is the ramification index of $v$ in $K$ that can be computed as $\ord_v(p)$. Since $K$ is a CM-field, every archimedean place $\infty^{\prime}\mid \infty$ is in fact complex, and by definition of Weil $p^n$-number of weight $m$, we have $\norm{\pi}_{\infty^\prime} = (p^n)^m$. Combining these observations, we deduce
\begin{equation}
\label{ProdFormulaWeil}
\sum_{v\in V_{p,K}} \frac{\ord_v{\pi}}{\ord_v(p)}\cdot [K_v:\mathbb{Q}_p] -\sum_{\infty^\prime\in V_{\infty,K}} nm = 0.
\end{equation}
Notice that by definition of $W^K$ and from the fact that $\pi \in W^K$, the equality (\ref{ProdFormulaWeil}) is integrally divisible by $n$ (see (\ref{germCondition})). Define
\begin{equation}
i_K: W^K \to \mathbb{Z}[V_K],\quad \pi\mapsto
\sum_{v\in V_{p,K}} \frac{\ord_v{\pi}}{\ord_v(p^n)}\cdot [K_v:\mathbb{Q}_p]\cdot v -\sum_{\infty^\prime\in V_{\infty,K}} m\cdot\infty^\prime = 0.
\end{equation}

Since $W^K$ is torsion-free, this map is injective. Moreover, it follows from (\ref{ProdFormulaWeil}) that $i_K$ induces an injective map $W^K \to X^{\ast}(\mathbb{D}_{K/\mathbb{Q}})$.
\end{proof}

In order to construct the morphism $\varinjlim_{K} W^K \to X^\ast(\mathbb{D}_\mathbb{Q})$, we need to check that the morphisms $\omega_K$ are compatible with the transition maps $W^K \to W^{K^\prime}$ and $X^{\ast}(\mathbb{D}_{K/\mathbb{Q}})\to X^{\ast}(\mathbb{D}_{K^\prime/\mathbb{Q}})$, where $K^\prime/\mathbb{Q}$ is a CM-extension containing $K$. Recall, that he transition maps for the character modules of Kottwitz protori are induced by the maps (\ref{TransMaps}), that is by the same formulae as (\ref{MilneWeilTrMaps}). We denote by $\omega$ the injective homomorphism of $\Gal(\mathbb{Q}^{\text{al}}/\mathbb{Q})$-modules $W \to X^\ast(\mathbb{D}_{\mathbb{Q}})$ induced by $\omega_K$ for varying CM-fields $K$.

\subsection{Image of the Kottwitz class under $\omega$} Let $\omega: W \hookrightarrow X^{\ast}(\mathbb{D}_{\mathbb{Q}})$ be the homomorphism defined above. By abuse of notation, we also denote by $\omega$ the epimorphism of protori $\mathbb{D}_{\mathbb{Q}}\to P$. Because $H^2(\mathbb{Q}, \mathbb{D}_{\mathbb{Q}}) = \varprojlim H^2(\mathbb{Q}, \mathbb{D}_{K,S})$ and $H^2(\mathbb{Q}, P) = \varprojlim H^2(\mathbb{Q}, P^K)$, there is a well-defined morphism
\begin{equation}
\omega^\ast: H^2(\mathbb{Q}, \mathbb{D}_{\mathbb{Q}}) \to H^2(\mathbb{Q}, P).
\end{equation}

We are interested in computing the class $\omega^\ast(\alpha_{\mathbb{Q}})$, where $\alpha_{\mathbb{Q}}$ denotes the class of the Kottwitz gerbe in $H^2(\mathbb{Q},\mathbb{D}_\mathbb{Q})$. By Proposition \ref{MilneCohom}, $H^2(\mathbb{Q}, P)\hookrightarrow H^2(\mathbb{A},\ P)$, therefore it suffices to describe the image of the given class locally.

\begin{propos}
\label{KtWeil}
The image of $\omega^\ast(\alpha_\mathbb{Q})$ in $H^2(\mathbb{A}, P)$ is given by the class $(0_{\mathbb{A}^{\{p,\infty\}}}, \alpha_{p}, \alpha_{\mathbb{R}})$. Where $\alpha_{p}$ and $\alpha_{\mathbb{R}}$ are the canonical local classes.
\end{propos}

\begin{proof}
Let $\Gamma_l$ denote the absolute Galois group of $\mathbb{Q}_l$. Recall that since $\lim^1H^1(\mathbb{Q}_l, P) = 0$, we may apply the arguments of Section \ref{ReductionStep}, to describe local cohomology groups of the Weil protorus, namely
\begin{equation}
H^2(\mathbb{Q}_l, P)\simeq \Hom_{\mathbb{Z}}(W^{\Gamma_l},\mathbb{Q}/\mathbb{Z}),
\end{equation}
where $l$ any non-archimedean prime.

Recall also that the map $\omega: W \hookrightarrow X^\ast(\mathbb{D}_{\mathbb{Q}})$ induces the map of $\Gamma_l$-modules
\begin{equation}
\label{WeilClassLocal}
\loc_l: W \xrightarrow{\omega} X^\ast(\mathbb{D}_{\mathbb{Q}}) \hookrightarrow X^\ast(\mathbb{T}_{\mathbb{Q}}) \xrightarrow{\mu_l} \mathbb{Q},
\end{equation}
where $\mu_l$ is defined in Section \ref{LocalMapsDesc}. Namely, $\mu_l$ sends a sum of valuations with rational coefficients to the coefficient corresponding to the fixed $l$-adic prime of $\mathbb{Q}^{\text{al}}$.

Therefore, for any $l\neq p$, we have that the image of $W$ in $\mathbb{Q}$ of the composition $\loc_l$ in (\ref{WeilClassLocal}) is $0$. Hence, the corresponding class in $H^2(\mathbb{Q}_l, P)$ is trivial.

If $l=p$, we see that the composition $W^{\Gamma_p}\to W \xrightarrow{\loc_p} \mathbb{Q}\xrightarrow{[K_v:\mathbb{Q}_p]^{-1}}\mathbb{Q}$ is surjective. It follows from the fact that for any CM-field $K$ and $\pi\in W^K$ we have $\loc_v(\pi) = \frac{\ord_v(\pi)}{\ord_v(p^n)}\cdot [K_v:\mathbb{Q}_p] \in \mathbb{Z}$, where $v$ is the $p$-adic place of $K$ induced by the fixed $p$-adic place of $\mathbb{Q}^{\text{al}}$. If $\pi$ is moreover $\Gamma_v$-invariant, where $\Gamma_v = \Gal(K_v/\mathbb{Q}_p)$ is the decomposition group of $v$, then we deduce from the product formula and equality (\ref{ProdFormulaWeil}) that $\ord_v(\pi)/\ord_v(p^n)$ takes value $\frac{m}{2}$, where $m$ is the weight of $\pi$.  We obtain that the colimit of the cokernels of the maps
\begin{equation}
\label{WeilComponents}
(W^{K})^{\Gamma_v} \xrightarrow{\loc_v} \mathbb{Z} \xrightarrow{[K_v:\mathbb{Q}_p]^{-1}} \frac{1}{2\cdot[K_v:\mathbb{Q}_p]}\mathbb{Z},\quad \pi \mapsto \frac{\ord_v(\pi)}{\ord_v(p^n)}
\end{equation}
vanishes, since their orders are bounded by $[K_v:\mathbb{Q}_p]$, and the transition maps for $K\subset K^\prime$, where $K^\prime$ is a finite CM-field over $\mathbb{Q}$, are given by multiplication by the local degrees $[K^{\prime}_{v^\prime}:K_v]$, hence kill the cokernels eventually.

Consequently, the map $\Hom(\mathbb{Q},\mathbb{Q}/\mathbb{Z}) \to \Hom(W^{\Gamma_p},\mathbb{Q}/\mathbb{Z})$ is injective and maps the class of the natural projection $\mathbb{Q} \to \mathbb{Q}/\mathbb{Z}$, that is the class $\alpha_p$, to the image of $\omega^{\ast}(\alpha_{\mathbb{Q}})$ in $H^2(\mathbb{Q}_p, P)$.

To complete the proof, we notice that when $l = \infty$ is the arhimedean place, then $\loc_{\infty}(\pi) = -m$. Hence the image of $\omega^{\ast}(\alpha_{\mathbb{Q}})$ in $H^2(\mathbb{R},P)$ coincides with the image of the projection $\mathbb{Z}\to\mathbb{Z}/2\mathbb{Z}$ under the map $\Hom(\mathbb{Z},\Br(\mathbb{R})) \to \Hom(W^{\Gamma_{\infty}},\Br(\mathbb{R}))$, where $\Gamma_{\infty} \simeq \Gal(\mathbb{C}/\mathbb{R})$.
\end{proof}

\subsection{Conjectural relation to $\Mot(\mathbb{F})$} Let $\zeta(V,s)$ denote the zeta function of a variety $V$ over $\mathbb{F}_q$. Let $A^r(V)$ denote the quotient of the space of algebraic cycles of codimension $r$ on $V$ by the subspace of cycles numerically equivalent to $0$. In this section, we assume the full Tate conjecture.
\begin{conj}[Tate conjecture]
\label{TateConjecture}
For all smooth projective varieties over $\mathbb{F}_q$ and $r\geq0$, the dimension of $A^r(V)$ is equal to the order of the pole of $\zeta(V,s)$ at $s = r$.
\end{conj}

Under this assumption the category of motives over $\mathbb{F}$ can be described fairly explicitly. Namely, it follows from theorem of Jannsen \cite{Jann}, that $\Mot(\mathbb{F})$ is a semi-simple Tannakian category over $\mathbb{Q}$. Moreover, the classes of simple objects correspond to the orbits of the action $\Gal(\mathbb{Q}^{\text{al}}/\mathbb{Q})$ on  $W$, the abelian group of germs of Weil numbers defined in Section \ref{WeilNumbersGerms}. Applying Tannakian duality, we see that $\Mot(\mathbb{F})$ can be recovered up to equivalence from the class it defines in $H^2(\mathbb{Q},P)$ ($\fpqc$-cohomology). Using the fact that $H^2(\mathbb{Q},P)\hookrightarrow H^2(\mathbb{A}, P)$, we deduce that it suffices to describe the images of the class of $\Mot(\mathbb{F})$ in cohomology groups of completions of $\mathbb{Q}$ at all places. One can apply Proposition \ref{TannakaClassificationLocal} to compute the local classes via isocrystals. Namely, it turns out that the localization functor (see Theorem 1.18 in \cite{Milne1994} for the definition)
\begin{equation*}
\Mot(\mathbb{F})\otimes \mathbb{Q}_p \to \Isoc(\mathbb{F})
\end{equation*}
corresponds to the morphism of gerbes induced by the map of characters $W \to \mathbb{Q}$ given by the formula
\begin{equation}
\label{MotivesIsocBand}
\pi \mapsto \frac{\ord_p(\pi)}{n}
\end{equation}
where $\ord_p$ is the extension of the $p$-adic valuation on $\mathbb{Q}$ corresponding to a chosen embedding $\mathbb{Q}^{\text{al}}\to \mathbb{Q}_p^{\text{al}}$. The latter can be deduced from the section ``The isocrystal of a motive'' in \cite{Milne1994}. Finally, one obtains the following description of $\Mot(\mathbb{F})$.

\begin{theorem}[Langlands-Rapoport, Milne (Th.3.29 in \cite{Milne1994})]
Assume the Tate conjecture. Then the category of motives over $\mathbb{F}$ is a semi-simple Tannakian category over $\mathbb{Q}$. Its band is given by the protorus $P$ corresponding to the $\Gal(\mathbb{Q}^{\text{al}}/\mathbb{Q})$-module $W$. The equivalence class of $\Mot(\mathbb{F})$ in $H^2(\mathbb{Q}, P)$ is identified with the class $(0_{\mathbb{A}^{\{p,\infty\}}}, \alpha_p, \alpha_{\mathbb{R}})\in H^2(\mathbb{A}, P)$ via the inclusion $H^2(\mathbb{Q}, P)\hookrightarrow H^2(\mathbb{A}, P)$.
\end{theorem}

The naturality of such a description is treated in Theorem 6.3 in \cite{MilneGerbes}, which combined with Proposition \ref{KtWeil} leads us to the following statement.
\begin{cor}
\label{TateScholze}
Let $\omega^\ast(\Rep(\Kt_\mathbb{Q}))$ be the subcategory of $\Rep(\Kt_\mathbb{Q})$, defined by the image of the global Kottwitz class $\alpha_{\mathbb{Q}}$ along the map $\omega: W \hookrightarrow X^\ast(\mathbb{D}_{\mathbb{Q}})$. There exists an equivalence $\omega^\ast(\Rep(\Kt_\mathbb{Q}))\to\Mot(\mathbb{F})$ compatible with the natural realization functors.
\end{cor}

\renewcommand{\refname}{References}
%\addcontentsline{toc}{section}{References}

%%%%%%%%%%%%%%%%%%%%%%%%%%%%%%%%%%%%%%%%%%%%%%%%%%%%%%%%%%%%%%%%%%%%

\end{document}